\documentclass[a4paper]{amsart}
\usepackage{hyperref}
\usepackage{txfonts, amsmath,amstext,amsthm,amscd,amsopn,verbatim,amssymb, amsfonts}
\usepackage{fullpage}

\usepackage[bbgreekl]{mathbbol}
\usepackage{makecell}
\usepackage{enumerate}
\usepackage{tikz}
\usetikzlibrary{matrix}
\usetikzlibrary{shapes}
\usetikzlibrary{arrows}
\usetikzlibrary{calc,3d}
\usetikzlibrary{decorations,decorations.pathmorphing,decorations.pathreplacing,decorations.markings,snakes}
\usetikzlibrary{through}

\tikzset{int/.style={circle,draw,fill,inner sep=1.4pt,outer sep=0pt},dot/.style={circle,fill,inner sep=0pt,outer sep=0pt}}

\numberwithin{equation}{section}

\theoremstyle{plain}
  \newtheorem{thm}{Theorem}[section]
  \newtheorem{prop}[thm]{Proposition}
  
  \newtheorem{cor}[thm]{Corollary}
  \newtheorem{conjecture}{Conjecture}
  \newtheorem{lemma}[thm]{Lemma}
\theoremstyle{definition}
  \newtheorem{defi}[thm]{Definition}

 \newtheorem*{thm*}{Theorem}

\newcommand{\K}{{\mathbb{K}}}
\newcommand{\Z}{{\mathbb{Z}}}
\newcommand{\N}{{\mathbb{N}}}

\newcommand{\mA}{\mathrm{A}}
\newcommand{\mB}{\mathrm{B}}
\newcommand{\mC}{\mathrm{C}}
\newcommand{\mV}{\mathrm{V}}
\newcommand{\mE}{\mathrm{E}}
\newcommand{\mF}{\mathrm{F}}
\newcommand{\mH}{\mathrm{H}}
\newcommand{\mP}{\mathrm{P}}

\newcommand{\Id}{\mathrm{Id}}

\newcommand{\CC}{\mathrm{C}}
\newcommand{\DD}{\mathrm{D}}

\newcommand{\fGC}{\mathrm{fGC}}
\newcommand{\GC}{\mathrm{GC}}
\newcommand{\HGC}{\mathrm{HGC}}
\newcommand{\fHGC}{\mathrm{fHGC}}
\newcommand{\fHGCc}{\mathrm{fHGCc}}
\newcommand{\fHGCd}{\mathrm{fHGCd}}
\newcommand{\HC}{\mathrm{UR}}
\newcommand{\HL}{\mathrm{HL}}
\newcommand{\fGCc}{\mathrm{fGCc}}

\newcommand{\OC}{\fGC_1^{\geq 1}[-3]}
\newcommand{\OCC}{\fGC_1^{\dagger}[-3]}

\newcommand{\gra}{\mathit{gra}}

\newcommand{\grt}{\mathfrak {grt}}

\DeclareMathOperator{\sgn}{sgn}

\newcommand{\mxto}[1]{{\overset{#1}{\longmapsto}}}

\begin{document}
\title{Differentials on graph complexes III - deleting a vertex}

\author{Marko \v Zivkovi\' c}
\address{Institute of Mathematics\\ University of Zurich\\ Winterthurerstrasse 190 \\ 8057 Zurich, Switzerland}
\email{the.zivac@gmail.com}


\thanks{The author has been partially supported by the Swiss National Science Foundation, grant 200021 150012 and SwissMAP NCCR.}

\begin{abstract}
We prove that the hairy graph complex with the extra differential is almost acyclic. This justifies the construction of hairy graph cohomology classes by the waterfall mechanism. The main idea used in the paper is a new differential, deleting a vertex. We hope that the new differential can have further applications in the study of graph cohomology.
\end{abstract}

\maketitle

\section{Introduction}

Generally speaking, graph complexes are graded vector spaces of formal linear combinations of isomorphism classes of some kind of graphs, with the standard differential defined by vertex splitting (or, dually, edge contraction). The various graph cohomology theories are arguably some of the most fascinating objects in homological algebra. Each of graph complexes play a certain role in a subfield of homological algebra or algebraic topology. They have an elementary and simple combinatorial definition, yet we know almost nothing about what their cohomology actually is.

In this paper we consider two kinds of graph complexes. The most basic graph complexes are introduced by Maxim Kontsevich. These complexes come in versions $\GC_n$, where $n$ ranges over integers (see \ref{ss:GC} for the definition). Physically, $\GC_n$ is formed by vacuum Feynman diagrams of a topological field theory in dimension $n$. Alternatively, $\GC_n$ governs the deformation theory of $E_n$ operads in algebraic topology \cite{grt}. Some examples of graphs are:
$$
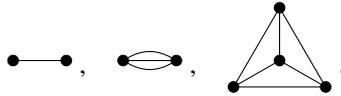
\begin{tikzpicture}[baseline=-1ex,scale=.7]
\node[int] (a) at (0,0) {};
\node[int] (b) at (1,0) {};
\draw (a) edge (b);
\end{tikzpicture}
\;,\quad
\begin{tikzpicture}[baseline=-1ex,scale=.7]
\node[int] (a) at (0,0) {};
\node[int] (b) at (1,0) {};
\draw (a) edge (b);
\draw (a) edge[bend left=30] (b);
\draw (a) edge[bend right=30] (b);
\end{tikzpicture}
\;,\quad
\begin{tikzpicture}[baseline=-1ex,scale=.7]
\node[int] (a) at (0,0) {};
\node[int] (b) at (90:1) {};
\node[int] (c) at (210:1) {};
\node[int] (d) at (330:1) {};
\draw (a) edge (b);
\draw (a) edge (c);
\draw (a) edge (d);
\draw (b) edge (c);
\draw (b) edge (d);
\draw (c) edge (d);
\end{tikzpicture}
\;.
$$

The other kind of graph complexes we consider are \emph{hairy graph complexes}. These are complexes spanned by graphs with external legs (``hairs''). These complexes come in versions $\HGC_{m,n}$ where $m$ and $n$ range over integers (see \ref{ss:HGC} for the definition). They compute the rational homotopy of the spaces of embeddings of disks modulo immersions, fixed at the boundary $ \overline{\mathrm{Emb}_\partial}(\mathbb{D}^m,\mathbb{D}^n)$, provided that $n-m\geq 3$, see \cite{FTW}. Some examples of hairy graphs are:
$$
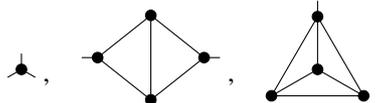
\begin{tikzpicture}[baseline=-1ex,scale=.7]
\node[int] (a) at (0,0) {};
\draw (a) edge (90:.3);
\draw (a) edge (120+90:.3);
\draw (a) edge (240+90:.3);
\end{tikzpicture}
\;,\quad
\begin{tikzpicture}[baseline=-2ex,scale=.7]
\node[int] (v1) at (-1,0){};
\node[int] (v2) at (0,.8){};
\node[int] (v3) at (1,0){};
\node[int] (v4) at (0,-.8){};
\draw (v1)  edge (v2);
\draw (v1)  edge (v4);
\draw (v2)  edge (v3);
\draw (v3)  edge (v4);
\draw (v4)  edge (v2);
\draw (v1)  edge (-1.3,0);
\draw (v3)  edge (1.3,0);
\end{tikzpicture}
\;,\quad
\begin{tikzpicture}[baseline=-1ex,scale=.7]
\node[int] (a) at (0,0) {};
\node[int] (b) at (90:1) {};
\node[int] (c) at (210:1) {};
\node[int] (d) at (330:1) {};
\draw (b) edge (0,1.3);
\draw (a) edge (b);
\draw (a) edge (c);
\draw (a) edge (d);
\draw (b) edge (c);
\draw (b) edge (d);
\draw (c) edge (d);
\end{tikzpicture}
\;.
$$

Both kind of complexes split into the product of subcomplexes with fixed loop order and hairy graph complex splits into the product of subcomplexes with fixed number of hairs, cf.\ \eqref{eq:split1} and \eqref{eq:split3}. Furthermore, the complexes $\GC_n$ and $\GC_{n'}$, respectively $\HGC_{m,n}$ and $\HGC_{m',n'}$, are isomorphic up to some unimportant degree shifts if $m\equiv m' \text{ mod } 2$ and 
$n\equiv n' \text{ mod } 2$. Hence it suffices to understand 2 possible cases of $\GC_n$ and 4 possible cases of $\HGC_{m,n}$ according to parity of $m$ and $n$.

The long standing open problem we are attacking in this paper is the following.

\smallskip

{\bf Open Problem:} Compute the cohomologies $H\left(\GC_n\right)$ and $H\left(\HGC_{m,n}\right)$.

\smallskip

Very little is known about those homologies, and a very few tools are available to compute them. The most non-trivial result is that $H^0(\GC_2)$ is isomorphic to Grothendiech-Teichm\"uller Lie algebra $\grt_1$, shown by Willwacher in \cite{grt}. There are also some results connecting hairy and non-hairy complexes, see \cite{vogel} and \cite{TW}.

This paper is the third in the series of papers \cite{DGC1}, \cite{DGC2}, dealing with extra differentials on graph complexes. The basic idea is to deform the standard differential $\delta$ to $\delta+\delta^{extra}$ making the complex (almost) acyclic. The extra peace does not fix loop order or the number of hairs, and the spectral sequence can be found such that the standard differential is the first differential. Therefore, on the first page of the spectral sequence we see the standard cohomology we are interested in, and because the whole differential is acyclic, classes of it cancel with each other on further pages, cf.\ Table \ref{tbl:evencanceling2}. Therefore, classes come in pairs.

In the first paper of the series \cite{DGC1} the deformed differential was introduced for non-hairy graphs $\GC_n$ for both parity of $n$. In the second paper of the series \cite{DGC2} the deformed differential was introduced for hairy graphs $\HGC_{m,n}$ for even $m$ and both parities of $n$, while for odd $m$ the existence of the suitable extra differential $\Delta$ was just conjectured. We call the spectral sequence arising from this differential the \emph{first}.

\begin{conjecture}[{\cite[Conjecture 1]{DGC2}}]
\label{conj:main}
$H\left(\HGC_{-1,n},\delta+\Delta\right)=0$ for all $n$. 
\end{conjecture}

Furthermore, \cite{grt}, \cite{TW} and \cite {TW2} introduce another deformed differentials on $\HGC_{n,n}$ and $\HGC_{n-1,n}$ making them quasi-isomorphic to $\GC_n$:
\begin{equation*}
 H\left(\HGC_{n,n},D'\right)\cong H\left(\GC_n\right), \quad\quad H\left(\HGC_{n-1,n},D'\right)\cong H\left(\GC_n\right).
\end{equation*}
One checks that all classes that come from $H\left(\GC_n\right)$ live in one-hair part of $H\left(\HGC\right)$. Therefore, the spectral sequence argument and the canceling mechanism can be used for this extra differential for other classes of $H\left(\HGC\right)$, so they come in pairs. We call the spectral sequence arising from $D'$ the \emph{second}. This pairs are different from those that come from the first spectral sequence.

One can start from one-hair classes that come from $H\left(\GC\right)$, find its pair using the first spectral sequence, then find its pair using the second spectral sequence, and so on. This mechanism is called ``waterfall mechanism''. It is depicted in \cite[Figure 1]{DGC2} for even $m$, and still pending in \cite[Figure 2]{DGC2} for odd $m$.

The main purpose of this paper is to prove Conjecture \ref{conj:main} (\cite[Conjecture 1]{DGC2}), making the tentative waterfall mechanism real also in this case.
\begin{thm}
\label{thm:main}
\begin{itemize}\item[]
\item The Complex $\left(\HGC_{-1,1},\delta+\Delta\right)$ is acyclic.
\item The Cohomology $H\left(\HGC_{-1,0},\delta+\Delta\right)$ is one-dimensional, the class being represented by a graph with one vertex and three hairs on it.
\end{itemize}
\end{thm}

We prove the theorem by using a new operation called \emph{deleting a vertex}. In non-hairy complex $\GC_n$ we denote it by $D$ and it deletes a vertex and reconnects its edges to other vertices in all possible ways, summed over all vertices. For $n=-1$ $D$ is of degree 1, and under some weak assumptions it is again a differential.

We hope that deleting a vertex $D$ can have further uses. However, in this paper we only get one small extra result in Corollary \ref{cor:even2}, strengthening the result from the first paper in this series about $H\left(\fGC_n\right)$ for even $n$ (cf. \cite[Corollary 4]{DGC1}). Computer calculations imply that the complex $\left(\GC_{-1},D\right)$ is far from acyclic, but there is a hope that $\left(\GC_{-1},\delta+D\right)$ is acyclic. If so, it will lead to a kind of waterfall mechanism in $\GC_n$ for odd $n$.

In the hairy complex we will introduce a sort of deleting a vertex only in certain cases with low number of hairs, needed in this paper. They are probably auxiliary maps with less general importance than $D$.

As shown in Theorem \ref{thm:main} $\left(\HGC_{-1,n},\delta+\Delta\right)$ is acyclic for $n$ odd, while for $n$ even it has one class of cohomology. By the definition, those complexes are spanned by connected graphs. Disconnected version stays acyclic for $n$ odd, but for $n$ even the class in connected version produces a number of non-trivial classes in disconnected version. They are studied in the second claim of Proposition \ref{prop:bHGCc}. It introduces a list of quasi-isomorphisms from some part of non-hairy graph complexes to the disconnected version of the hairy graph complex.

\subsection*{Structure of the paper}
In Section \ref{sec:background} we define notions, recall results needed in this paper and systematize the notation.
Section \ref{s:D} introduces the deleting a vertex $D$ in non-hairy graphs and provides some results about it.
In Section \ref{s:contraints} we introduce some new subcomplexes of hairy graph complex needed later in the paper, while Section \ref{s:Dh} introduces few maps that delete vertices in hairy complex and shows some needed results.
Finally, in sections \ref{s:Hevenodd} and \ref{s:Hoddeven} we prove the first and the second part of Theorem \ref{thm:main}. Appendix \ref{app:groupinv} clarifies the use of complexes with distinguishable vertices. All straightforward (\ref{app:standard}) or technical (\ref{app:technical}) results are moved to the last two sections of the appendix.

\subsection*{Acknowledgements}
I am very grateful to my PhD adviser Thomas Willwacher for reading the draft of this paper and many suggestions. I also thank Anton Khoroshkin for a fruitful discussion.
\bibliographystyle{plain}

\section{Background and definitions}\label{sec:background}
In this section we recall basic notation and several results shown in the literature that will be used in the paper.

\subsection{General notation}

We work over a field $\K$ of characteristic zero. All vector spaces and differential graded vector spaces are assumed to be $\K$-vector spaces. 

Graph complexes as vector spaces are generally defined by the graphs that span them. When we say \emph{a graph} in a graph complex, we only mean the base graph, while any linear combination (or a series) of graphs will be called \emph{an element} of the graph complex.

Let $\left([C],d\right)$ be a graph complex spanned by a set of graphs $C$ and $[D]\subset[C]$ be a subspace spanned by a set of graphs $D\subset C$. If $[D]$ is closed under the differential, $([D],d)$ is a subcomplex.
Let $D^c=S\setminus D$ be the complement of $D$ in $C$. If $[D^c]$ is closed under the differential, $([D^c],d)$ is a subcomplex, and the quotient $([C]/[D^c],d)$ is well defined. In that case, we will identify $[D]$ with $[C]/[D^c]$ and by abuse of notation talk about the complex $([D],d)$, meaning the complex $\left([C]/[D^c],d\right)$. It is the complex similar to $([C],d)$ where all graphs in $D^c$ are identified with $0$.

\subsection{Kontsevich's graph complexes}
\label{ss:GC}
In this subsection we quickly recall the construction of the hairless graph complexes. For more details see \cite{grt}, or for more elementary definition see \cite{eulerchar}.

Consider the set of directed graphs $\gra_{v,e}^{\geq i}$ with $v>0$ distinguishable vertices and $e\geq 0$ distinguishable directed edges (numbered from $1$ to $v$, respectively $e$), all vertices being at least $i$-valent ($i\in\{0,1,2,3\}$), without tadpoles (edges that start and end at the same vertex).

For $n\in\Z$ we define a degree of an element of $\gra_{v,e}^{\geq i}$ to be $d=(v-1)n+(1-n)e$. We may say that the degree of a vertex is $n$ and the degree of an edge is $1-n$. Let
\begin{equation}
V_{v,e}^{\geq i}:=\left\langle\gra_{v,e}^{\geq i}\right\rangle[-(v-1)n-(1-n)e]
\end{equation}
be the vector space of formal series of $\gra_{v,e}^{\geq i}$ with coefficients in $\K$. It is a graded vector space with non-zero term only in degree $d=(v-1)n+(1-n)e$.

Let $S_k$ be the $k$-th symmetric group. There is a natural right action of the group $S_v\times \left(S_e\ltimes S_2^{\times e}\right)$ on $\gra_{v,e}^{\geq i}$, where $S_v$ permutes vertices, $S_e$ permutes edges and $S_2^{\times e}$ changes the direction of edges.
Let $\sgn_v$, $\sgn_e$ and $\sgn_2$ be one-dimensional representations of $S_v$, respectively $S_e$, respectively $S_2$, where the odd permutation reverses the sign. They can be considered as representations of the whole product $S_v\times \left(S_e\ltimes S_2^{\times e}\right)$.

The \emph{full graph complex} is
\begin{equation}
\label{def:fGCi}
 \fGC_n^{\geq i}:=\left\{
\begin{array}{ll}
\prod_{v,e}\left(V_{v,e}^{\geq i}\otimes\sgn_e\right)_{S_v\times \left(S_e\ltimes S_2^{\times e}\right)}
\simeq\prod_{v,e}\left(V_{v,e}^{\geq i}\otimes\sgn_e\right)^{S_v\times \left(S_e\ltimes S_2^{\times e}\right)}
\qquad&\text{for $n$ even,}\\
\prod_{v,e}\left(V_{v,e}^{\geq i}\otimes\sgn_v\otimes\sgn_2^{\otimes e}\right)_{S_v\times \left(S_e\ltimes S_2^{\times e}\right)}
\simeq\prod_{v,e}\left(V_{v,e}^{\geq i}\otimes\sgn_v\otimes\sgn_2^{\otimes e}\right)^{S_v\times \left(S_e\ltimes S_2^{\times e}\right)}
\qquad&\text{for $n$ odd.}
\end{array}
\right.
\end{equation}
Here the group in the subscript means taking the space of coinvariants, and the group in superscript means taking the space of invariants. Because the group is finite, the two spaces are isomorphic to each other, as stated.

The differential $\delta$ on $\fGC_n$ is defined as follows:
\begin{equation}
 \delta(\Gamma):=\sum_{x\in V(\Gamma)}\frac{1}{2}s_x(\Gamma) - a_x(\Gamma),
\end{equation}
where $V(\Gamma)$ is the set of vertices of $\Gamma$, $s_x$ stands for ``splitting of $x$'' and means inserting
\begin{tikzpicture}[scale=.5]
 \node[int] (a) at (0,0) {};
 \node[int] (b) at (1,0) {};
 \draw (a) edge (b);
\end{tikzpicture}
instead of the vertex $x$ and summing over all possible ways of connecting the edges that have been connected to $x$ to the new two vertices, and $a_x$ stands for ``adding an edge at $x$'' and  means adding
\begin{tikzpicture}[scale=.5]
 \node[int] (a) at (0,0) {};
 \node[int] (b) at (1,0) {};
 \draw (a) edge (b);
 \node[above left] at (a) {$\scriptstyle x$};
\end{tikzpicture}
on the vertex $x$. Unless $x$ is an isolated vertex, $a_x$ will cancel two terms of the splitting $s_x$. To precisely define the sign of the resulting graph, we set that, before acting of $S_v\times \left(S_e\ltimes S_2^{e}\right)$, a new vertex and an edge get the next free number and the edge is directed towards the new vertex. If a vertex next to an $N$-fold edge ($N$ edges between same two vertices) is split such that one new vertex gets $k$ edges, and another $N-k$, there is a factor $\binom{N}{k}$, coming from distinguishing edges. One can check that the minimal valence condition $\geq i$ for $i\in\{0,1,2,3\}$ is preserved under the differential, so the complex is well defined. It is also clear that the differential is indeed of degree $1$.

Since the differential does not change the number $e-v$, the full graph complex $\fGC_n^{\geq i}$ splits into the direct product of subcomplexes:
\begin{equation}
\label{eq:splitb}
\left(\fGC_n^{\geq i},\delta\right)=\left(\prod_{b\in\Z}\mB^b\fGC_n^{\geq i},\delta\right),
\end{equation}
where $\mB^b\fGC_n^{\geq i}$ is the part of $\fGC_n^{\geq i}$ spanned by graphs in which $e-v=b$.

For the chosen parity of $n$, the actual number $n$ matters only for the degree shift of the subcomplexes $\mB^b\fGC_n^{\geq i}$ in the product. If we know the cohomology for one $n$, cohomology for another $n$ of the same parity is straightforward. Therefore, in this paper we will only deal with the complexes $\fGC_0^{\geq i}$ and $\fGC_1^{\geq i}$, as representatives for the even and the odd complexes. Note that in $\fGC_0^{\geq i}$ the degree is $d=e$ and in $\fGC_1^{\geq i}$ the degree is $d=v-1$.

Let
\begin{equation}
\label{def:fGCci}
\fGCc_n^{\geq i}\subset\fGC_n^{\geq i}
\end{equation}
be the subcomplex spanned by the connected graphs.

If $i=0$ we will usually omit the superscript $\geq 0$, e.g.\ we consider
\begin{equation}
\label{def:fGC}
\fGC_n:=\fGC_n^{\geq 0}.
\end{equation}
We introduce a shorter notation
\begin{equation}
\label{def:GC}
\GC_n:=\fGCc_n^{\geq 3}
\end{equation}
because that complex is particularly important.

Cohomology of these subcomplexes is related to the cohomology of the full complexes, as shows the following result.
\begin{prop}[{\cite[Proposition 3.4]{grt}}]
\label{prop:Simpl}
\begin{equation*}
 H\left(\fGCc_0\right)=H\left(\fGCc_0^{\geq 2}\right)=H\left(\GC_0\right)\oplus\bigoplus_{j\geq 1}\K[-4j-1],
\end{equation*}
\begin{equation*}
 H\left(\fGCc_1\right)=H\left(\fGCc_1^{\geq 2}\right)=H\left(\GC_1\right)\oplus\bigoplus_{j\geq 0}\K[-4j-2].
\end{equation*}
\end{prop}

\subsection{The spectral sequence of \cite{DGC1}}
\label{ss:DGC1}

In the first paper of this series we introduced deformed differentials on the graph complexes above. In this paper we only need the even case, $n=0$. There is an extra differential $\nabla$ on $\fGC_0$ that acts by adding one edge in all possible ways. Every edge is added twice, once from one vertex to the other and once in the other way round. It holds that $\delta\nabla+\nabla\delta=0$, so the differential on $\fGC_0$ can be deformed to $\delta+\nabla$. The result we need is the following.
\begin{prop}[{\cite[Corollary 4]{DGC1}}]
There is a spectral sequence converging to
\[
 H\left(\fGCc_0, \delta+\nabla\right)=0
\]
whose $E^1$ term is 
\[
 H\left(\fGCc_0,\delta\right).
\]
\end{prop}

The result implies that the homological classes of $H\left(\fGCc_0,\delta\right)$ have to be canceled on later pages of the spectral sequence. Therefore they come in pairs.

\subsection{Hairy graph complex}
\label{ss:HGC}
The hairy graph complexes $\HGC_{m,n}$ in general are defined and studied in the second paper of this series, \cite{DGC2}. There are essentially four types of complexes depending on the parity of $n$ and $m$. In this paper we are interested in two types, those with parity of edges and hairs being the opposite, i.e.\ $m$ is odd. We will deal with the complexes $\HGC_{-1,0}$ and $\HGC_{-1,1}$ as representatives of those types.

For $h\in\N_0$ let $\mH^h\fHGC_{-1,n}^{\geq i}$ be the complex spanned by graphs similar to those of $\fGC_n^{\geq i}$, but with $h$ hairs attached to vertices, strictly defined as follows.
Consider the set of directed graphs $\gra_{v,e,h}^{\geq i}$ with $v>0$ distinguishable vertices, $e\geq 0$ distinguishable directed edges and $h\geq 0$ distinguishable hairs attached to some vertices, all vertices being at least $i$-valent ($i\in\{0,1,2,3\}$), without tadpoles (edges that start and end at the same vertex). In hairy graphs, the valence includes also the attached hairs.

For $n,m\in\Z$ let a degree of an element of $\gra_{v,e,h}^{\geq i}$ be $d=-m+vn+(1-n)e+(m+1-n)h$. In this paper we consider only the case $m=-1$, so $d=1+vn+(1-n)e-nh$. Let
\begin{equation}
V_{v,e,h}^{\geq i}:=\langle\gra_{v,e,h}^{\geq i}\rangle[-1-vn-(1-n)e+nh]
\end{equation}
be the vector space of formal series of $\gra_{v,e,h}^{\geq i}$ with coefficients in $\K$. It is a graded vector space with non-zero term only in degree $d=1+vn+(1-n)e-nh$.

There is a natural right action of the group $S_v\times S_h\times \left(S_e\ltimes S_2^{\times e}\right)$ on $\gra_{v,e,h}^{\geq i}$, where $S_v$ permutes vertices, $S_h$ permutes hairs, $S_e$ permutes edges and $S_2^{\times e}$ changes the direction of edges.
Let $\sgn_v$, $\sgn_h$, $\sgn_e$ and $\sgn_2$ be one-dimensional representations of $S_v$, respectively $S_h$, respectively $S_e$, respectively $S_2$, where the odd permutation reverses the sign. They can be considered as representations of the whole product $S_v\times S_h\times \left(S_e\ltimes S_2^{\times e}\right)$.

The \emph{full hairy graph complex} is
\begin{equation}
\label{def:HfHGCi}
 \mH^h\fHGC^{\geq i}_{-1,n}:=\left\{
\begin{array}{ll}
\prod_{v,e}\left(V_{v,e,h}^{\geq i}\otimes\sgn_e\right)_{S_v\times S_h\times \left(S_e\ltimes S_2^{\times e}\right)}
\qquad&\text{for $n$ even,}\\
\prod_{v,e}\left(V_{v,e,h}^{\geq i}\otimes\sgn_v\otimes\sgn_h\otimes\sgn_2^{\otimes e}\right)_{S_v\times S_h\times \left(S_e\ltimes S_2^{\times e}\right)}
\qquad&\text{for $n$ odd.}
\end{array}
\right.
\end{equation}
Because the group is finite, the space of invariants may replace the space of coinvariants.

The differential is similar to the one of $\fGC_n^{\geq i}$:
\begin{equation}
 \delta(\Gamma)=\sum_{x\in V(\Gamma)}\frac{1}{2}s_x(\Gamma) - a_x(\Gamma) - h(x) e_x(\Gamma),
\end{equation}
where in ``splitting of $x$'' $s_x$ hairs are also attached in all possible ways to the new two vertices. There is a factor $\binom{N}{k}$ for splitting a vertex with $N$ hairs into vertices with $k$ and $N-k$ hairs, like in splitting of a multiple edge. ``Adding an edge at $x$'' $a_x$ is the same as before and $e_x$ stands for ``extracting a hair at $x$'' and means adding
\begin{tikzpicture}[scale=.5]
 \node[int] (a) at (0,0) {};
 \node[int] (b) at (1,0) {};
 \draw (a) edge (b);
 \node[above left] at (a) {$\scriptstyle x$};
 \draw (b) edge (1.3,0.3);
\end{tikzpicture}
on the vertex $x$ instead of one hair, while $h(x)$ is the number of hairs on the vertex $x$. Unless $x$ is an isolated vertex with a hair, $h(x)e_x$ will cancel two terms of the splitting $s_x$. It is easily seen that $\delta$ is indeed a differential, i.e.\ $\delta^2=0$. In particular
\begin{equation}
\label{eq:HGC-GCshift}
\mH^0\fHGC_{-1,n}^{\geq i}=\fGC_n^{\geq i}[-1-n].
\end{equation}

We define the full hairy graph complex
\begin{equation}
\label{def:fHGCi}
 \fHGC_{-1,n}^{\geq i}:=\prod_{h\geq 0}\mH^h\fHGC_{-1,n}^{\geq i}.
\end{equation}
We often need a subcomplex without hairless part:
\begin{equation}
\label{def:H>fHGCi}
 {\mH^{\geq 1}\fHGC_{-1,n}^{\geq i}}:=\prod_{h\geq 1}\mH^h\fHGC_{-1,n}^{\geq i}.
\end{equation}

The differential does not change the number $e-v$ again, so the full hairy graph complex $\fHGC_{-1,n}^{\geq i}$ splits into the direct product of (in each degree) finite dimensional subcomplexes:
\begin{equation}
\label{eq:splitbhairy}
\left(\fHGC_{-1,n}^{\geq i},\delta\right)=\left(\prod_{b\in\Z}\mB^b\fHGC_{-1,n}^{\geq i},\delta\right),
\end{equation}
where $\mB^b\fHGC_{-1,n}^{\geq i}$ is the part of $\fHGC_n^{\geq i}$ spanned by graphs where $e-v=b$. The same is true for $\mH^{\geq 1}\fHGC_{-1,n}^{\geq i}$.

Similarly to the non-hairy complex, for a chosen parity of $n$, the actual number $n$ matters only for the degree shift of the subcomplexes $\mB^b\fHGC_{-1,n}^{\geq i}$ in the product. If we know the cohomology for one $n$, cohomology for another $n$ of the same parity is straightforward. Therefore, in this paper we will only deal with complexes $\fHGC_{-1,0}^{\geq i}$ and $\fHGC_{-1,1}^{\geq i}$, and their subcomplexes, as representatives for even and odd complexes. Note that in $\fHGC_{-1,0}^{\geq i}$ the degree is $d=e+1$ and in $\fHGC_{-1,1}^{\geq i}$ the degree is $d=v+1-h$.

Similarly to the non-hairy complex, let 
\begin{equation}
\label{def:fHGCci}
\fHGCc_{-1,n}^{\geq i}\subset\fHGC_{-1,n}^{\geq i}
\end{equation}
be the subcomplex spanned by connected graphs.

In the hairy case, we skip the superscript if $i=1$, e.g.\ we consider
\begin{equation}
\label{def:fHGC}
\fHGC_{-1,n}=\fHGC_{-1,n}^{\geq 1},
\end{equation}
because those complexes will be used the most, unlike in the non-hairy case where skipping superscript means including all valences.
We also introduce a shorter notation
\begin{equation}
\label{def:HGC}
\HGC_{-1,n}:=\mH^{\geq 1}\fHGCc_{-1,n}^{\geq 3}.
\end{equation}
This complex is the closest to the one defined in \cite{DGC2}. Strictly speaking, $\HGC_{-1,n}$ for odd $n$ defined in \cite{DGC2} has another graph that we do not allow here, the graph with no vertices and two hairs. It comes naturally if a hair is understood as an edge towards 1-valent vertex of another type, being the case in that paper.


\subsection{Extra differential in hairy graph complexes}

In \cite{DGC2} an extra differential $\Delta$ has been defined on $\HGC_{-1,n}$ that anti-commutes with $\delta$, so that $\delta+\Delta$ is also a differential.
Here, on ${\fHGC_{-1,n}^{\geq i}}$ we define the following additional operation $\Delta:\mH^h\fHGC_{-1,n}^{\geq i}\rightarrow\mH^{h-1}\fHGC_{-1,n}^{\geq i}$:
\begin{equation}
 \Delta\Gamma\,=\sum_{x\in V(\Gamma)}\Delta_x=\sum_{x\in V(\Gamma)}h(x)\tilde\Delta_x,
\end{equation}
where $\tilde\Delta_x$ deletes a hair on vertex $x$ and connects $x$ to other vertices in all possible ways with a new edge, and $\Delta_x=h(x)\tilde\Delta_x$ for $h(x)$ being the number of hairs on $x$. To precisely define the sign of the resulting graph, we set that, before the acting of $S_v\times S_h\times \left(S_e\ltimes S_2^{\times e}\right)$, a new edge get the next free number. Restricted to $\HGC_{-1,n}$, $\Delta$ is the same differential as defined in \cite{DGC2}. It is again straightforward that $\Delta$ squares to zero and anti-commutes with $\delta$, so $\delta+\Delta$ is also a differential.

\subsection{Some simple graphs}

Here we define some simple graphs. They live in the hairy graph complexes, and if they have no hairs, also in the standard graph complexes.

Let
\begin{equation}
\sigma_a:=
\begin{tikzpicture}[baseline=-.5ex]
\node[int] (a) at (0,0) {};
\draw (a) edge (90:.2);
\draw (a) edge (360/7+90:.2);
\node[dot] at (2*360/7+90:.15) {};
\node[dot] at (2.5*360/7+90:.15) {};
\node[dot] at (3*360/7+90:.15) {};
\draw (a) edge (4*360/7+90:.2);
\draw (a) edge (5*360/7+90:.2);
\draw (a) edge (6*360/7+90:.2);
\end{tikzpicture}
\end{equation}
be a graph with one vertex and $a\geq 0$ hairs, and we call it a \emph{star}.
Let
\begin{equation}
\lambda_a:=
\begin{tikzpicture}
 \node[int] (a) at (0,0) {};
 \node[int] (b) at (.5,0) {};
 \draw (a) edge (b);
 \draw (a) edge (360/7:.2);
 \draw (a) edge (2*360/7:.2);
 \draw (a) edge (3*360/7:.2);
 \draw (a) edge (4*360/7:.2);
 \node[dot] at (5*360/7:.15) {};
 \node[dot] at (5.5*360/7:.15) {};
 \node[dot] at (6*360/7:.15) {};
\end{tikzpicture}
\end{equation}
with $a-1$ hairs on one vertex for $a\geq 1$.
Hairless graphs that exist in $\fGC_n$ are denoted by simpler notation $\sigma:=\sigma_0$ and $\lambda:=\lambda_1$.
Because of symmetry reasons in $\fHGC_{-1,1}$ there can not be more than one hair on the same vertex, so in that complex only $\sigma_1$ and $\lambda_2$ exist, together with hairless $\sigma_0$ and $\lambda_1$.

For two graphs $\Gamma$ and $\Gamma'$, the graph $\Gamma\cup\Gamma'$ is the disconnected graph that consists of $\Gamma$ and $\Gamma'$. For the matter of sign, all vertices, edges and hairs of the first graph come before those of the second one. For a graph $\Gamma$ and $n\geq 1$, $\Gamma^{\cup n}$ is the graph that consist of $n$ copies of the graphs $\Gamma$ put together.

The following graph in $\fHGC_{-1,n}$ will be used often:
\begin{equation}
\alpha:=\sum_{n\geq 1}\frac{1}{n!}\; \sigma_1^{\cup n}=\;
\begin{tikzpicture}[baseline=-.5ex]
\node[int] (a) at (0,0) {};
\draw (a) edge (0,.2);
\end{tikzpicture}\;
+\frac{1}{2}\;
\begin{tikzpicture}[baseline=-.5ex]
\node[int] (a) at (0,0) {};
\draw (a) edge (0,.2);
\end{tikzpicture}\;
\begin{tikzpicture}[baseline=-.5ex]
\node[int] (a) at (0,0) {};
\draw (a) edge (0,.2);
\end{tikzpicture}\;
+\frac{1}{3!}\;
\begin{tikzpicture}[baseline=-.5ex]
\node[int] (a) at (0,0) {};
\draw (a) edge (0,.2);
\end{tikzpicture}\;
\begin{tikzpicture}[baseline=-.5ex]
\node[int] (a) at (0,0) {};
\draw (a) edge (0,.2);
\end{tikzpicture}\;
\begin{tikzpicture}[baseline=-.5ex]
\node[int] (a) at (0,0) {};
\draw (a) edge (0,.2);
\end{tikzpicture}\;
+\dots
\end{equation}

\subsection{List of graph complexes}

For the reader's convenience, we provide Table \ref{tbl:gc} of all graph complexes used and defined in the paper, with a short explanation and the reference to the definition.

The most of notations of graph complexes, except some technical complexes, are of the form $\mathrm{pNe_i^s}$, meaning as follows.
\begin{itemize}
 \item The first small letter $\mathrm{p}$ is either $\mathrm{f}$, meaning full complex, either omitted. General complexes are full, but some its subcomplexes are particularly important, and they are labeled without $\mathrm{f}$.
 \item Capital letters $\mathrm{N}$ is the main type of complex. It can be:
 \begin{itemize}
  \item $\mathrm{GC}$ - graph complex,
  \item $\mathrm{HGC}$ - hairy graph complex. 
 \end{itemize}
 \item The ending $\mathrm{e}$ deals with connectivity and if stated means:
 \begin{itemize}
  \item $\mathrm{c}$ - connected graphs,
  \item $\mathrm{d}$ - disconnected graphs.
 \end{itemize}
 \item The subscript $\mathrm{i}$ defines degrees and parity (1 number for ordinary complexes and 2 numbers for hairy complexes).
 \item The superscript $\mathrm{s}$ deals mostly with valences and if stated means:
 \begin{itemize}
  \item $\geq i$ - all vertices are at least $i$-valent,
  \item $\dagger$ - all vertices at least 1-valent and no connected component $\lambda_1$,
  \item ${\ddagger}$ - all vertices at least 1-valent and no connected component $\sigma_1$ or $\lambda_2$,
  \item ${\dagger\ddagger}$ - all vertices at least 1-valent and no connected component $\lambda_1$, $\sigma_1$ or $\lambda_2$,
  \item $*$ - all vertices at least 2-valent and  hairy vertices at least 3-valent, see Definition \ref{defi:bounded}.
 \end{itemize} 
\end{itemize}

\begin{table}[h]
\begin{tabular}{| c | l | c |}
\hline
Notation & Explanation & Reference  \\
\hline
$\fGC_n^{\geq i}$ & Graph complex with graphs whose vertices are at least $i$-valent & \eqref{def:fGCi}  \\
\hline
$\fGC_n$ & $=\fGC_n^{\geq 0}$ & \eqref{def:fGC}  \\
\hline
$\fHGC_{-1,n}^{\geq i}$ & Hairy graph complex with graphs whose vertices are at least $i$-valent & \eqref{def:fHGCi}  \\
\hline
$\fHGC_{-1,n}$ & $=\fHGC_{-1,n}^{\geq 1}$ & \eqref{def:fHGC}  \\
\hline
$\mathrm{pNc_i^s}$ & Subcomplex of $\mathrm{pN_i^s}$ spanned by connected graphs & \eqref{def:fGCci}, \eqref{def:fHGCci} \\
\hline
$\GC_n$ & $=\fGCc_n^{\geq 3}$ & \eqref{def:GC}  \\
\hline
$\mH^{\geq 1}\mathrm{pHGCe_i^s}$ & $\subset\mathrm{pHGCe_i^s}$ without hairless graphs & \eqref{def:HfHGCi}  \\
\hline
$\HGC_{-1,n}$ & $=\mH^{\geq 1}\fGCc_n^{\geq 3}$ & \eqref{def:HGC}  \\
\hline
${\mathrm{fHGC^\dagger}_{-1,n}}$ & $\subset\mathrm{fHGC}_{-1,n}$ with no connected component $\lambda_1$ & Definition \ref{defi:bounded} \\
\hline
${\mathrm{fHGC^\ddagger}_{-1,n}}$ & $\subset\mathrm{fHGC}_{-1,n}$ with no connected component $\sigma_1$ or $\lambda_2$ & Definition \ref{defi:bounded} \\
\hline
${\mathrm{fHGC^{\dagger\ddagger}}_{-1,n}}$ & ${\mathrm{fHGC^\dagger}_{-1,n}}\cap{\mathrm{fHGC^\ddagger}_{-1,n}}$ & Definition \ref{defi:bounded} \\
\hline
$\fHGC_{-1,n}^{/\ddagger}$ & $\fHGC_{-1,n}/\fHGC_{-1,n}^{\ddagger}$ & Definition \ref{defi:unbounded} \\
\hline
$\fHGC_{-1,n}^{\dagger/{\dagger\ddagger}}$ & $\fHGC_{-1,n}^{\dagger}/\fHGC_{-1,n}^{{\dagger\ddagger}}$ & Definition \ref{defi:unbounded} \\
\hline
${\HC_{-1,n}}$ & $\subset\fHGC_{-1,n}^{\dagger}$ with all connected components in $\sigma_1$ or $\lambda_2$ & Definition \ref{defi:unbounded} \\
\hline
$\HL_1$ & $=\left(\Delta+D^{(1)}\right)\left(\mH^1\fHGC_{-1,1}^\dagger\right)$ & \eqref{def:HL1} \\
\hline
$\HL_0$ & $=\Delta\left(\mH^1\fHGC_{-1,0}\right)$ & \eqref{def:HL0} \\
\hline
$\mH^\flat{\mathrm{pHGC}_{-1,n}^{\mathrm{s}}}$ & $=\mH^{\geq 1}\mathrm{pHGC}_{-1,n}^{\mathrm{s}}\oplus\HL_n$ & \makecell{\eqref{def:fHGC1'}, \eqref{def:bHGC1'}, \eqref{def:fHGC13'}, \\ \eqref{def:fHGC0'}, \eqref{def:bHGC0'}} \\
\hline
$\fHGCd_{-1,0}$ & $\subset \fHGC_{-1,0}$ spanned by disconnected graphs & \eqref{def:bHGCd} \\
\hline
\end{tabular}
\vspace{5pt}
\caption{\label{tbl:gc}
The table of graph complexes used in the paper. Universal letters $\mathrm{p}$, $\mathrm{N}$, $\mathrm{e}$, $\mathrm{i}$ and $\mathrm{s}$ mean any letter, and also without the letter.}
\end{table}

\subsection{Splitting of complexes and spectral sequences}

We often need subcomplex of the particular complex spanned by graphs that have some number fixed, e.g.\ number of vertices. Let $v$ be the number of vertices in a graph, $e$ the number of edges, $h$ the number of hairs, $c$ the number of connected components and $p$ the number of hairy connected components (connected components with at least one hair). Prefixes are listed in Table \ref{tbl:prefix}.

\begin{table}[h]
\begin{tabular}{| c | l |}
\hline
Prefix & Explanation  \\
\hline
$\mV^k$ & Graphs with $v=k$ \\
\hline
$\mE^k$ & Graphs with $e=k$ \\
\hline
$\mH^k$ & Graphs with $h=k$ \\
\hline
$\mC^k$ & Graphs with $c=k$ \\
\hline
$\mB^k$ & Graphs with $e-v=k$ \\
\hline
$\mA^k$ & Graphs with $e+h=k$ \\
\hline
$\mF^k$ & Graphs with $e+h-v=k$ \\
\hline
$\mP^k$ & Graphs with $p=k$ \\
\hline
\end{tabular}
\vspace{5pt}
\caption{\label{tbl:prefix}
The table of prefixes that determine certain subcomplexes of a complex. The prefixes mean the subcomplex of the complex that follows, spanned by graphs with the fixed number stated.}
\end{table}

We have already used prefixes $\mB$ and $\mH$ in \eqref{eq:splitb}, \eqref{def:HfHGCi} and \eqref{eq:splitbhairy}.

Let $\Gamma$ be an element of a graph complex $\CC$. We also use prefixes from above in front of the graph $\Gamma$, indicating the part of $\Gamma$ with fixed number. E.g.\ $\mH^5\Gamma$ is the part of $\Gamma$ with $5$ hairs.

Some differentials do not change some numbers from the above, so the complex splits as a direct product of complexes with fixed that number, e.g.
\begin{equation}
\label{eq:split1}
\left(\fGC_n^{\geq i},\delta\right)=\prod_{b\in\Z}\left(\mB^b\fGC_n^{\geq i},\delta\right),
\end{equation}
\begin{equation}
\left(\fGC_n^{\geq i},\nabla\right)=\prod_{v\in\N}\left(\mV^v\fGC_n^{\geq i},\nabla\right),
\end{equation}
\begin{equation}
\label{eq:split3}
\left(\fHGC_{-1,n}^{\geq i},\delta\right)=\prod_{b\in\Z}\left(\mB^b\fHGC_{-1,n}^{\geq i},\delta\right)=
\prod_{h\in\N}\left(\mH^h\fHGC_{-1,n}^{\geq i},\delta\right)=
\prod_{b\in\Z}\prod_{h\in\N}\left(\mB^b\mH^h\fHGC_{-1,n}^{\geq i},\delta\right),
\end{equation}
\begin{equation}
\label{eq:split4}
\left(\fHGC_{-1,n}^{\geq i},\Delta\right)=\prod_{v\in\N}\left(\mV^v\fHGC_{-1,n}^{\geq i},\Delta\right)=
\prod_{a\in\Z}\left(\mA^a\fHGC_{-1,n}^{\geq i},\Delta\right)=
\prod_{v\in\N}\prod_{a\in\Z}\left(\mA^a\mV^v\fHGC_{-1,n}^{\geq i},\Delta\right),
\end{equation}
\begin{equation}
\label{eq:split5}
\left(\fHGC_{-1,n}^{\geq i},\delta+\Delta\right)=\prod_{f\in\Z}\left(\mF^f\fHGC_{-1,n}^{\geq i},\delta+\Delta\right).
\end{equation}

We also use superscripts in the prefix of the form of inequality, e.g.\ $\mV^{\geq k}$, that obviously means the subcomplex spanned by graphs which fulfill the inequality, e.g.
\begin{equation}
\mV^{\geq v}\CC=\prod_{k\geq v}\mV^k\CC.
\end{equation}

We have already used this notation in \eqref{def:H>fHGCi}.
Those subcomplexes with inequality often form a filtration of the complex. To this filtration a spectral sequence is associated. We say that the spectral sequence is \emph{on the number} given by the prefix.
E.g.\ the spectral sequence of $\left(\fGC_n^{\geq i},\delta+\Delta\right)$ on the number of vertices is the one associated to the filtration $\left(\mV^{\geq v}\fGC_n^{\geq i},\delta+\Delta\right)$. It is $\geq v$, not $\leq v$, because the differential $\delta+\Delta$ can increase, but not decrease the number of vertices. Its first differential is clearly $\Delta$.

We say that a spectral sequence converges correctly if it converges to the cohomology of the whole complex. To ensure correct convergence standard spectral sequence arguments, i.e.\ those from \cite[Appendix C]{DGC1}, are used. For doing so, it is often useful for a complex to be finite dimensional in each degree. It is mostly not the case, but after splitting a complex as in \eqref{eq:split1}--\eqref{eq:split4}, each subcomplex often has that property. So, the spectral sequence arguments can be used for each of them, and so we can compute the cohomology of the whole complex.

The following superscript will also be used.
\begin{equation}
\mB^{<f,par}\CC:=\prod_{i>0}\mB^{f-2i}\CC,
\end{equation}
i.e.\ it is the product of subcomplexes with $e-v<f$ of the same parity as $f$.

As explained in the appendix \ref{app:groupinv} we use the notation $\bar\mV^v\CC$ for the space similar to $\mV^v\CC$ but with distinguishable vertices, i.e.\ the space of coinvariants of other groups (permuting edges, hairs, etc.) before taking coinvariants of the symmetric group $S_v$ that permutes vertices. It holds that
\begin{equation}
\mV^v\CC=\left(\bar\mV^v\CC\right)^{S_v}.
\end{equation}
Let $S_{v-1}$ act on $\bar\mV^v\CC$ as sub-action of $S_v$ that permutes the first $v-1$ vertices, leaving the last vertex. We define
\begin{equation}
\dot\mV^v\CC:=\left(\bar\mV^v\CC\right)^{S_{v-1}}.
\end{equation}
We may need the total spaces
\begin{equation}
\bar\mV\CC:=\prod_{v\geq 1}\bar\mV^v\CC,
\end{equation}
\begin{equation}
\dot\mV\CC:=\prod_{v\geq 1}\dot\mV^v\CC.
\end{equation}

\section{Deleting a vertex in non-hairy graphs}
\label{s:D}

In this section we introduce a new operation $D$ on the non-hairy graph complexes $\fGC_n$ which we call ``deleting a vertex''. Under some weak conditions, it holds that $D^2=0$ and a grading can be settled such that $D$ is a differential. We also obtain one further result about the spectral sequence of \cite{DGC1} for $n=0$. The results will be used for the hairy graph complex in Section \ref{s:Hoddeven}.

\subsection{Even case}
For a graph $\Gamma\in\fGC_0$ we call the set of its vertices $V(\Gamma)$ and define
\begin{equation}
 D(\Gamma) \,:= \sum_{x\in V(\Gamma)}D_x = \sum_{x\in V(\Gamma)}(-1)^{v(x)}\tilde D_x
\end{equation}
where $v(x)$ is the valence of the vertex $x$, $\tilde D_x$ deletes the vertex $x$ and sums over all ways of reconnecting edges that were connected to $x$ to the other vertices, skipping graphs with a tadpole, and $D_x=(-1)^{v(x)}\tilde D_x$. An example is sketched in Figure \ref{fig:exD}.

\begin{figure}[h]
$$
\begin{tikzpicture}[baseline=0ex]
\node[int] (a) at (0,0) {};
\node[int] (b) at (90:1) {};
\node[int] (c) at (210:1) {};
\node[int] (d) at (330:1) {};
\node[above left] at (b) {$\scriptstyle x$};
\draw (a) edge (b);
\draw (a) edge (c);
\draw (a) edge (d);
\draw (b) edge (c);
\draw (b) edge (d);
\draw (c) edge (d);
\end{tikzpicture}
\quad\mxto{D_x}\quad(-1)^3\;
\begin{tikzpicture}[baseline=0ex]
\node[int] (a) at (0,0) {};
\node[int] (c) at (210:1) {};
\node[int] (d) at (330:1) {};
\draw (a) edge[->] (0,1);
\draw (a) edge (c);
\draw (a) edge (d);
\draw (c) edge[->] (-.8,.5);
\draw (d) edge[->] (.8,.5);
\draw (c) edge (d);
\end{tikzpicture}
$$
\caption{\label{fig:exD}
Example of the action $D_x$, deleting the vertex $x$. The graph at the right means the sum over all graphs that can be formed by attaching ends of arrows to vertices, without making a tadpole.}
\end{figure}
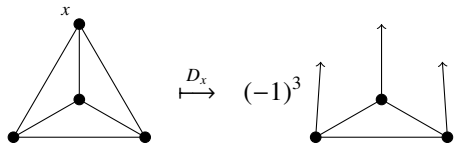

We can restrict $D$ to $\fGC_0^{\geq 1}$ and $\fGC_0^{\geq 2}$. The following propositions about $D$ are stated in the broadest space possible, i.e.\ in $\fGC_0^{\geq i}$ for the smallest $i$ where it holds.

\begin{prop}
\label{DelEvenD2}
In $\fGC_0^{\geq 2}$ it holds that
\begin{equation*}
D^2=0.
\end{equation*}
\end{prop}
\begin{proof}
For a graph $\Gamma$
\begin{equation*}
D^2(\Gamma)\,=\sum_{\substack{x,y\in V(\Gamma)\\x\neq y}}D_yD_x(\Gamma).
\end{equation*}
We now fix $x,y\in V(\Gamma)$, $x\neq y$.
Since $D$ does not change the number of edges, we can distinguish them. Vertices of $\Gamma\in\fGC_0^{\geq 2}$ are at least 2-valent, so there exist an edge between $x$ and another vertex that is not $y$. We choose one of them and call it $f$.

$D_yD_x$ first deletes vertex $x$, reconnects its edges to other vertices, deletes vertex $y$ and reconnects its edges, including those that came from $x$, to other vertices. Let us fix one way of reconnecting edges that are not $f$ and the final destination of $f$. For that way of reconnecting, in $D_yD_x$ there are two terms: the one where $f$ goes directly to its final destination with $D_x$ and the one where $f$ goes to $y$ with $D_x$ and to its final destination with $D_y$. In the later the valence of $y$ while applying $D_y$ is by $1$ bigger than its valence while applying $D_y$ in the former term, making the terms cancel each other because of the sign dependence on valence. It follows that $D_yD_x(\Gamma)=0$ and therefore $D^2(\Gamma)=0$.
\end{proof}

\begin{prop}
\label{DelEvenD4}
In $\fGC_0^{\geq 1}$ it holds that
\begin{equation*}
D^4=0.
\end{equation*}
\end{prop}
\begin{proof}
From the previous proof we see that $D_yD_x(\Gamma)$ may only be non-zero if $x$ is $1$-valent vertex with an edge connecting to $y$.

So $D_zD_yD_x(\Gamma)\neq 0$ implies $x$ is of that kind. But also $D_zD_y(D_x(\Gamma))\neq 0$ implies $y$ is $1$-valent and connected to $z$ in $D_x(\Gamma)$. Since $D_x$ did not change the valence of $y$, already in $\Gamma$ it was $1$-valent, and vertices $x$ and $y$ formed $\lambda=
\begin{tikzpicture}[scale=.5]
 \node[int] (a) at (0,0) {};
 \node[int] (b) at (1,0) {};
 \draw (a) edge (b);
\end{tikzpicture}$.

Now let us pick a term in $D^4(\Gamma)$, say $D_wD_zD_yD_x(\Gamma)$. It being non-zero implies $D_zD_yD_x(\Gamma)\neq 0$, so $x$ and $y$ form $\lambda$.
Similarly, $D_wD_zD_y(D_x(\Gamma))\neq 0$ imply $y$ and $z$ form $\lambda$ in $D_x(\Gamma)$. Since the edge at $z$ came from $x$, $z$ was isolated in $\Gamma$ what is not possible in $\fGC_0^{\geq 1}$. Therefore $D_wD_zD_yD_x(\Gamma)$ is always $0$ and $D^4=0$.
\end{proof}

Recall the extra differential $\nabla:\fGC_0\rightarrow\fGC_0$ that acts by adding one edge in all possible ways defined in \cite{DGC1} and cited in \ref{ss:DGC1}.

\begin{prop}
\label{DelEvenDd}
In $\fGC_0$ it holds that
\begin{equation*}
\delta D - D \delta = \nabla.
\end{equation*}
\end{prop}
\begin{proof}
\begin{multline*}
\left(\delta D (\Gamma) - D \delta (\Gamma)\right) = \delta\left(\sum_{x}D_x(\Gamma)\right) - D\left(\sum_y\frac{1}{2}s_y(\Gamma) - a_y(\Gamma)\right) = 
\sum_{x}\sum_{\substack{y\\y\neq x}}\left(\frac{1}{2}s_yD_x(\Gamma)-a_yD_x(\Gamma)\right) \\
- \frac{1}{2}\sum_y\left(\sum_{\substack{x\\x\neq y}}D_xs_y(\Gamma)+D_ys_y(\Gamma)+D_zs_y(\Gamma)\right)
+ \sum_y\left(\sum_{\substack{x\\x\neq y}}D_xa_y(\Gamma)+D_ya_y(\Gamma)+D_za_y(\Gamma)\right) = \\
=\frac{1}{2}\sum_{\substack{x,y\\x\neq y}}s_yD_x(\Gamma)
-\sum_{\substack{x,y\\x\neq y}}a_yD_x(\Gamma)
-\frac{1}{2}\sum_{\substack{x,y\\x\neq y}}D_xs_y(\Gamma)
-\sum_xD_zs_x(\Gamma)
+\sum_{\substack{x,y\\x\neq y}}D_xa_y(\Gamma)
+\sum_xD_xa_x(\Gamma)
+\sum_xD_za_x(\Gamma),
\end{multline*}
where $x$ and $y$ run through $V(\Gamma)$ and $z$ is the name of the newly added vertex.

For different $x,y\in V(\Gamma)$ it easily follows that
\begin{equation*}
s_yD_x(\Gamma)=D_xs_y(\Gamma).
\end{equation*}

The term $D_zs_x(\Gamma)$ first splits a new vertex $z$ from the vertex $x$, connects them with an edge, say $g$, and reconnects some of the edges from $x$ to $z$. Afterwards it deletes vertex $z$ and reconnects edges from it to other vertices, possibly back to $x$, and reconnects $g$ also to other vertex.
The final result is that some of the edges are reconnected from $x$ to other vertices, and there is a new edge $g$ connecting $x$ and some other vertex.
The edge $g$ can be seen as added at the end, and before that, since number of edges is not changed, we can distinguish edges.
Suppose there is an edge that at the end stays at $x$, and call it $f$. We fix one way of reconnecting all other edges.
For that way of reconnecting there are two terms in $D_zs_x(\Gamma)$: one where $f$ stays at $x$ and one where it is reconnected to $z$ and back to $x$, and they cancel each other because valence of $z$ differ by one in them.
Therefore everything what survives from $D_zs_x(\Gamma)$ is reconnecting all edges from $x$ to other vertices and adding an edge $g$ from $x$ to some other vertex, say $y$. That is exactly the same as deleting vertex $x$ and adding an edge at $y$, with the opposite sign because the valence of the vertex being deleted differs by one. So we get
\begin{equation*}
D_zs_x(\Gamma)=-\sum_{\substack{y\\y\neq x}}a_yD_x(\Gamma).
\end{equation*}

An easy argument, that is left to the reader, implies
\begin{equation*}
D_xa_x(\Gamma)=-\sum_{\substack{y\\y\neq x}}D_xa_y(\Gamma).
\end{equation*}

The last term $\sum_xD_za_x(\Gamma)$ clearly adds an edge from $x$ to some other vertex, so it holds that
\begin{equation*}
\sum_xD_za_x(\Gamma)=\nabla(\Gamma).
\end{equation*}
By equations obtained, all except the last term of the above expression cancel, and the claimed formula follows.
\end{proof}

Propositions \ref{DelEvenD2} and \ref{DelEvenDd} easily imply $D\nabla+\nabla D=0$ in $\fGC_0^{\geq 2}$. But we need a bit stronger result.

\begin{prop}
\label{DelEvenDnab}
In $\fGC_0^{\geq 1}$ it holds that
\begin{equation*}
D\nabla+\nabla D=0.
\end{equation*}
\end{prop}
\begin{proof}
$D\nabla$ puts an edge in all possible ways and then deletes a vertex, say $x$. If the new edge has been connected to $x$, it is moved to another vertex. Let finally the new edge connect vertices $y$ and $z$. To that position it can come in three different ways:
\begin{itemize}
\item directly been connected to $y$ and $z$ by $\nabla$, what is exactly the corresponding term from $\nabla D$;
\item  been connected from $y$ to $x$ by $\nabla$ and then moved to $z$ by $D_x$, what is the negative of the term in $\nabla D$ because of the sign changes in deleting vertex $x$ with one more valence;
\item and the same from $z$ to $x$, what is also the negative of the term in $\nabla D$.
\end{itemize}
All terms in $\nabla D$ are come like this, so indeed $D\nabla=\nabla D-2\nabla D=-\nabla D$.
\end{proof}

\begin{prop}
\label{DelEvenD2nab}
In $\fGC_0^{\geq 1}$ it holds that
\begin{equation*}
\nabla D^2=0.
\end{equation*}
\end{prop}
\begin{proof}
From the proof of Proposition \ref{DelEvenD2} the term $D_yD_x(\Gamma)$ of $D^2(\Gamma)$ may only be non-zero if $x$ is a $1$-valent vertex with an edge connecting to $y$. Then $D_yD_x$ deletes $x$, $y$ and the edge between them, reconnects all other edges from $y$ elsewhere and adds an edge in all possible ways. Then $\nabla$ adds another edge in all possible ways. Adding one edge on one place and another edge on another place cancels with adding edges in the opposite order, hence the result.
\end{proof}

%
%

\subsection{Odd case}
Like in the even case, for a graph $\Gamma\in\fGC_1$ we define
\begin{equation}
 D(\Gamma) \,:= \sum_{x\in V(\Gamma)}D_x = \sum_{x\in V(\Gamma)}(-1)^{v(x)}\tilde D_x.
\end{equation}
Note that multiple edges are now possible. Similarly as with splitting, if an $N$-fold edge went to the vertex being deleted, there is a factor $\binom{N}{k_1,k_2,\dots}$ where $k_i$ is the number of edges from that multiple edge that go to the vertex $i$.
We can restrict $D$ to $\fGC_1^{\geq 1}$.

\begin{prop}
\label{DelOddD2}
On $\fGC_1^{\geq 1}$ it holds that
\begin{equation*}
D^2=0.
\end{equation*}
\end{prop}
\begin{proof}
Like in the proof of Proposition \ref{DelEvenD2} we write
\begin{equation*}
D^2(\Gamma)\,=\sum_{\substack{x,y\in V(\Gamma)\\x\neq y}}D_yD_x(\Gamma),
\end{equation*}
fix $x,y\in V(\Gamma)$ and distinguish edges. If there is an edge from $x$ to a vertex other than $y$ the same reasoning from Proposition \ref{DelEvenD2} leads to $D_yD_x(\Gamma)=0$. If not, all edges from $x$ go to $y$ and let there be $k>0$ of them. Let $\Gamma'$ be the graph obtained from $\Gamma$ by deleting vertex $x$ and all $k$ edges at $x$. Then $D_yD_x(\Gamma)$ is actually $D_y(\Gamma')$ where we add $k$ edges in all possible ways. But adding an edge from one vertex to another cancels with adding an opposite edge, leading to the conclusion $D_yD_x(\Gamma)=0$. Therefore it again holds that $D^2=0$.
\end{proof}

\begin{prop}
\label{DelOddDd}
On $\fGC_1$ it holds that
\begin{equation*}
\delta D + D \delta = 0.
\end{equation*}
\end{prop}
\begin{proof}
The argument, with a bit of care for the signs, is the same as in the proof of Proposition \ref{DelEvenDd}. Only in the last term adding an edge from one vertex to another cancels with adding the opposite one, leading to the result $0$.
\end{proof}

\subsection{More about the spectral sequence of \cite{DGC1} for the even case}

Recall from \cite[Corollary 4]{DGC1} that $H\left(\fGCc_0, \delta+\nabla\right)=0$ and that there is a spectral sequence converging to it whose $E^1$ term (i.e.\ the first page) is $H\left(\fGCc_0,\delta\right)$ with the differential $\nabla$. The spectral sequence is on the number $b=e-v$. The following corollary calculates homologies of some similar complexes.

\begin{cor}
\begin{enumerate}\item[]
\item $H\left(\fGCc_0^{\geq 2}, \delta+\nabla\right)=0$;
\item $H\left(\fGC_0^{\geq 2}, \delta+\nabla\right)=0$.
\end{enumerate}
\end{cor}
\begin{proof}
\begin{enumerate}\item[]
\item On the mapping cone of the inclusion $\left(\fGCc_0^{\geq 2}, \delta+\nabla\right)\hookrightarrow\left(\fGCc_0, \delta+\nabla\right)$ we set up a spectral sequence mentioned above, on the number $b=e-v$. The complex with the first differential is the mapping cone of the inclusion $\left(\fGCc_0^{\geq 2},\delta\right)\hookrightarrow\left(\fGCc_0,\delta\right)$. It is acyclic by Proposition \ref{prop:Simpl}. The spectral sequence converges correctly because the space in each degree $e$ is finitely dimensional, so the whole mapping cone is acyclic. That leads to $H\left(\fGCc_0^{\geq 2}, \delta+\nabla\right)=H\left(\fGCc_0, \delta+\nabla\right)=0$.


\item On $H\left(\fGC_0^{\geq 2}, \delta+\nabla\right)$ we set up a spectral sequence on the number of connected components. It clearly converges to the cohomology of the whole complex. In the $c$-th row there is a complex
$$
\left(\left(\fGCc_0^{\geq 2}, \delta+\nabla\right)^{\otimes c}\right)^{S_c}[1-c]
$$
that is acyclic by (1), hence the result.
\end{enumerate}
\end{proof}

On $\fGC_0^{\geq 1}$ we define a conjugated differential:
\begin{equation}
\label{TildeDeltac}
\tilde\delta:=e^{ D}(\delta+\nabla)e^{- D}.
\end{equation}
Since $D$ is nilpotent, it is well defined. Using results of Propositions \ref{DelEvenD4} to \ref{DelEvenD2nab} we can calculate $\tilde\delta$ explicitly:
\begin{multline}
\label{TildeDeltac2}
\tilde\delta=\left(\Id+ D+\frac{ D^2}{2}+\frac{ D^3}{6}\right)
(\delta+\nabla)
\left(\Id- D+\frac{ D^2}{2}-\frac{ D^3}{6}\right)=\\
=\delta+\nabla+ D\delta+ D\nabla-\delta  D-\nabla  D -  D\delta  D
+\frac{ D^2\delta}{2}-\frac{ D^2\delta D}{2}
+\frac{\delta D^2}{2}+\frac{ D\delta D^2}{2}=
\delta +  D\nabla.
\end{multline}

Note that the differential $\tilde\delta$ can not change $b=e-v$ by an odd amount. Therefore complexes with that differential split into the direct sum of two complexes, one with even and one with odd $b$.
The following result is now straightforward.

\begin{cor}
\label{cor:even2}
There is a spectral sequence converging to
$$
H\left(\fGC_0^{\geq 2}, \tilde\delta\right)=0
$$
whose first page is
$$
H\left(\fGC_0^{\geq 2},\delta\right)
$$
Furthermore, in all spectral sequences differentials on odd pages are $0$.
\end{cor}

The corollary is similar to \cite[Corollary 4]{DGC1}, but this time the complex includes disconnected graphs. Indeed, \cite[Corollary 4]{DGC1} implies the same result for disconnected graphs, and in this sense our result is weaker. There is the filtration on $b=e-v$ in both cases, but different total differentials $\delta+\nabla$ and $\tilde\delta$, that are both acyclic. So, in both cases there are cancellations of classes in $H\left(\fGC_0^{\geq 2},\delta\right)$ as drawn in Table \ref{tbl:evencanceling2}. For disconnected graphs our result is stronger because we know that there are no cancellations on odd pages.

\begin{table}[h]
\begin{tikzpicture}
\matrix (mag) [matrix of nodes,ampersand replacement=\&]
{
{}\& 0 \& 1 \& 2 \& 3 \& 4 \& 5 \& 6 \& 7 \& 8 \& 9 \&10 \&11 \&12 \&13 \&14 \&15 \&16 \&17 \&18 \&19 \&20 \&21 \&22 \&23 \&24 \&25 \&26 \&27 \\
-1\& 0 \& 0 \& 0 \& 0 \& 0 \& 0 \& 0 \& 0 \& 0 \& 0 \& 0 \& 0 \& 0 \& 0 \& 0 \& 0 \& 0 \& 0 \& 0 \& 0 \& 0 \& 0 \& 0 \& 0 \& 0 \& 0 \& 0 \& 0 \\
0 \&   \& 0 \& 0 \& 0 \& 0 \& 1 \& 0 \& 0 \& 0 \& 1 \& 0 \& 0 \& 0 \& 1 \& 1 \& 0 \& 0 \& 1 \& 1 \& 0 \& 0 \& 1 \& 2 \& 0 \& 0 \& 1 \& 2 \& 1 \\
1 \&   \&   \& 0 \& 0 \& 0 \& 0 \& 0 \& 0 \& 0 \& 0 \& 0 \& 0 \& 0 \& 0 \& 0 \& 0 \& 0 \& 0 \& 0 \& 0 \& 0 \& 0 \& 0 \& 0 \& 0 \& 0 \& 0 \& 0 \\
2 \&   \&   \&   \& 0 \& 0 \& 0 \& 1 \& 0 \& 0 \& 0 \& 0 \& 1 \& 0 \& 0 \& 0 \& 1 \& 0 \& 0 \& 0 \& 1 \& 1 \& 0 \& 0 \& 0 \& 1 \& 1 \& 0 \& 0 \\
3 \&   \&   \&   \&   \& 0 \& 0 \& 0 \& 0 \& 0 \& 0 \& 0 \& 0 \& 0 \& 0 \& 0 \& 0 \& 0 \& 0 \& 0 \& 0 \& 0 \& 0 \& 0 \& 0 \& 0 \& 0 \& 0 \& 0 \\
4 \&   \&   \&   \&   \&   \& 0 \& 0 \& 0 \& 0 \& 0 \& 1 \& 0 \& 1 \& 0 \& 0 \& 1 \& 0 \& 1 \& 0 \& 1 \& 0 \& 1 \& 0 \& 1 \& 1 \& 1 \& 1 \& 0 \\
5 \&   \&   \&   \&   \&   \&   \& 0 \& 0 \& 0 \& 0 \& 0 \& 0 \& 0 \& 0 \& 0 \& 1 \& 0 \& 0 \& 0 \& 0 \& 1 \& 0 \& 0 \& 0 \& 1 \& 0 \& 0 \& 0 \\
6 \&   \&   \&   \&   \&   \&   \&   \& 0 \& 0 \& 0 \& 0 \& 0 \& 0 \& 0 \& 1 \& 0 \& 1 \& 0 \& 1 \& 1 \& 0 \& 1 \& 0 \& 2 \& 0 \& 1 \& 0 \& 2 \\
7 \&   \&   \&   \&   \&   \&   \&   \&   \& 0 \& 0 \& 0 \& 0 \& 0 \& 0 \& 0 \& 0 \& 1 \& 0 \& 0 \& 1 \& 0 \& 2 \& 0 \& 0 \& 1 \& 1 \& 1 \& 0 \\
8 \&   \&   \&   \&   \&   \&   \&   \&   \&   \& 0 \& 0 \& 0 \& 0 \& 0 \& 0 \& 0 \& 0 \& 0 \& 1 \& 0 \& 2 \& 1 \& 1 \& 1 \& 1 \& 2 \& 1 \& 2 \\
9 \&   \&   \&   \&   \&   \&   \&   \&   \&   \&   \& 0 \& 0 \& 0 \& 0 \& 0 \& 0 \& 0 \& 0 \& 0 \& 0 \& 1 \& 0 \& 1 \& 2 \& 0 \& 3 \& 0 \& 3 \\
10\&   \&   \&   \&   \&   \&   \&   \&   \&   \&   \&   \& 0 \& 0 \& 0 \& 0 \& 0 \& 0 \& 0 \& 0 \& 0 \& 0 \& 0 \& 2 \& 0 \& ? \& ? \& ? \& ? \\
11\&   \&   \&   \&   \&   \&   \&   \&   \&   \&   \&   \&   \& 0 \& 0 \& 0 \& 0 \& 0 \& 0 \& 0 \& 0 \& 0 \& 0 \& 0 \& 0 \& 2 \& ? \& ? \& ? \\
12\&   \&   \&   \&   \&   \&   \&   \&   \&   \&   \&   \&   \&   \& 0 \& 0 \& 0 \& 0 \& 0 \& 0 \& 0 \& 0 \& 0 \& 0 \& 0 \& 0 \& 0 \& 3 \& ? \\
};
\draw (mag-2-1.north west) -- (mag-1-29.south east);
\draw (mag-1-2.north west) -- (mag-15-1.south east);
\draw[-latex, thick] (mag-3-7) edge (mag-5-8);
\draw[-latex, thick] (mag-3-11) edge (mag-7-12);
\draw[-latex, thick] (mag-5-13) edge (mag-7-14);
\draw[-latex, thick] (mag-3-15) edge (mag-9-16);
\draw[-latex, thick] (mag-8-17) edge (mag-10-18);
\draw[-latex, thick] (mag-8-22) edge (mag-10-23);
\draw[-latex, thick] (mag-10-21) edge (mag-12-22);
\draw[-latex, thick] (mag-11-23) edge (mag-13-24);
\draw[-latex, thick] (mag-12-25) edge (mag-14-26);
\draw[-latex, thick] (mag-10-23) edge (mag-12-24);
\end{tikzpicture}
\caption{\label{tbl:evencanceling2} Table of dimensions of cohomology $H\left(\fGC_0,\delta\right)$. The column number represents the number of edges $e$ and the row number represents $b=e-v$. Known cancellations are depicted by arrows.}
\end{table}

\section{Constraints on hairy graphs}
\label{s:contraints}

In this section we study subcomplexes of Hairy graph complex spanned by graphs that fulfill certain constraints, e.g.\ minimal valence of vertices.

\subsection{Simplifying full graph complex}

Proposition \ref{prop:Simpl} (\cite[Proposition 3.4]{grt}) can easily be extended to the hairy case as follows.

\begin{prop}
\label{prop:fHGC-fHGC3}
\begin{equation*}
H\left(\fHGCc_{-1,n}^{\geq 0},\delta\right)=H\left(\fHGCc^{\geq 2}_{-1,n},\delta\right)=
H\left(\fHGCc_{-1,n}^{\geq 3},\delta\right)\oplus\bigoplus_{\substack{j\geq 3\\j\equiv 2n+1\mod 4}}\K[-1-j].
\end{equation*}
\end{prop}
\begin{proof}
The same argument as in \cite[Proposition 3.4]{grt}.
\end{proof}

In the subcomplex $\fHGCc=\fHGCc^{\geq 1}\subset\fHGCc^{\geq 0}$ only one graph is excluded, the one vertex graph $\sigma=\sigma_0=\bullet$. This makes the cohomology $H\left(\fHGCc_{-1,n},\delta\right)$ different from $H\left(\fHGCc_{-1,n}^{\geq 0},\delta\right)=H\left(\fHGCc^{\geq 2}_{-1,n},\delta\right)$ by one class. The class is represented by $\lambda=\lambda_1=
\begin{tikzpicture}[scale=.5,baseline=-.65ex]
 \node[int] (a) at (0,0) {};
 \node[int] (b) at (1,0) {};
 \draw (a) edge (b);
\end{tikzpicture}$.
If this graph is excluded too, the resulting complex will again be quasi-isomorphic to the original one. In the next definition we define that complex and some other ``excluding'' complexes that does not change the cohomology by the following proposition.

\begin{defi}
\label{defi:bounded}
Let $\fHGC^{\dagger}\subset\fHGC$ be the subcomplex spanned by graphs that do not have $\lambda$ as a connected component.

Let the \emph{bounded graph complex} $\fHGC^{\ddagger}\subset\fHGC$ be the subcomplex spanned by graphs that do not have $\sigma_1=
\begin{tikzpicture}[baseline=-.5ex]
\node[int] (a) at (0,0) {};
\draw (a) edge (0,.2);
\end{tikzpicture}$
or
$
\lambda_2=
\begin{tikzpicture}[baseline=-.5ex]
\node[int] (a) at (0,0) {};
\node[int] (b) at (.5,0) {};
\draw (a) edge (0,.2);
\draw (a) edge (b);
\end{tikzpicture}$
as a connected component. Let $\fHGC^{\dagger\ddagger}=\fHGC^{\dagger}\cap\fHGC^{\ddagger}$.

Let $\fHGC^{*}\subset\fHGC^{\geq 2}$ be the subcomplex spanned by graphs that do not have 2-valent vertex with a hair.
\end{defi}

\begin{prop}
\label{prop:fHGCc-d-dd-fHGCc*}
\begin{equation*}
H\left(\fHGCc^{\geq 0}_{-1,n},\delta\right)=
H\left(\fHGCc^{\dagger}_{-1,n},\delta\right)=
H\left(\fHGCc^{{\dagger\ddagger}}_{-1,n},\delta\right)=
H\left(\fHGCc^{\geq 2}_{-1,n},\delta\right)=
H\left(\fHGCc_{-1,n}^{*},\delta\right),
\end{equation*}
\begin{equation*}
H\left(\fHGCc_{-1,n},\delta\right)=
H\left(\fHGCc^{\ddagger}_{-1,n},\delta\right).
\end{equation*}
\end{prop}
\begin{proof}
Straightforward.
\end{proof}

\begin{cor}
\label{cor:fHGC-d-dd-fHGC*}
\begin{equation*}
H\left(\fHGC^{\geq 0}_{-1,n},\delta\right)=
H\left(\fHGC^{\dagger}_{-1,n},\delta\right)=
H\left(\fHGC^{{\dagger\ddagger}}_{-1,n},\delta\right)=
H\left(\fHGC^{\geq 2}_{-1,n},\delta\right)=
H\left(\fHGC_{-1,n}^{*},\delta\right),
\end{equation*}
\begin{equation*}
H\left(\fHGC_{-1,n},\delta\right)=
H\left(\fHGC^{\ddagger}_{-1,n},\delta\right).
\end{equation*}
\end{cor}
\begin{proof}
The complexes are just the symmetric product of their connected parts, so the corollary directly follows from the proposition. 
\end{proof}

By Proposition \ref{prop:fHGC-fHGC3} all classes missing in ``3-valent complex'' are in hairless part, so Proposition \ref{prop:fHGCc-d-dd-fHGCc*} implies also the following corollary.

\begin{cor}
\label{cor:H>fHGCc3-H>fHGCc}
\begin{multline*}
H\left(\mH^{\geq 1}\fHGCc^{\geq 0}_{-1,n},\delta\right)=
H\left(\mH^{\geq 1}\fHGCc_{-1,n},\delta\right)=
H\left(\mH^{\geq 1}\fHGCc^{\dagger}_{-1,n},\delta\right)=
H\left(\mH^{\geq 1}\fHGCc^{\ddagger}_{-1,n},\delta\right)=\\
=H\left(\mH^{\geq 1}\fHGCc^{{\dagger\ddagger}}_{-1,n},\delta\right)=
H\left(\mH^{\geq 1}\fHGCc^{\geq 2}_{-1,n},\delta\right)=
H\left(\mH^{\geq 1}\fHGCc_{-1,n}^{*},\delta\right)=
H\left(\HGC_{-1,n},\delta\right).
\end{multline*}
\end{cor}

Note that
$\left(\mH^{\geq 1}\fHGCc^{\geq 0}_{-1,n},\delta\right)=
\left(\mH^{\geq 1}\fHGCc_{-1,n},\delta\right)=
\left(\mH^{\geq 1}\fHGCc^{\dagger}_{-1,n},\delta\right)$
and
$\left(\mH^{\geq 1}\fHGCc^{\ddagger}_{-1,n},\delta\right)
=\left(\mH^{\geq 1}\fHGCc^{{\dagger\ddagger}}_{-1,n},\delta\right)$ already as complexes.

\subsection{The bounded graph complex with the extra differential}

The motivation for introducing the bounded graph complex is the following.
In the hairy complexes $\left(\fHGC_{-1,n},\delta+\Delta\right)$ the number of hairs $h$ is unbounded. That makes the spectral sequences on $h$ (that has $\delta$ as the first differential) and the one on the number of connected components $c$ unbounded from above and we can not use standard results for convergence.

The complex splits as in \eqref{eq:split5}:
\begin{equation*}
\left(\fHGC_{-1,n},\delta+\Delta\right)=\prod_{f\in\Z}\left(\mF^f\fHGC_{-1,n},\delta+\Delta\right)
\end{equation*}
where $f=e+h-v$. In every $\mF^f\fHGC_{-1,n}$ for the fixed degree $d=1+vn+(1-n)e-nh$, the number of edges $e=d+nf-1$ is fixed too.
So, increasing the number of hairs $h$ increases the number of vertices $v$ by the same amount. Since $e$ is fixed, that increase will eventually force a graph to have connected components that are stars $\sigma_1=
\begin{tikzpicture}[baseline=-.5ex]
\node[int] (a) at (0,0) {};
\draw (a) edge (0,.2);
\end{tikzpicture}$.
The purpose of this subsection is to show that components $\sigma_1$ are mostly irrelevant for the cohomology, so we can disallow them and bound $h$ from above while preserving the result about cohomology.

\begin{defi}
\label{defi:unbounded}
The \emph{unbounded hairy graph complex} $\fHGC_{-1,n}^{/\ddagger}$ is the quotient complex $\fHGC_{-1,n}/\fHGC_{-1,n}^{\ddagger}$.
Similarly, $\fHGC_{-1,n}^{\dagger/{\dagger\ddagger}}$ is the quotient complex $\fHGC_{-1,n}^{\dagger}/\fHGC_{-1,n}^{{\dagger\ddagger}}$. They are spanned by graphs that have at least one connected component $\sigma_1$ or $\lambda_2$.

The \emph{unbounded remainder} $\HC_{-1,n}\subset\fHGC_{-1,n}^{\dagger}$ is the subcomplex spanned by graphs with all connected components $\sigma_1$ or $\lambda_2$.
\end{defi}


Lemma \ref{lem:HC} from the appendix implies that $H\left(\HC_{-1,n},\delta+\Delta\right)$ is one-dimensional, the class being represented by
$$\alpha=\sum_{n\geq 1}\frac{1}{n!} \sigma_1^{\cup n}.$$

\begin{prop}
\label{prop:uHGC}
$H\left(\fHGC_{-1,n}^{\dagger/{\dagger\ddagger}},\delta+\Delta\right)$ is one-dimensional, the class being represented by $\alpha$.
%
\end{prop}
\begin{proof}

Let $\tilde\Gamma\in\fHGC^{\dagger/{\dagger\ddagger}}_{-1,n}$ be a graph. It may be either $\tilde\Gamma\in\HC_{-1,n}$ or we can write $\tilde\Gamma=\Gamma\cup\gamma$ where $\Gamma\in\fHGC_{-1,n}^{\dagger\ddagger}$ and $\gamma\in\HC_{-1,n}$. In the latter case we call $\Gamma$ the \emph{main part} of $\tilde\Gamma$ and $\gamma$ the \emph{secondary} part of $\tilde\Gamma$. In the former case the whole graph is the secondary part, while the main part is empty.

Let us set up a spectral sequence of $\left(\fHGC_{-1,n}^{\dagger/{\dagger\ddagger}},\delta+\Delta\right)$ on the number of edges in the main part, empty main part having 0 edges. It is easily seen that the differential can not decrease that number. To ensure the correct convergence we split unbounded complex similarly as the full complex:
\begin{equation*}
\left(\fHGC_{-1,n}^{\dagger/{\dagger\ddagger}},\delta+\Delta\right)=\prod_{f\in\Z}\left(\mF^f\fHGC_{-1,n}^{\dagger/{\dagger\ddagger}},\delta+\Delta\right).
\end{equation*}
In every $\mF^f\fHGC_{-1,n}^{\dagger/{\dagger\ddagger}}$ for the fixed degree $d=1+vn+(1-n)e-nh$, the total number of edges $e=d+nf-1$ is fixed too, so the number of edges in the main part is bounded and therefore the spectral sequence converges correctly.

One can check that the first part of the differential is the one that acts within the secondary part only. Therefore the complexes on the first page are the direct product of complexes for the fixed main part, which are all clearly isomorphic to $\HC_{-1,n}$. Therefore, using Lemma \ref{lem:HC}, on the first page of the spectral sequence we have classes represented by $\Gamma\cup\alpha$ for $\Gamma\in\fHGC_{-1,n}^{\dagger\ddagger}$, together with the class $[\alpha]$ itself.

On the second page matters the part of the differential that acts within the main part only, and the one that connects the main part to the secondary part. The element is now uniquely determined by the main part $\Gamma$ so we can investigate what does the differential do to it:
$$
\Gamma\mapsto\delta(\Gamma)+\Delta(\Gamma) + \sum_{x\in V(\Gamma)} h(x) e_x + \sum_{x\in V(\Gamma)} a_x = \Delta(\gamma) + \sum_{x\in V(\Gamma)} \frac{1}{2} s_x.
$$
The first sum corresponds to connecting a hair to one $\sigma_1$ in $\alpha$, and the second sum corresponds to connecting a hair from one $\sigma_1$ to the main part.

Lemma \ref{lem:antennas} shows that this complex is acyclic. So, only the class without the main part, $[\alpha]$ survives. That what was to be demonstrated.
\end{proof}

\begin{cor}
\label{cor:uHGC}
$H\left(\fHGC_{-1,0}^{/\ddagger},\delta+\Delta\right)$ is two-dimensional, the classes being represented by $\alpha$ and $\lambda\cup\alpha$.
\end{cor}
\begin{proof}
A graph in $\fHGC_{-1,0}^{/\ddagger}$ can have at most one connected component $\lambda$ because of the symmetry reasons. So it is $\Gamma\in\fHGC_{-1,0}^{\dagger/\dagger\ddagger}$ or $\Gamma\cup\lambda$ for $\Gamma\in\fHGC_{-1,0}^{\dagger/\dagger\ddagger}$. So
$$
\left(\fHGC_{-1,0}^{/\ddagger},\delta+\Delta\right)=\left(\fHGC_{-1,0}^{\dagger/\dagger\ddagger}\oplus\fHGC_{-1,0}^{\dagger/\dagger\ddagger}\cup\lambda,\delta+\Delta\right).
$$
We set up a spectral sequence on the number of $\lambda$-s (two rows) that clearly converges correctly, and on the first page we have two copies of $\left(\fHGC_{-1,0}^{\dagger/\dagger\ddagger},\delta+\Delta\right)$. By Proposition \ref{prop:uHGC} there is a class of cohomology in both rows, namely $\alpha$ and $\lambda\cup\alpha$. It is easily seen that they go to $0$ by the whole differential, so they represent classes of the whole cohomology.
\end{proof}

\section{Deleting vertices in hairy graphs}
\label{s:Dh}

\subsection{One hair}

Let $D^{(1)}:\mH^1\fHGC_{-1,n}\rightarrow\mH^0\fHGC_{-1,n}$ be defined on a graph $\Gamma$ as
\begin{equation}
\label{def:D1}
D^{(1)}(\Gamma)=(-1)^{v(x)}\tilde{D}_x
\end{equation}
where $x$ is the vertex with the hair, $v(x)$ is its valence and $\tilde{D}_x$ deletes vertex $x$ and the hair, and sums over all ways of reconnecting edges that were connected to $x$ to the other vertices, skipping graphs with a tadpole. An example is sketched in Figure \ref{fig:exD1}.

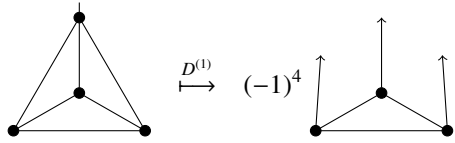
\begin{figure}[h]
$$
\begin{tikzpicture}[baseline=0ex]
\node[int] (a) at (0,0) {};
\node[int] (b) at (90:1) {};
\node[int] (c) at (210:1) {};
\node[int] (d) at (330:1) {};
\draw (b) edge (0,1.2);
\draw (a) edge (b);
\draw (a) edge (c);
\draw (a) edge (d);
\draw (b) edge (c);
\draw (b) edge (d);
\draw (c) edge (d);
\end{tikzpicture}
\quad\mxto{D^{(1)}}\quad(-1)^4\;
\begin{tikzpicture}[baseline=0ex]
\node[int] (a) at (0,0) {};
\node[int] (c) at (210:1) {};
\node[int] (d) at (330:1) {};
\draw (a) edge[->] (0,1);
\draw (a) edge (c);
\draw (a) edge (d);
\draw (c) edge[->] (-.8,.5);
\draw (d) edge[->] (.8,.5);
\draw (c) edge (d);
\end{tikzpicture}
$$
\caption{\label{fig:exD1}
Example of the action $D^{(1)}$, deleting the hairy vertex. The graph at the right means the sum over all graphs that can be formed by attaching ends of arrows to vertices, skipping graphs with a tadpole.}
\end{figure}

Also, the ``pushing the hair'' $\tilde D^{(p)}:\mH^1\fHGC_{-1,n}\rightarrow\mH^1\fHGC_{-1,n}$ be defined on a graph $\Gamma$ as
\begin{equation}
\tilde D^{(p)}(\Gamma)=(-1)^{v(x)}\sum_{y}v(x,y)\tilde{D}^y_x(\Gamma)
\end{equation}
where $x$ is again the vertex with the hair, $y$ runs through all vertices of $\Gamma$, $v(x,y)$ is the number of edges between vertices $x$ and $y$, and whenever $v(x,y)>0$ $\tilde{D}^y_x$ deletes vertex $x$, the hair and one edge between $x$ and $y$, sums over all ways of reconnecting the other edges that were connected to $x$ to the other vertices, skipping graphs with a tadpole, and adds a hair on $y$.

For a graph $\Gamma\in\mH^1\fHGC_{-1,n}^{\geq 2}$ let 
\begin{equation}
D^{(p)}(\Gamma):=\frac{1}{v(x)-1}\tilde D^{(p)}(\Gamma),
\end{equation}
where $x$ is the hairy vertex.
An example is sketched in Figure \ref{fig:exDp}.

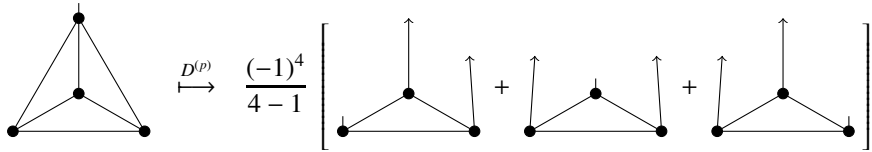
\begin{figure}[h]
$$
\begin{tikzpicture}[baseline=0ex]
\node[int] (a) at (0,0) {};
\node[int] (b) at (90:1) {};
\node[int] (c) at (210:1) {};
\node[int] (d) at (330:1) {};
\draw (b) edge (0,1.2);
\draw (a) edge (b);
\draw (a) edge (c);
\draw (a) edge (d);
\draw (b) edge (c);
\draw (b) edge (d);
\draw (c) edge (d);
\end{tikzpicture}
\quad\mxto{D^{(p)}}\quad\frac{(-1)^4}{4-1}\;\left[\;
\begin{tikzpicture}[baseline=0ex]
\node[int] (a) at (0,0) {};
\node[int] (c) at (210:1) {};
\node[int] (d) at (330:1) {};
\draw (a) edge[->] (0,1);
\draw (a) edge (c);
\draw (a) edge (d);
\draw (c) edge (-.866,-.3);
\draw (d) edge[->] (.8,.5);
\draw (c) edge (d);
\end{tikzpicture}
\;+\;
\begin{tikzpicture}[baseline=0ex]
\node[int] (a) at (0,0) {};
\node[int] (c) at (210:1) {};
\node[int] (d) at (330:1) {};
\draw (a) edge (0,.2);
\draw (a) edge (c);
\draw (a) edge (d);
\draw (c) edge[->] (-.8,.5);
\draw (d) edge[->] (.8,.5);
\draw (c) edge (d);
\end{tikzpicture}
\;+\;
\begin{tikzpicture}[baseline=0ex]
\node[int] (a) at (0,0) {};
\node[int] (c) at (210:1) {};
\node[int] (d) at (330:1) {};
\draw (a) edge[->] (0,1);
\draw (a) edge (c);
\draw (a) edge (d);
\draw (c) edge[->] (-.8,.5);
\draw (d) edge (.866,-.3);
\draw (c) edge (d);
\end{tikzpicture}
\;\right]
$$
\caption{\label{fig:exDp}
Example of the action $D^{(p)}$, pushing the hair. The graphs at the right mean the sum over all graphs that can be formed by attaching ends of arrows to vertices, without making a tadpole.}
\end{figure}

\begin{lemma}
\label{lem:D1'}
On $\mH^1\fHGC_{-1,n}^{\geq 2}$ it holds that
$$
D^{(1)}=\Delta D^{(p)}.
$$
\end{lemma}
\begin{proof}
Clear.
\end{proof}

\begin{lemma}
\label{lem:D1''}
On $\mH^1\fHGC_{-1,0}^{*}$ it holds that 
$$
D^{(1)}D^{(p)}=0.
$$
\end{lemma}
\begin{proof}
Recall that in $\fHGC_{-1,0}$ there are no multiple edges.
Let $\Gamma\in\mH^1\fHGC_{-1,n}^{*}$ be a graph, and $x$ the hairy vertex in $\Gamma$. One term in $D^{(1)}D^{(p)}(\Gamma)$ will delete the hair, the vertex $x$, one of its neighbours $y$ and edge between them, and reconnect all edges that were connected to $x$ and $y$ elsewhere.

Recall that the constraint $*$ means that all vertices are at least 2-valent, and all hairy vertices are at least 3-valent. Therefore $x$ has at least two edges adjacent to it. Chose an edge $e$ at $x$ that did not go towards $y$. That edge can go directly to its final destination, or first to $y$ and then to final destination. Those two terms will cancel, implying the result.
\end{proof}

\begin{prop}
\label{prop:D1}
In $\mH^1\fHGC_{-1,n}$ it holds that
$$
\delta D^{(1)}-D^{(1)}\delta=\Delta.
$$
\end{prop}
\begin{proof}
Let $\Gamma\in\fHGC_{-1,n}^1$ be a graph and let $x$ be the hairy vertex in $\Gamma$. Then
\begin{multline*}
\delta D^{(1)} (\Gamma) - D^{(1)} \delta (\Gamma) = \\
=\sum_{\substack{y\\y\neq x}}\left(\frac{1}{2}s_yD_x(\Gamma)-a_yD_x(\Gamma)\right)
-D_x\sum_{\substack{y\\y\neq x}}\left(\frac{1}{2}s_y(\Gamma)-a_y(\Gamma)\right)
-D_x\left(s'_x(\Gamma)-a_x(\Gamma)\right)+D_ze_x(\Gamma)=\\
=\frac{1}{2}\sum_{\substack{y\\x\neq y}}s_yD_x(\Gamma)
-\sum_{\substack{y\\x\neq y}}a_yD_x(\Gamma)
-\frac{1}{2}\sum_{\substack{y\\x\neq y}}D_xs_y(\Gamma)
+\sum_{\substack{y\\x\neq y}}D_xa_y(\Gamma)
-D_zs'_x(\Gamma)
+D_xa_x(\Gamma)
+D_ze_x(\Gamma),
\end{multline*}
where $y$ runs through $V(\Gamma)$, $z$ is the name of the newly added vertex and $s'_x$ is the part of $s_x$ which moves the hair to the new vertex $z$. It holds that $s_x=2s'_x$.

Using the same arguments as in the proof of Proposition \ref{DelEvenDd} it follows that:
\begin{equation*}
s_yD_x(\Gamma)=D_xs_y(\Gamma),
\end{equation*}
\begin{equation*}
D_zs'_x(\Gamma)=-\sum_{\substack{y\\y\neq x}}a_yD_x(\Gamma),
\end{equation*}
\begin{equation*}
D_xa_x(\Gamma)=-\sum_{\substack{y\\y\neq x}}D_xa_y(\Gamma).
\end{equation*}
It clearly holds that
\begin{equation*}
D_ze_x(\Gamma)=\Delta(\Gamma),
\end{equation*}
so the claimed formula follows.
\end{proof}

\subsection{Two hairs}

In the case with even hairs we define $D^{(2)}:\mH^2\fHGC_{-1,0}\rightarrow\mH^0\fHGC_{-1,0}$. Let $\Gamma\in\mH^2\fHGC_{-1,0}$ be a graph whose both hairs are on the $3$-valent vertex $y$, and let $x$ be another end of the only edge at $y$, like in the picture:
$$
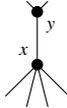
\begin{tikzpicture}[scale=1.4,baseline=1ex]
 \node[int] (a) at (0,0) {};
 \node[int] (b) at (0,.5) {};
 \draw (a) edge (.3,-.3);
 \draw (a) edge (.1,-.4);
 \draw (a) edge (-.3,-0.3);
 \draw (a) edge (-.1,-.4);
 \draw (a) edge (b);
 \draw (b) edge (-.1,.6);
 \draw (b) edge (.1,.6);
 \node[above left] at (a) {$\scriptstyle x$};
 \node[below right] at (b) {$\scriptstyle y$};
\end{tikzpicture}
$$
We call this structure \emph{a flower} on the vertex $y$ with the root $x$. Then
\begin{equation}
D^{(2)}(\Gamma):=(-1)^{v(x)}\tilde{D}_{x,y}
\end{equation}
where $v(x)$ is the valence of $x$, $\tilde{D}_{x,y}$ deletes the flower, i.e.\ vertices $x$ and $y$, the edge between them and both hairs, and sums over all ways of reconnecting other edges that were connected to $x$ to the other vertices, skipping graphs with a tadpole. For the matter of sign we consider the edge in the flower to be the last one. If $\Gamma$ is not of that type, $D^{(2)}(\Gamma)=0$. An example is sketched in Figure \ref{fig:exD2}.

\begin{figure}[h]
$$
\begin{tikzpicture}[scale=1,baseline=3ex]
\node[int] (a) at (0,0) {};
\node[int] (b) at (90:1) {};
\node[int] (c) at (210:1) {};
\node[int] (d) at (330:1) {};
\node[int] (e) at (0,2) {};
\draw (e) edge (-.14,2.14);
\draw (e) edge (.14,2.14);
\draw (b) edge (e);
\draw (a) edge (b);
\draw (a) edge (c);
\draw (a) edge (d);
\draw (b) edge (c);
\draw (b) edge (d);
\draw (c) edge (d);
\end{tikzpicture}
\quad\mxto{D^{(2)}}\quad(-1)^4\;
\begin{tikzpicture}[baseline=0ex]
\node[int] (a) at (0,0) {};
\node[int] (c) at (210:1) {};
\node[int] (d) at (330:1) {};
\draw (a) edge[->] (0,1);
\draw (a) edge (c);
\draw (a) edge (d);
\draw (c) edge[->] (-.8,.5);
\draw (d) edge[->] (.8,.5);
\draw (c) edge (d);
\end{tikzpicture}
$$
\caption{\label{fig:exD2}
Example of the action $D^{(2)}$, deleting the flower. The graph at the right means the sum over all graphs that can be formed by attaching ends of arrows to vertices, without making a tadpole.}
\end{figure}
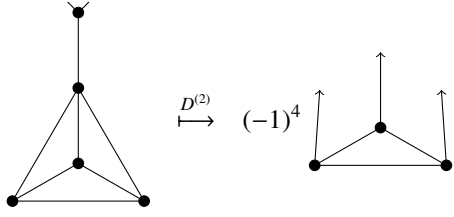

\begin{prop}
\label{prop:D2}
\begin{itemize}
 \item[]
 \item In the connected part $\mH^2\fHGCc_{-1,0}$ it holds that
$$
D^{(1)}\Delta=2\nabla\left(\delta D^{(2)}+ D^{(2)}\delta\right).
$$
 \item In $\mH^2\fHGC_{-1,1}$ it holds that
$$D^{(1)}\Delta=0.$$
\end{itemize}
\end{prop}
\begin{proof}
We do the proof of the first claim for four different cases of $\Gamma\in\mH^2\fHGCc_{-1,0}$:
\begin{enumerate}
\item $\Gamma$ has hairs on two different vertices, $x$ and $y$. By the definition $\delta D^{(2)}(\Gamma)=\delta(0)=0$. Differential $\delta$ can not move hairs to the same vertex, so it is also $D^{(2)}\delta(\Gamma)=0$. The left-hand side is
$$
D^{(1)}\Delta(\Gamma)=D^{(1)}(\Delta_x(\Gamma)+\Delta_y(\Gamma))=D_y\Delta_x(\Gamma)+D_x\Delta_y(\Gamma).
$$
In the first term $\Delta_x$ first deletes the hair on $x$ and connects an edge $f$ from $x$ to all other vertices. Then $D_y$ deletes a vertex $y$ and reconnects all its edges to other vertices in all possible ways. If $f$ has been connected to $y$, it is also reconnected to all other vertices, what is exactly the same term as if $f$ has been connected at first to its final destination, but with the opposite sign. So, the terms cancel and $D_y\Delta_x(\Gamma)=0$. The same argument leads to $D_x\Delta_y(\Gamma)=0$, so the formula holds.

\item $\Gamma$ has both hairs on the same vertex that is $2$-valent. Because $\Gamma$ is connected it must be $\Gamma=\sigma_2=
\begin{tikzpicture}[baseline=-.5ex]
\node[int] (a) at (0,0) {};
\draw (a) edge (30:.2);
\draw (a) edge (150:.2);
\end{tikzpicture}$,
and the formula is easily checked.

\item $\Gamma$ has both hairs on the same vertex $y$ that is $3$-valent. Let $x$ be another end of the only edge at $y$, i.e.\ there is a flower on $y$ with the root $x$.
\begin{multline*}
\delta D^{(2)}(\Gamma)+D^{(2)}\delta(\Gamma)=\\
\sum_{\substack{w\\w\neq x,y}}\left(\frac{1}{2}s_wD^{(2)}(\Gamma)-a_wD^{(2)}(\Gamma)\right)
+D^{(2)}\sum_{\substack{w\\w\neq x,y}}\left(\frac{1}{2}s_w(\Gamma)-a_w(\Gamma)\right)
+D^{(2)}\left(\frac{1}{2}s_x(\Gamma)-a_x(\Gamma)\right)
+D^{(2)}f_y(\Gamma)
\end{multline*}
where $s'_y$ is the term of $s_y$ that splits $y$ to $y$ and $z$ where both hairs stay at $y$ and the edge goes to $z$. Other terms of $s_y$ cancel out with $a_y$ and $h(y)e_y$. Using similar arguments as in the proof of Proposition \ref{DelEvenDd} one easily checks that
$$
s_wD^{(2)}(\Gamma)=-D^{(2)}s_w(\Gamma),
$$
$$
\frac{1}{2}D^{(2)}s_x(\Gamma)=\sum_{\substack{w\\w\neq x,y}}a_wD^{(2)}(\Gamma),
$$
$$
D^{(2)}a_x(\Gamma)=-\sum_{\substack{w\\w\neq x,y}}D^{(2)}a_w(\Gamma),
$$
$$
2\nabla D^{(2)}s'_y(\Gamma)=D^{(1)}\Delta(\Gamma),
$$
so the formula follows.

\item $\Gamma$ has both hairs on the same vertex $x$ that is more than $3$-valent. Then by the definition $\delta D^{(2)}(\Gamma)=\delta(0)=0$ and the only term that remains from $D^{(2)}\delta(\Gamma)$ is $D^{(2)}s'_x(\Gamma)$, where $s'_x$ is the term of $s_x$ that splits $x$ to $x$ and $y$ where both hairs stay at $x$ and all edges go to $y$. It still holds $2\nabla D^{(2)}s'_y(\Gamma)=D^{(1)}\Delta(\Gamma)$, so the formula follows.
\end{enumerate}

The proof of the second claim of the proposition is the same as the first case above.
\end{proof}

Note that the previous proposition fails for the disconnected graph $\Gamma\cup\sigma_2\in\mH^2\fHGC_{-1,0}$ where $\Gamma\in\mH^0\fHGC_{-1,0}$.

\subsection{Three hairs}
\begin{prop}
\label{prop:D3}
For the connected part $\mH^3\fHGCc_{-1,0}$ it holds that
$$
\nabla D^{(2)}\Delta=0.
$$
\end{prop}
\begin{proof}
The only possibility that $D^{(2)}\Delta(\Gamma)$ is not $0$ is when $\Delta(\Gamma)$ has a flower, say at a vertex $y$ with the root $x$. If an edge from $x$ to $y$ is created by $\Delta$, $y$ was disconnected from the rest of the graph in $\Gamma$, what is not possible. So the flower already existed in $\Gamma$. We have 2 cases:
\begin{enumerate}
\item The third hair in $\Gamma$ is on the vertex $w\neq x$. Then the part of $\Delta$ that saves the flower deletes the hair on $w$ and makes an edge between $w$ and another vertex $z$ that is not $x$. If $z=x$, $D^{(2)}$ will move it again to another vertex, cancelling the term where $\Delta$ sent it to its final destination immediately.
\item The third hair in $\Gamma$ is on the vertex $x$. $\Delta$ deletes it and makes an edge between $x$ and another vertex. $D^{(2)}$ then reconnects the edge from $x$ to another vertex, so the resulting action is adding an edge in all possible ways. Then $\nabla$ adds another edge in all possible ways. Adding one edge on one place and another edge on another place cancels with adding edges in the opposite order.
\end{enumerate}
\end{proof}

Note that the previous proposition fails for the disconnected graph $\Gamma\cup\sigma_3\in\mH^3\fHGC_{-1,0}$ where $\Gamma\in\mH^0\fHGC_{-1,0}$.

\section{Hairy complex, even edges and odd hairs}
\label{s:Hevenodd}

In this section we prove the first part of Theorem \ref{thm:main}, i.e.\ that $H\left(\HGC_{-1,1},\delta+\Delta\right)=0$. The following diagram describes the way to do that. `Almost acyclic' means that there are only a few classes of cohomology that are easy to calculate.
$$
\begin{tikzpicture}[scale=1.3]
 \node (a) at (0,-.2) {};
 \node (b) at (0,-1) {$\left(\fHGC_{-1,1}\oplus\fGC_1^{\geq 1}[-3],\Delta+D^{(1)}\right)$ is almost acyclic, except for the classes without hairs};
 \node (b1) at (0,-2) {$\left(\fHGC_{-1,1}^\dagger\oplus\fGC_1^\dagger[-3],\Delta+D^{(1)}\right)$ is almost acyclic, except for the classes without hairs};
 \node (c) at (0,-3) {$\left(\mH^\flat\fHGC_{-1,1}^\dagger,\Delta+D^{(1)}\right)$ is almost acyclic};
 \node (d) at (0,-4) {$\left(\mH^\flat\fHGC_{-1,1}^\dagger,\delta'+\Delta+D^{(1)}\right)$ is almost acyclic};
 \node (e) at (0,-5) {$\left(\mH^\flat\fHGC_{-1,1}^{\dagger\ddagger},\delta'+\Delta+D^{(1)}\right)$ is acyclic};
 \node (f) at (0,-6) {$\left(\mH^\flat\fHGC^*_{-1,1},\delta'+\Delta+D^{(1)}\right)$ is acyclic};
 \node (g) at (0,-7) {$\left(\mH^{\geq 1}\fHGC^*_{-1,1},\delta+\Delta\right)$ is acyclic};
 \node (h) at (0,-8) {$\left(\mH^{\geq 1}\fHGCc_{-1,1}^*,\delta+\Delta\right)$ is acyclic};
 \node (i) at (0,-9) {$\left(\HGC_{-1,1},\delta+\Delta\right)$ is acyclic};
 \draw (a) edge[->,double] node[right] {\ref{prop:fHGCDeltaOdd}} (b);
 \draw (b) edge[->,double] node[right] {\ref{prop:fHGC+DeltaOdd}} (b1);
 \draw (b1) edge[->,double] node[right] {\ref{cor:DeltaOdd}} (c);
 \draw (c) edge[->,double] node[right] {\ref{prop:fHGC-1'}} (d);
 \draw (d) edge[->,double] node[right] {\ref{prop:bHGC-1}} (e);
 \draw (e) edge[->,double] node[right] {\ref{prop:HGCo-1}} (f);
 \draw (f) edge[->,double] node[right] {\ref{prop:HGCdis-1}} (g);
 \draw (g) edge[->,double] node[right] {\ref{prop:fHGCc1}} (h);
 \draw (h) edge[->,double] node[right] {\ref{prop:HGC1}} (i);
\end{tikzpicture}
$$

Can this be shown in a shorter way? The easiest way to show that a complex with a differential $\delta+\Delta$ is acyclic is to make the spectral sequence in which the first differential is $\Delta$, and to use the fact that the complex with the differential $\Delta$ is acyclic. But neither $\left(\mH^{\geq 1}\fHGC_{-1,1}^{\geq 1},\Delta\right)$ nor $\left(\fHGC_{-1,1}^{\geq 1},\Delta\right)$ is acyclic, there are classes with $1$ or no hairs. For technical reasons we need to disallow $\lambda$ as a connected component and change the constraint $\geq 1$ to $\dagger$. We then change the hairless part in $\left(\fHGC_{-1,1}^\dagger,\Delta\right)$ to kill classes with 1 or no hairs and make the new complex, $\left(\mH^\flat\fHGC_{-1,1}^\dagger,\Delta\right)$, almost acyclic (Corollary \ref{cor:DeltaOdd}).

The spectral sequence argument now leads to the conclusion that the complex with the differential $\delta+\Delta$ is almost acyclic (Proposition \ref{prop:fHGC-1'}). But it is not our intended result, there is a complicated hairless part. To remove it (Proposition \ref{prop:HGCo-1}), we need a change of constraint to $*$ (at least 2-valent vertices, at least 3-valent hairy vertices). The standard result that the change of constraint does not change the cohomology of the standard differential $\delta$ (Corollary \ref{cor:fHGC-d-dd-fHGC*}), can be used in the spectral sequence with the standard differential being the first one (Proposition \ref{prop:HGCo-1}). But the spectral sequence is bounded, and hence converges correctly, only if we change to the bounded complex $\mH^\flat\fHGC_{-1,1}^{\dagger\ddagger}$ before that (Proposition \ref{prop:bHGC-1}).

Recall that in the complex $\HGC_{-1,1}$ and all other complexes we are working with in this section, the degree is $d=v+1-h$.

\subsection{The differential $\Delta$}
\label{ss:Delta}

In this subsection we want to study the cohomology of $\left(\fHGC_{-1,1},\Delta\right)$. We will actually study a slightly different complex with an extra term $\mH^0\fHGC_{-1,1}[-1]$: 
\begin{equation}
\left(\fHGC_{-1,1}\oplus\mH^0\fHGC_{-1,1}[-1],\Delta+D^{(1)}\right)
\end{equation}
where $\Delta:\fHGC_{-1,1}\rightarrow\fHGC_{-1,1}$ and $D^{(1)}:\mH^1\fHGC_{-1,1}\rightarrow\mH^0\fHGC_{-1,1}[-1]$ is deleting a hairy vertex defined in \eqref{def:D1}. Recall that the degree is $d=v+1-h$, so the degree shift $[-1]$ is necessarily to make the differential of the degree $+1$. Proposition \ref{prop:D2} shows that $D^{(1)}\Delta=0$ in $\fHGC_{-1,1}$, so the differential indeed squares to $0$.

Recall \eqref{eq:HGC-GCshift}, so $\mH^0\fHGC_{-1,1}[-1]=\fGC_{1}^{\geq 1}[-3]$. For simplicity we use the latter notation.

Similarly to \eqref{eq:split4}, our new complex splits as the product of subcomplexes with fixed number of vertices $v$, with the extra term having $v-1$ vertices:
\begin{equation}
\label{eq:split41}
\left(\fHGC_{-1,1}\oplus\fGC_1^{\geq 1}[-3],\Delta+D^{(1)}\right)=\prod_{v\in\N}\left(\mV^v\fHGC_{-1,1}\oplus\mV^{v-1}\fGC_1^{\geq 1}[-3],\Delta+D^{(1)}\right).
\end{equation}
In each subcomplex, the degree $d=v+1-h$ is up to the shift equal to the negative number of hairs $-h$. We may write it as:
$$
\begin{tikzpicture}[scale=1.6]
 \node at (0,.6) {$v+1$};
 \node at (-2,.6) {$v$};
 \node at (-4,.6) {$v-1$};
 \node at (-6,.6) {$v-2$};
 \node at (-7.5,.6) {$d=$};
 \node (a0) at (0,0) {$\mH^0\mV^v\fHGC_{-1,1}$};
 \node at (0,-.4) {$\oplus$};
 \node (b0) at (0,-.8) {$\mV^{v-1}\OC$};
 \node (a1) at (-2,0) {$\mH^1\mV^v\fHGC_{-1,1}$};
 \node (a2) at (-4,0) {$\mH^2\mV^v\fHGC_{-1,1}$};
 \node (a3) at (-6,0) {$\mH^3\mV^v\fHGC_{-1,1}$};
 \node (a4) at (-7.5,0) {$\dots$};
 \draw (a4) edge[->] node[above] {$\Delta$} (a3);
 \draw (a3) edge[->] node[above] {$\Delta$} (a2);
 \draw (a2) edge[->] node[above] {$\Delta$} (a1);
 \draw (a1) edge[->] node[above] {$\Delta$} (a0);
 \draw (a1) edge[->] node[below] {$D^{(1)}$} (b0);
\end{tikzpicture}
$$

\begin{prop}
\label{prop:fHGCDeltaOdd}
\begin{itemize}\item[]
\item $H\left(\mV^1\fHGC_{-1,1},\Delta\right)$ is one-dimensional, the class being represented by $\sigma_1=
\begin{tikzpicture}[baseline=-.5ex]
\node[int] (a) at (0,0) {};
\draw (a) edge (0,.2);
\end{tikzpicture}\,$;
\item $H_d\left(\mV^v\fHGC_{-1,1}\oplus\mV^{v-1}\OC,\Delta+D^{(1)}\right)=0$ for $v>1$ and $d\leq v$.
\end{itemize}
\end{prop}
Note that the second claim of the proposition does not say anything about the cohomology at degree $d=v+1$, and for $d\geq v+2$ it is trivially $0$.
\begin{proof}
$\mV^{v-1}\OC$ is isomorphic to the subcomplex of $\mV^v\fGC_1[-2]$ spanned by graphs with an isolated vertex, the isomorphism being adding an isolated vertex $\cup\sigma_0$. Since the proposition does not say anything about the cohomology at degree $v+1$ we may safely replace $\mV^{v-1}\OC$ with the whole $\mV^v\fGC_1[-2]$. The purpose is to make $D^{(1)}$ not change the number of vertices, so the new $D^{(1)}:\mV^v\mH^1\fHGC_{-1,1}\rightarrow\mV^v\fGC_1[-2]$ reconnects all edges from the hairy vertex and deletes it, but restores the vertex without its hair.

Since the differential does not change the number of vertices, we can use Proposition \ref{prop:Hinv} and work with fixed number of vertices and distinguish them. Let $V^v:=\bar\mV^v\fHGC_{-1,1}$ and $W^v:=\bar\mV^v\fGC_1[-2]$.

The whole complex $\mV^1\fHGC_{-1,1}$ is one dimensional and generated by $\sigma_1$, so the first statement of the proposition is clear.

We will show that for $v>1$ $H_{v+1-h}\left(V^v+W^v,\Delta+D^{(1)}\right)=0$ for $h\geq 1$ by induction on $v$. That will conclude the proof of the proposition. The claim includes the claim that $H_{v+1-h}(V^v,\Delta)=0$ for $h\geq 2$. It is also true for $v=1$, and it will actually be used as the assumption of the induction. So, we may use the case $v=1$ as the base of the induction. Suppose now that the assumption is true for $v-1$ vertices, $v\geq 2$.

On $V^v$ we choose one vertex, say the last one, and set up a spectral sequence on the total valence $s$ of non-chosen vertices. So, an edge between non-chosen vertices counts twice, a hair on non-chosen vertex and an edge between non-chosen vertex and the chosen vertex counts ones, and hair on the chosen vertex does not count. The differential $\Delta$ can not decrease $s$ and splits $\Delta=\Delta_0+\Delta_1$ where $\Delta_0$ is the part that does not change $s$. $\Delta_0$ connects a hair from a non-chosen vertex to the chosen vertex and $\Delta_1$ connects something to a non-chosen vertex increasing $s$ always by $1$.

As in \eqref{eq:split4}, we can furthermore split the complex as the product of subcomplexes with fixed $a=e+h$:
\begin{equation*}
(V^v,\Delta)=\prod_{a\in\Z}\left(\mA^aV^v,\Delta\right).
\end{equation*}
For fixed degree $v+1-h$, i.e.\ fixed number of hairs $h$, $s$ can get only finitely possible values, so the spectral sequence converges correctly in each $\left(\mA^aV^v,\Delta\right)$, and therefore in the whole $(V^v,\Delta)$.

On the first page of the spectral sequence there is the cohomology $(V^v,\Delta_0)$. Let $\beta:V^v\rightarrow V^v$ be the sum over all edges at the chosen vertex of deleting that edge (as heading towards the chosen edge for the matter of sign) and putting a hair on non-chosen vertex that was connected to that edge, unless it makes the chosen vertex 0-valent, being forbidden by definition. If the chosen vertex is not hairless and 1-valent, it is clear that $\Delta_0\beta+\beta\Delta_0=C\Id$ where $C$ is the number of edges at the chosen vertex plus the number of hairs on non-chosen vertices. So $H\left(V^v,\Delta_0\right)=0$ unless the chosen vertex is isolated with a hair and there are no other hairs, or it is hairless 1-valent vertex.

By gluing a 1-vertex graph $\cup\sigma_0$ or $\cup\sigma_1$ in complexes $V^v$ with distinguishable vertices we mean adding a new vertex with the highest number, so that it becomes chosen.

Every graph of the form $\Gamma\cup\sigma_1$, $\Gamma\in \mH^0V^{v-1}$, clearly represents a cohomology class. Let us call it a class of the \emph{first type}. Graphs with a hairless $1$-valent chosen vertex would not form classes if the vertex would be allowed to be $0$-valent. Therefore, cutting that possibility implies that classes are represented by graphs of the form $\Delta_0(\Gamma\cup\sigma_0)=: c(\Gamma)$ for $\Gamma\in V^{v-1}$. Let us call them classes of the \emph{second type}. It is easily seen that $c$ is an isomorphism of degree $2$, so classes of the second type on the second page of the spectral sequence are indeed equal to classes of $H(V^{v-1}[-2])$. Classes are sketched in Figure \ref{fig:ssVOdd}.

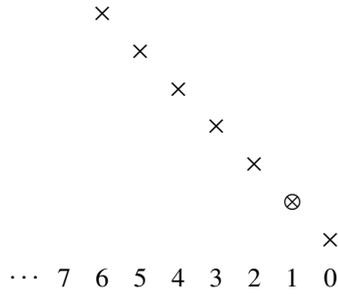
\begin{figure}[h]
\begin{tikzpicture}[scale=.5]
 \node at (0,-1) {$0$};
 \node at (-1,-1) {$1$};
 \node at (-2,-1) {$2$};
 \node at (-3,-1) {$3$};
 \node at (-4,-1) {$4$};
 \node at (-5,-1) {$5$};
 \node at (-6,-1) {$6$};
 \node at (-7,-1) {$7$};
 \node at (-8,-1) {$\dots$};
 \node at (0,0) {$\times$};
 \node at (-1,1) {$\otimes$};
 \node at (-2,2) {$\times$};
 \node at (-3,3) {$\times$};
 \node at (-4,4) {$\times$};
 \node at (-5,5) {$\times$};
 \node at (-6,6) {$\times$};
\end{tikzpicture}
\caption{\label{fig:ssVOdd}
Classes on the first page of the spectral sequence of $V^v$. The numbers at the bottom are the numbers of hairs $h$, while the degree is $d=v+1-h$. Classes of the second type are labeled by $\times$ and the position where there are both classes is labeled by $\otimes$.}
\end{figure}

Let $\CC$ be the complex isomorphic to the line of the second-type classes of the second page of the spectral sequence and it is depicted in Figure \ref{fig:ssSecondOdd}.
$\mH^hV^{v-1}$ is part of degree $d=v-h$ in $V^{v-1}$, and of degree $d=v+2-h$ in $V^{v-1}[-2]$. Here we set up the 2-row spectral sequence with the first type class $\Gamma\cup\sigma_1$, $\Gamma\in \mH^0V^{v-1}$, in the first row, and $V^{v-1}[-2]$ represented second type classes in the second row. First type class $\Gamma\cup\sigma_1$ is sent to the graph obtained from $\Gamma$ by adding an antenna in all possible ways. The isomorphic element in $\mH^1V^{v-1}$ is $\chi^1(\Gamma)$, where $\chi^1$ adds a hair in all possible ways.

\begin{figure}[h]
\begin{tikzpicture}[scale=2]
 \node at (-1,-.5) {$v+1$};
 \node at (-2,-.5) {$v$};
 \node at (-3,-.5) {$v-1$};
 \node at (-4,-.5) {$v-2$};
 \node at (-5,-.5) {$v-3$};
 \node at (-6,-.5) {$d=$};
 \node (b2) at (-2,.7) {$\mH^0 V^{v-1}$};
 \node (a1) at (-1,0) {$\mH^1 V^{v-1}$};
 \node (a2) at (-2,0) {$\mH^2 V^{v-1}$};
 \node (a3) at (-3,0) {$\mH^3 V^{v-1}$};
 \node (a4) at (-4,0) {$\mH^4 V^{v-1}$};
 \node (a5) at (-5,0) {$\mH^5 V^{v-1}$};
 \node (a6) at (-6,0) {$\dots$};
 \draw (a6) edge[->] node[below] {$\Delta$} (a5);
 \draw (a5) edge[->] node[below] {$\Delta$} (a4);
 \draw (a4) edge[->] node[below] {$\Delta$} (a3);
 \draw (a3) edge[->] node[below] {$\Delta$} (a2);
 \draw (a2) edge[->] node[below] {$\Delta$} (a1);
 \draw (b2) edge[->] node[above] {$\chi^1$} (a1);
\end{tikzpicture}
\caption{\label{fig:ssSecondOdd} The complex $\CC$ split into the another 2-row spectral sequence.}
\end{figure}

The second row is, by the induction hypothesis, acyclic for the degree $d\leq v$.
We do not know whether the degree $d=v$ first type class $\Gamma$ from the first row is canceled with something in the second row at the degree $d=v+1$. If it is, $H_d(V^v)=0$ for degree $d\leq v$.

If it is not, $H_d(V^v)=0$ for degree $d\leq v-1$ and there is a class in degree $d=v$ in $\CC$ represented by $\Gamma+\Gamma_2(\Gamma)$ for $\Gamma_2(\Gamma)\in \mH^2V^{v-1}$ ($\Gamma_2$ depends on $\Gamma$), where $\chi^1(\Gamma)+\Delta(\Gamma_2(\Gamma))=0$. Back in the starting complex $V^v$ the isomorphic element is $\Gamma\cup\sigma_1+\Delta_0(\Gamma_2(\Gamma)\cup\sigma_0)$ from the second page. It is clearly sent to $0$ by the whole $\Delta$, so it represents a class of the starting complex $V^v$, in the degree $d=v$.

Now we come back to the whole complex $\left(V^v+W^v,\Delta+D^{(1)}\right)$. Let us set up a spectral sequence of three rows: $V^v$, the part of $W^v$ where the chosen vertex is isolated, and the rest of $W^v$ (see Figure \ref{fig:ssVWOdd}).
The spectral sequence clearly converges correctly.
On the first page in the first row in degree $d=v$ ($h=1$) we may have classes represented by $\Gamma\cup\sigma_1+\Delta_0(\Gamma_2(\Gamma)\cup\sigma_0)$ for some $\Gamma\in \mH^0V^{v-1}$ and $\Gamma_2(\Gamma)\in \mH^2V^{v-1}$, as shown above. In the degree $d=v+1$ ($h=0$) of the first page we are not interested.
In the other rows in the degree $d=v+1$ ($h=0$) there is still the whole space, particularly in the second row there is $\Gamma\cup\sigma_0$ for every $\Gamma\in \mH^0V^{v-1}$.

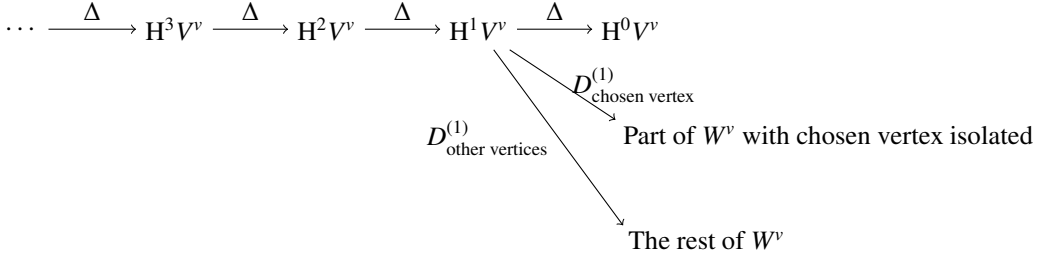
\begin{figure}[h]
\begin{tikzpicture}[scale=2]
 \node (a0) at (0,0) {$\mH^0V^{v}$};
 \node (b0) at (1.3,-.7) {Part of $W^v$ with chosen vertex isolated};
 \node (c0) at (0.5,-1.4) {The rest of $W^v$};
 \node (a1) at (-1,0) {$\mH^1V^{v}$};
 \node (a2) at (-2,0) {$\mH^2V^{v}$};
 \node (a3) at (-3,0) {$\mH^3V^{v}$};
 \node (a4) at (-4,0) {$\dots$};
 \draw (a4) edge[->] node[above] {$\Delta$} (a3);
 \draw (a3) edge[->] node[above] {$\Delta$} (a2);
 \draw (a2) edge[->] node[above] {$\Delta$} (a1);
 \draw (a1) edge[->] node[above] {$\Delta$} (a0);
 \draw (a1) edge[->] node[right] {$D^{(1)}_{\text{chosen vertex}}$} (-.1,-.6);
 \draw (a1) edge[->] node[left] {$D^{(1)}_{\text{other vertices}}$} (-.05,-1.3);
\end{tikzpicture}
\caption{\label{fig:ssVWOdd}Complex $\left(V^v+W^v,\Delta+D^{(1)}\right)$ split into 3-row spectral sequence.}
\end{figure}

On the second page $\Gamma\cup\sigma_1+\Delta_0(\Gamma_2(\Gamma)\cup\sigma_0)$ is mapped by part of $D^{(1)}$ to the part of $W^v$ where the chosen vertex is isolated, i.e.\ chosen vertex has been deleted (and restored). The only part of $\Gamma\cup\sigma_1+\Delta(\Gamma_2(\Gamma)\cup\sigma_0)$ that has a hair on the chosen vertex is $\Gamma\cup\sigma_1$, so the differential is actually $\Gamma\cup\sigma_1+\Delta(\Gamma_2(\Gamma)\cup\sigma_0)\mapsto\Gamma\cup\sigma_0$. It is clearly an injection, making the cohomology in the degree $d=v$ acyclic.

So, in both cases, if the class represented by $\Gamma\cup\sigma_1+\Delta_0(\Gamma_2(\Gamma)\cup\sigma_0)$ in degree $d=v$ of $H\left(V^v,\Delta\right)$ exists or not, in the whole complex $H_d\left(V^v+W^v,\Delta+D^{(1)}\right)=0$ for degrees $d\leq v$. That was to be demonstrated.
\end{proof}

\subsection{Removing $\lambda$}
In this subsection we transform the result to the complex $\fHGC_{-1,1}^\dagger$ that does not have $\lambda=
\begin{tikzpicture}[baseline=-.65ex]
  \node[int] (v) at (0,0) {};
  \node[int] (w) at (0.5,0) {};
  \draw (v) edge (w);
\end{tikzpicture}
$ as a connected component.

\begin{prop}
\label{prop:fHGC+DeltaOdd}
\begin{itemize}\item[]
\item $H\left(\mV^1\fHGC_{-1,1}^\dagger,\Delta\right)$ is one-dimensional, the class being represented by $\sigma_1=
\begin{tikzpicture}[baseline=-.5ex]
\node[int] (a) at (0,0) {};
\draw (a) edge (0,.2);
\end{tikzpicture}\,$;
\item $H_d\left(\mV^v\fHGC_{-1,1}^\dagger\oplus\mV^{v-1}\OCC,\Delta+D^{(1)}\right)=0$ for $v>1$ and $d\leq v$.
\end{itemize}
\end{prop}
\begin{proof}
We again do the proof on induction on $v$. For $v=1$ there can not be any $\lambda$, so $\mV^1\fHGC_{-1,1}^\dagger=\mV^1\fHGC_{-1,1}$, so the result is indeed the same as in the Proposition \ref{prop:fHGCDeltaOdd}. For $v=2$ there is $\lambda$ only in the hairless part and it represents a cohomology class, so in degrees we are considering it does not change the result of Proposition \ref{prop:fHGCDeltaOdd} either.

Let us pick $v>2$ and assume that the proposition holds for any number of vertices smaller than $v$.

On $\left(\mV^v\fHGC_{-1,1}\oplus\mV^{v-1}\OC,\Delta+D^{(1)}\right)$ we set up a spectral sequence on the number of $\lambda$-s. The differential can not increase that number, and it is bounded, so the spectral sequence converges correctly. The lowest row on the first page is our intended complex $\left(\mV^v\fHGC_{-1,1}^\dagger\oplus\mV^{v-1}\OCC,\Delta+D^{(1)}\right)$.

The first differential in the other rows does not effect any $\lambda$, so it is the same as the complex without them, but now with fewer vertices (by 2, 4, etc.), with a degree shift. All of them are acyclic by the assumption of induction in degrees that correspond to more than one hair ($d\leq v-1$). Therefore, if there is a class with a hair ($d\leq v$) in the last row, it can not be cancelled by anything, contradicting the result that the whole complex is acyclic in that degrees (Proposition \ref{prop:fHGCDeltaOdd}).
\end{proof}

To simplify the result we define another complex
\begin{equation}
\mH^\flat\mV^v\fHGC_{-1,1}^\dagger:=\mH^{\geq 1}\mV^v\fHGC_{-1,1}^\dagger\oplus\left(\Delta+D^{(1)}\right) \left(\mH^1\mV^v\fHGC_{-1,1}^\dagger\right).
\end{equation}
We have changed the term with the highest degree $\mH^0\mV^v\fHGC_{-1,1}^\dagger\oplus\mV^{v-1}\OCC$ with its subspace \linebreak
$\left(\Delta+D^{(1)}\right)\left(\mH^1\mV^v\fHGC_{-1,1}^\dagger\right)$, the image of the differential, to ensure the acyclicity at that degree. The whole complex, including all numbers of vertices, is
\begin{equation}
\label{def:fHGC1'}
\mH^\flat\fHGC_{-1,1}^\dagger:=\prod_{v\in\N}\mH^\flat\mV^v\fHGC_{-1,1}^\dagger=\mH^{\geq 1}\fHGC_{-1,1}^\dagger\oplus\HL_1
\end{equation}
where
\begin{equation}
\label{def:HL1}
\HL_1:=\left(\Delta+D^{(1)}\right)\left(\mH^1\fHGC_{-1,1}^\dagger\right)\subset\mH^0\fHGC_{-1,1}^\dagger\oplus\OCC
\end{equation}
is the hairless part.

\begin{cor}
\label{cor:DeltaOdd}
$H\left(\mH^\flat\fHGC_{-1,1}^\dagger,\Delta+D^{(1)}\right)$ is one-dimensional, the class being represented by $\sigma_1$.
\end{cor}

\subsection{The differential $\delta+\Delta$}
\label{ss:deltaDelta}

On $\fHGC_{-1,1}^\dagger$ there is the standard differential $\delta$. We extend it to $\delta':\fHGC_{-1,1}^\dagger\oplus\OC\rightarrow\fHGC_{-1,1}^\dagger\oplus\OC$ as follows
\begin{equation}
\delta'(\Gamma,\gamma)=\left(\delta(\Gamma),\mH^0\Gamma-\delta(\gamma)\right),
\end{equation}
where $\mH^h\Gamma$ is part of $\Gamma$ with $h$ hairs. It clearly squares to $0$ and has degree $1$, so it is a differential.

Differentials $\delta$ and $\Delta$ anti-commute, so $\left(\mH^{\geq 1}\fHGC_{-1,1}^\dagger,\delta+\Delta\right)$ is a complex. Proposition \ref{prop:D1} implies
\begin{multline*}
\left(\Delta+D^{(1)}\right)\delta'(\Gamma,\gamma)+\delta'\left(\Delta+D^{(1)}\right)(\Gamma,\gamma)=\left(\Delta+D^{(1)}\right)\left(\delta(\Gamma),H^0\Gamma-\delta(\gamma)\right)+\delta'\left(\Delta(\Gamma),D^{(1)}\mH^1\Gamma\right)=\\
=\left(\Delta\delta(\Gamma),D^{(1)}\mH^1\delta(\Gamma)\right)+\left(\delta\Delta(\Gamma),\mH^0\Delta(\Gamma)-\delta D^{(1)}\mH^1\Gamma\right)=(0,0),
\end{multline*}
i.e.\ $\delta'$ and $\Delta+D^{(1)}$ anti-commute and $\left(\mH^{\geq 1}\fHGC_{-1,1}^\dagger\oplus\OCC,\delta'+\Delta+D^{(1)}\right)$ is also a complex. Because of the same reason the restriction $\delta':\mH^\flat\fHGC_{-1,1}^\dagger\rightarrow\mH^\flat\fHGC_{-1,1}^\dagger$ is well defined, so $\left(\mH^\flat\fHGC_{-1,1}^\dagger,\delta'\right)$ and $\left(\mH^\flat\fHGC_{-1,1}^\dagger,\delta'+\Delta+D^{(1)}\right)$ are also complexes.

\begin{prop}
\label{prop:fHGC-1'}
$H\left(\mH^\flat\fHGC_{-1,1}^\dagger,\delta'+\Delta+D^{(1)}\right)$ is one-dimensional, the class being represented by \linebreak
$\alpha=\sum_{n\geq 1}\frac{1}{n!} \sigma_1^{\cup n}$.
\end{prop}
\begin{proof}
We set up a spectral sequence on $v$ from the splitting \eqref{eq:split41}, such that the first differential is $\Delta+D^{(1)}$. By Corollary \ref{cor:DeltaOdd} on the first page survives only $\sigma_1$.

We can split the complex $\left(\fHGC_{-1,1}^\dagger\oplus\OCC,\delta'+\Delta+D^{(1)}\right)$ as the product of subcomplexes with fixed $f=e+h-v$:
\begin{equation*}
\left(\fHGC_{-1,1}^\dagger\oplus\OCC,\delta'+\Delta+D^{(1)}\right)=\prod_{f\in\Z}\left(\mF^f\fHGC_{-1,1}^\dagger\oplus\mF^f\OCC,\delta'+\Delta+D^{(1)}\right).
\end{equation*}
The same splitting can be done for the subcomplex $\mH^\flat\fHGC_{-1,1}^\dagger$:
\begin{equation*}
\left(\mH^\flat\fHGC_{-1,1}^\dagger,\delta'+\Delta+D^{(1)}\right)=\prod_{f\in\Z}\left(\mF^f\mH^\flat\fHGC_{-1,1}^\dagger,\delta'+\Delta+D^{(1)}\right).
\end{equation*}
For fixed $v$ and degree $d=v+1-h$ the numbers of edges $e$ and hairs $h$ are bounded in each $\mF^f\mH^\flat\fHGC_{-1,1}^\dagger$. The standard spectral sequence argument (e.g.\ \cite[Proposition 19]{DGC1}) implies that the spectral sequence converges correctly in each $\left(\mF^f\mH^\flat\fHGC_{-1,1}^\dagger,\delta'+\Delta+D^{(1)}\right)$, and therefore in the whole $\left(\mH^\flat\fHGC_{-1,1}^\dagger,\delta'+\Delta+D^{(1)}\right)$.

The cohomology of $\left(\mH^\flat\fHGC_{-1,1}^\dagger,\delta'+\Delta+D^{(1)}\right)$ is therefore one-dimensional. One checks that $\alpha$ is mapped to $0$, and since $\sigma_1$ is the highest part of $\alpha$, $\alpha$ represents the class in $H\left(\mH^\flat\fHGC^\dagger,\delta'+\Delta+D^{(1)}\right)$.
\end{proof}

\subsection{Bounded complex}
Let
\begin{equation}
\label{def:bHGC1'}
\mH^\flat\fHGC_{-1,1}^{\dagger\ddagger}:=\mH^{\geq 1}\fHGC_{-1,1}^{\dagger\ddagger}\oplus\HL_1\subset\mH^\flat\fHGC_{-1,1}^\dagger,
\end{equation}
where $\HL_1$ is as in \eqref{def:HL1} and $\fHGC_{-1,1}^{\dagger\ddagger}$ is defined in Definition \ref{defi:bounded}.

\begin{prop}
\label{prop:bHGC-1}
The complex $\left(\mH^\flat\fHGC_{-1,1}^{\dagger\ddagger},\delta'+\Delta+D^{(1)}\right)$ is acyclic.
\end{prop}
\begin{proof}
We write
$$
\left(\mH^\flat\fHGC_{-1,1}^\dagger,\delta'+\Delta+D^{(1)}\right)=
\left(\fHGC_{-1,1}^{\dagger/{\dagger\ddagger}}\oplus\mH^\flat\fHGC_{-1,1}^{\dagger\ddagger},\delta'+\Delta+D^{(1)}\right).
$$
On it we set up a spectral sequence of two obvious rows: $\fHGC_{-1,1}^{\dagger/{\dagger\ddagger}}$ and $\mH^\flat\fHGC_{-1,1}^{\dagger\ddagger}$. The spectral sequence clearly converges correctly.

Proposition \ref{prop:uHGC} implies that in the first page there is a class $[\alpha]$ in the first row. The class survives all pages because of Proposition \ref{prop:fHGC-1'} that says that the whole complex has the class $[\alpha]$. Therefore the second row has to be acyclic. That was to be demonstrated.
\end{proof}

\subsection{At least 2-valent vertices}
Let
\begin{equation}
\label{def:fHGC13'}
\mH^\flat\fHGC^*_{-1,1}:=\mH^{\geq 1}\fHGC^*_{-1,1}\oplus\HL_1\subset\mH^\flat\fHGC_{-1,1}^\dagger.
\end{equation}
Recall Definition \ref{defi:bounded}, that $\fHGC^*_{-1,1}$ is the complex spanned by graphs whose vertices are at least 2-valent, and hairy vertices are at least 3-valent.

\begin{prop}
\label{prop:HGCo-1}
The inclusion $\left({\mH^\flat\fHGC^*_{-1,1}},\delta'+\Delta+D^{(1)}\right)\hookrightarrow\left(\mH^\flat\fHGC_{-1,1}^{\dagger\ddagger},\delta'+\Delta+D^{(1)}\right)$ is a quasi-\linebreak-isomorphism.
\end{prop}
\begin{proof}
We will show that the mapping cone is acyclic. On it let us set up a spectral sequence on the number of hairs $h$.

We again split complexes as the product of subcomplexes with fixed $f=e+h-v$:
\begin{equation*}
\left({\mH^\flat\fHGC^*_{-1,1}},\delta'+\Delta+D^{(1)}\right)=\prod_{f\in\Z}\left(\mF^f{\mH^\flat\fHGC^*_{-1,1}},\delta'+\Delta+D^{(1)}\right),
\end{equation*}
\begin{equation*}
\left(\mH^\flat\fHGC_{-1,1}^{\dagger\ddagger},\delta'+\Delta+D^{(1)}\right)=\prod_{f\in\Z}\left(\mF^f\mH^\flat\fHGC_{-1,1}^{\dagger\ddagger},\delta'+\Delta+D^{(1)}\right).
\end{equation*}
For fixed degree $d=v+1-h$ the number of edges $e$ is fixed in each $\mF^f{\mH^\flat\fHGC^*_{-1,1}}$ and $\mF^f\mH^\flat\fHGC_{-1,1}^{\dagger\ddagger}$. Also, increasing the number of hairs $h$ increases the number of vertices $v$ by the same amount. Since the number of edges $e$ is fixed, for $h$ big enough, there will be an isolated vertex. But that is not possible in either ${\mH^\flat\fHGC^*_{-1,1}}$ or $\mH^\flat\fHGC_{-1,1}^{\dagger\ddagger}$. So, the spectral sequence of the mapping cone of the inclusion $\left(\mF^f{\mH^\flat\fHGC^*_{-1,1}},\delta'+\Delta+D^{(1)}\right)\hookrightarrow\left(\mF^f\mH^\flat\fHGC_{-1,1}^{\dagger\ddagger},\delta'+\Delta+D^{(1)}\right)$
is bounded above and converges correctly for every $f$, and therefore also the spectral sequence of the mapping cone of the whole inclusion $\left({\mH^\flat\fHGC^*_{-1,1}},\delta'+\Delta+D^{(1)}\right)\hookrightarrow\linebreak
\left(\mH^\flat\fHGC_{-1,1}^{\dagger\ddagger},\delta'+\Delta+D^{(1)}\right)$ converges correctly. This convergence is the very reason why we introduced the bounded complex.

On the first page of the spectral sequence, for $h=0$ there is a mapping cone of the identity $\left(\HL_1,\delta'\right)\rightarrow\left(\HL_1,\delta'\right)$, so it is acyclic. For $h>0$ there is a mapping cone of the inclusion $\left(\fHGC^*_{-1,1},\delta\right)\rightarrow\left(\fHGC_{-1,1}^{\dagger\ddagger},\delta\right)$. It is acyclic by Corollary \ref{cor:fHGC-d-dd-fHGC*}. That concludes the proof.
\end{proof}

\subsection{Removing the hairless part}
\begin{prop}
\label{prop:HGCdis-1}
The complex $\left(\mH^{\geq 1}\fHGC^*_{-1,1},\delta+\Delta\right)$ is acyclic.
\end{prop}
\begin{proof}
On $\left({\mH^\flat\fHGC^*_{-1,1}},\delta'+\Delta+D^{(1)}\right)$ we set up a spectral sequence of two obvious rows: $\mH^{\geq 1}\fHGC^*_{-1,1}$ and $\HL_1$. Clearly the spectral sequence converges correctly. Propositions \ref{prop:bHGC-1} and \ref{prop:HGCo-1} imply that the whole complex is acyclic, so all classes of the first page cancel out. We claim that there are no classes on the first page.

Suppose the opposite, that there is a class in the first row represented by $\Gamma\in\mH^{\geq 1}\fHGC^*_{-1,1}$. Let $\Delta=\Delta_0+\Delta_1$ where $\Delta_0:\mH^{\geq 1}\fHGC^*_{-1,1}\rightarrow\mH^{\geq 1}\fHGC^*_{-1,1}$ and $\Delta_1:\mH^{\geq 1}\fHGC^*_{-1,1}\rightarrow\HL_1$. It holds that $(\delta+\Delta_0)(\Gamma)=0$ and $\left(\Delta_1+D^{(1)}\right)(\Gamma)$ represents a class in $\HL_1$.
We write $\Gamma=\sum_{h\geq 1}\mH^h\Gamma$. Then it holds that
$$
\Delta\left(\mH^2\Gamma\right)+\delta\left(\mH^1\Gamma\right)=0
$$
and $\left(\Delta+D^{(1)}\right)\left(\mH^1\Gamma\right)$ represents a class in $\HL_1$.

Let $\gamma:=D^{(p)}\left(\mH^1\Gamma\right)$. It is for sure in $\mH^1\fHGC_{-1,1}^\dagger$. Lemma \ref{lem:D1'} implies
$$
\Delta(\gamma)=D^{(1)}\left(\mH^1\Gamma\right),
$$
and Lemma \ref{lem:D1''} implies
$$
D^{(1)}(\gamma)=0.
$$
This relation was the very reason of introducing the constraint $*$.
Propositions \ref{prop:D1} and \ref{prop:D2} imply that
$$
\Delta\left(\mH^1\Gamma\right)=\delta D^{(1)}\left(\mH^1\Gamma\right)-D^{(1)}\delta\left(\mH^1\Gamma\right)=\delta D^{(1)}\left(\mH^1\Gamma\right)+D^{(1)}\Delta\left(\mH^2\Gamma\right)=\delta D^{(1)}\left(\mH^1\Gamma\right).
$$
The equalities are diagrammatically expressed in Figure \ref{fig:OCcan}.

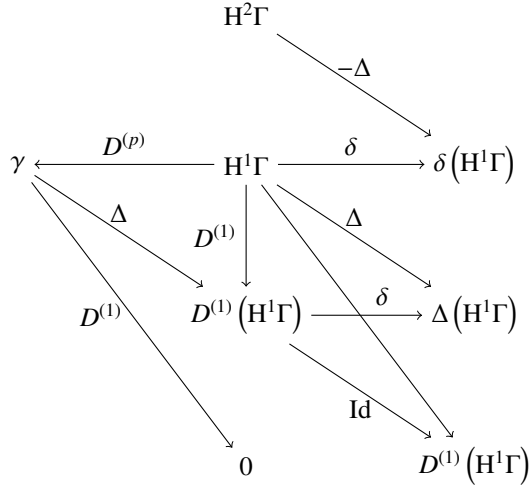
\begin{figure}[h]
\centering
\begin{tikzpicture}[scale=2]
 \node (a1) at (0,0) {$\mH^2\Gamma$};
 \node (b0) at (-1.5,-1) {$\gamma$};
 \node (b1) at (0,-1) {$\mH^1\Gamma$};
 \node (b2) at (1.5,-1) {$\delta\left(\mH^1\Gamma\right)$};
 \node (c1) at (0,-2) {$D^{(1)}\left(\mH^1\Gamma\right)$};
 \node (c2) at (1.5,-2) {$\Delta\left(\mH^1\Gamma\right)$};
 \node (d1) at (0,-3) {$0$};
 \node (d2) at (1.5,-3) {$D^{(1)}\left(\mH^1\Gamma\right)$};
 \draw (a1) edge[->] node[above] {$-\Delta$} (b2);
 \draw (b1) edge[->] node[above] {$\delta$} (b2);
 \draw (b1) edge[->] node[above] {$D^{(p)}$} (b0);
 \draw (b1) edge[->] node[left] {$D^{(1)}$} (c1);
 \draw (b1) edge[->] node[above] {$\Delta$} (c2);
 \draw (b1) edge[->] (d2);
 \draw (b0) edge[->] node[above] {$\Delta$} (c1);
 \draw (b0) edge[->] node[left] {$D^{(1)}$} (d1);
 \draw (c1) edge[->] node[above right] {$\delta$} (c2);
 \draw (c1) edge[->] node[below] {$\Id$} (d2);
\end{tikzpicture}
\caption{\label{fig:OCcan}Maps in $\mH^{\geq 1}\fHGC^*_{-1,1}\oplus\mH^0\fHGC_{-1,1}^\dagger\oplus\OCC$. The rows are, from the bottom up: $\OCC$, $\mH^0\fHGC_{-1,1}^\dagger$, $\mH^1\fHGC_{-1,1}^*$ and $\mH^2\fHGC_{-1,1}^*$.}
\end{figure}

Note that the complex in the figure is bigger than ${\mH^\flat\fHGC^*_{-1,1}}$, its hairless part is the whole $\mH^0\fHGC_{-1,1}^\dagger\oplus\OCC$ instead of only $\HL_1$. Indeed, both mentioned elements of the hairless part $\left(D^{(1)}\left(\mH^1\Gamma\right),0\right)$ and \linebreak
$\left(\delta\left(\mH^1\Gamma\right),D^{(1)}\left(\mH^1\Gamma\right)\right)$ are in $\HL_1$ because they are images of $\gamma$, respectively $\left(\mH^1\Gamma\right)$, under the action of $\left(\Delta+D^{(1)}\right)$.

Since $\left(\Delta+D^{(1)}\right)\left(\mH^1\Gamma\right)=\delta'\left(D^{(1)}\left(\mH^1\Gamma\right),0\right)$, the former is exact in $\HL_1$, contradicting the assumption.
\end{proof}

Proposition \ref{prop:fHGCc1} transforms the result to the connected complex, i.e.\ shows that $\left(\mH^{\geq 1}\fHGCc^*_{-1,1},\delta+\Delta\right)$ is acyclic. The following proposition shows the first part of Theorem \ref{thm:main}, i.e.\ that $\left(\HGC_{-1,1},\delta+\Delta\right)$ is acyclic.

\begin{prop}
\label{prop:HGC1}
The complex $\left(\HGC_{-1,1},\delta+\Delta\right)$ is acyclic.
\end{prop}
\begin{proof}
Recall that $\HGC_{-1,1}=\mH^{\geq 1}\fHGCc_{-1,1}^{\geq 3}$. On the mapping cone of the inclusion $\left(\HGC_{-1,1},\delta+\Delta\right)\to\left(\mH^{\geq 1}\fHGCc_{-1,1}^*,\delta+\Delta\right)$ we set up a spectral sequence on the number of hairs. The discussion from the proof of Proposition \ref{prop:HGCo-1} implies that the spectral sequence converges correctly.

On the first page there is the mapping cone of the inclusion $\left(\HGC_{-1,1},\delta\right)\to\left(\mH^{\geq 1}\fHGCc_{-1,1}^*,\delta\right)$. It is acyclic by Corollary \ref{cor:H>fHGCc3-H>fHGCc}, finishing the proof.
\end{proof}

\section{Hairy complex, odd edges and even hairs}
\label{s:Hoddeven}

In this section we prove the second part of Theorem \ref{thm:main}, i.e.\ that $H\left(\HGC_{-1,0},\delta+\Delta\right)$ is one-dimensional, the class being represented by the star $\sigma_3$. The proving strategy is very similar to the one in previous section.
$$
\begin{tikzpicture}[scale=1.3]
 \node (a) at (0,-.2) {};
 \node (b) at (0,-1) {$\left(\fHGC_{-1,0},\Delta\right)$ is almost acyclic, except for the classes without hairs};
 \node (c) at (0,-2) {$\left(\mH^\flat\fHGC_{-1,0},\Delta\right)$ is almost acyclic};
 \node (d) at (0,-3) {$\left(\mH^\flat\fHGC_{-1,0},\delta+\Delta\right)$ is almost acyclic};
 \node (e) at (0,-4) {$\left(\mH^\flat\fHGC_{-1,0}^{\ddagger},\delta+\Delta\right)$ is almost acyclic};
 \node (f) at (0,-5) {$\left(\mH^{\geq 1}\fHGCc_{-1,0}^{\ddagger},\delta+\Delta\right)$ is almost acyclic};
 \node (g) at (0,-6) {$\left(\HGC_{-1,0},\delta+\Delta\right)$ is almost acyclic};
 \draw (a) edge[->,double] node[right] {\ref{prop:fHGCDeltaEven}} (b);
 \draw (b) edge[->,double] node[right] {\ref{cor:Delta}} (c);
 \draw (c) edge[->,double] node[right] {\ref{prop:fHGC'Even}} (d);
 \draw (d) edge[->,double] node[right] {\ref{prop:bHGC0'}} (e);
 \draw (e) edge[->,double] node[right] {\ref{prop:bHGCc}} (f);
 \draw (f) edge[->,double] node[right] {\ref{cor:HGC0}} (g);
\end{tikzpicture}
$$

The problem this time is that $\left(\mH^{\geq 1}\fHGCc_{-1,0}^{\ddagger},\delta+\Delta\right)$ is not fully acyclic, there is a class represented by the star $\sigma_3$. This makes the complex with disconnected graphs $\left(\mH^{\geq 1}\fHGC_{-1,0}^{\ddagger},\delta+\Delta\right)$ far from acyclic. There are not only classes represented by graphs whose connected components are stars $\sigma_3$, but also any representative of a class in $H\left(\fGCc_0^{\geq 1}\right)$ with at least one connected component $\sigma_3$, to make the whole graph `hairy'. The second part of Proposition \ref{prop:bHGCc} shows that this cohomology is nicely governed by the hairless graph cohomology.

Therefore, the result that the complex with special hairless part $\left(\mH^\flat\fHGC_{-1,0}^{\ddagger},\delta+\Delta\right)$ is almost acyclic (Proposition \ref{prop:bHGC0'}) may be surprising. It seems that all disconnected classes are canceled with the hairless part. This is exactly what happens, as shown in Proposition \ref{prop:bHGCc}.

Recall that in the complex $\HGC_{-1,0}$ and all other complexes we are working with in this section, the degree is $d=e+1$.

\subsection{The differential $\Delta$}

Let $\chi^1:¸\mH^h\fHGC_{-1,0}\rightarrow\mH^{h+1}\fHGC_{-1,0}$ be the map that adds a hair in all possible ways. We also define $c:\mH^h\fHGC_{-1,0}\rightarrow\mH^{h-1}\fHGC_{-1,0}$,
\begin{equation}
c(\Gamma)=\sum_{x\in V(\gamma)}h(x)c_x(\Gamma)
\end{equation}
where $h(x)$ is the number of hairs on the vertex $x$ and $c_x(\Gamma)$ deletes one hair at $x$ and adds an antenna like $a_x$. An example is sketched in Figure \ref{fig:exc}.

\begin{figure}[h]
$$
\begin{tikzpicture}[baseline=0ex]
\node[int] (a) at (0,0) {};
\node[int] (b) at (90:1) {};
\node[int] (c) at (210:1) {};
\node[int] (d) at (330:1) {};
\draw (b) edge (.14,1.14);
\draw (b) edge (-.14,1.14);
\draw (c) edge (210:1.2);
\draw (a) edge (b);
\draw (a) edge (c);
\draw (a) edge (d);
\draw (b) edge (c);
\draw (b) edge (d);
\draw (c) edge (d);
\end{tikzpicture}
\quad\mxto{c}\quad 2\;
\begin{tikzpicture}[scale=1,baseline=3ex]
\node[int] (a) at (0,0) {};
\node[int] (b) at (90:1) {};
\node[int] (c) at (210:1) {};
\node[int] (d) at (330:1) {};
\node[int] (e) at (0,2) {};
\draw (b) edge (-.14,1.14);
\draw (c) edge (210:1.2);
\draw (b) edge (e);
\draw (a) edge (b);
\draw (a) edge (c);
\draw (a) edge (d);
\draw (b) edge (c);
\draw (b) edge (d);
\draw (c) edge (d);
\end{tikzpicture}
\quad + \quad
\begin{tikzpicture}[baseline=-1ex]
\node[int] (a) at (0,0) {};
\node[int] (b) at (90:1) {};
\node[int] (c) at (210:1) {};
\node[int] (d) at (330:1) {};
\node[int] (e) at (210:2) {};
\draw (b) edge (.14,1.14);
\draw (b) edge (-.14,1.14);
\draw (c) edge (e);
\draw (a) edge (b);
\draw (a) edge (c);
\draw (a) edge (d);
\draw (b) edge (c);
\draw (b) edge (d);
\draw (c) edge (d);
\end{tikzpicture}
$$
\caption{\label{fig:exc}
Example of the action $c$, transforming a hair into an antenna.}
\end{figure}
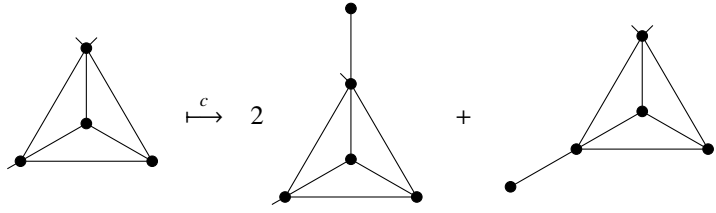

Let for $a\geq 2$
\begin{equation}
\rho_a:=\sum_{i=1}^{a-1}\frac{(-1)^i}{i!(a-1-i)!}\sigma_i\cup\lambda_{a-i}=
\sum_{i=1}^{a-1}\frac{(-1)^i}{i!(a-1-i)!}\;\;
\begin{tikzpicture}[baseline=-.5ex]
\node[int] (a) at (0,0) {};
\node at (0,.4) {$\scriptstyle i$};
\draw (a) edge (90:.2);
\draw (a) edge (360/7+90:.2);
\node[dot] at (2*360/7+90:.15) {};
\node[dot] at (2.5*360/7+90:.15) {};
\node[dot] at (3*360/7+90:.15) {};
\draw (a) edge (4*360/7+90:.2);
\draw (a) edge (5*360/7+90:.2);
\draw (a) edge (6*360/7+90:.2);
\end{tikzpicture}
\begin{tikzpicture}[baseline=-.5ex]
 \node[int] (a) at (0,0) {};
 \node[int] (b) at (.5,0) {};
 \node at (0,-.4) {$\scriptstyle a-i-1$};
 \draw (a) edge (b);
 \draw (a) edge (-360/7:.2);
 \draw (a) edge (-2*360/7:.2);
 \draw (a) edge (-3*360/7:.2);
 \draw (a) edge (-4*360/7:.2);
 \node[dot] at (-5*360/7:.15) {};
 \node[dot] at (-5.5*360/7:.15) {};
 \node[dot] at (-6*360/7:.15) {};
\end{tikzpicture}\,.
\end{equation}

As in \eqref{eq:split4} the complex $\left(\fHGC_{-1,0},\Delta\right)$ split into the double direct product of complexes, for fixed number of vertices $v$ and for fixed $a=e+h$:
\begin{equation}
\left(\fHGC_{-1,0},\Delta\right)=
\prod_{v\in\N}\prod_{a\in\Z}\left(\mA^a\mV^v\fHGC_{-1,0},\Delta\right).
\end{equation}

\begin{prop}
\label{prop:fHGCDeltaEven}
It holds that:
\begin{itemize}
\item $H\left(\mA^a\mV^1\fHGC_{-1,0},\Delta\right)$ is $1$-dimensional for $a\geq 1$, the class being represented by $\sigma_a$;
\item $H\left(\mA^a\mV^2\fHGC_{-1,0},\Delta\right)$ is acyclic for even $a\geq 2$ and $1$-dimensional for odd $a\geq 1$, the class being represented by $\lambda_a$;
\item $H\left(\mA^a\mV^3\fHGC_{-1,0},\Delta\right)$ is acyclic for odd $a\geq 1$ and $1$-dimensional for even $a\geq 2$, the class being represented by $\rho_a$;
\item $H_{a-h+1}\left(\mA^a\mV^v\fHGC_{-1,0},\Delta\right)=0$ for $v\geq 4$ and $h\geq 1$.
\end{itemize}
\end{prop}
Note that the last, general claim of the proposition does not say anything about the cohomology at degree $d=a+1$ and for $d>a+1$ it is trivially $0$.
\begin{proof}
For the matter of shortening the notation let
\begin{equation}
{W^a_v}:=\mA^a\mV^v\fHGC_{-1,0},\quad\quad{V^a_v}:=\mA^a\mV^v\mH^{\geq 1}\fHGC_{-1,0}.
\end{equation}

We prove the proposition recursively on $v$. Results for $v=1,2$ (first two claims of the proposition) are straightforward.

On ${W^a_v}$ and ${V^a_v}$ we choose one vertex and get
\begin{equation}
{\dot W^a_v}:=\mA^a\dot\mV^v\fHGC_{-1,0}, \quad {\dot V^a_v}:=\mA_a\dot\mV^v\mH^{\geq 1}\fHGC_{-1,0},
\end{equation}
the complexes with one vertex chosen, while the others are indistinguishable (see appendix \ref{app:groupinv}).

Let us set up a spectral sequence on ${\dot W^a_v}$ on the total valence $s$ of non-chosen vertices, including hairs. So, an edge between non-chosen vertices counts twice, a hair on a non-chosen vertex and an edge between a non-chosen vertex and the chosen vertex counts once, and hairs on the chosen vertex do not count. The differential can not decrease $s$ and splits $\Delta=\Delta_0+\Delta_1$ where $\Delta_0$ is the part that does not change $s$. $\Delta_0$ connects a hair from a non-chosen vertex to the chosen vertex and $\Delta_1$ connects an edge to a non-chosen vertex, increasing $s$ always by $1$.

For fixed $a=e+h$, $s$ can have only finitely many possible values, so the spectral sequence is finite and converges correctly.

On the first page of the spectral sequence there is the cohomology of $\left({\dot W^a_v},\Delta_0\right)$. Let $\chi:{\dot W^a_v}\rightarrow {\dot W^a_v}$ be the sum over all edges at the chosen vertex of deleting that edge (as a last edge in numbering, for the matter of sign) and putting a hair on the non-chosen vertex that was connected to that edge, unless it makes the chosen vertex $0$-valent, being forbidden by definition. If the chosen vertex is not hairless and $1$-valent, it is clear that $\Delta_0\chi+\chi\Delta_0=C\Id$ where $C$ is the number of edges at the chosen vertex plus the number of hairs on non-chosen vertices. So, a closed graph is exact unless the chosen vertex is isolated with some hairs and there are no other hairs, or it is hairless $1$-valent vertex.

Every graph of the form $\Gamma\dot\cup\sigma_h$, where $\Gamma\in \mH^0W_{v-1}^{a-h}$, $h\geq 1$ and $\dot\cup\sigma_h$ is the operation of gluing the star $\sigma_h$ that contains the chosen vertex, clearly represents a cohomology class. Let us call it a class of the \emph{first type}. Graphs with hairless $1$-valent chosen vertex would not form classes if the vertex would be allowed to be $0$-valent. Therefore, cutting that possibility implies that classes are represented by graphs of the form $\Delta_0(\Gamma\dot\cup\sigma_0)=:\dot c(\Gamma)$ for $\Gamma\in V_{v-1}^x$. Let us call them classes of the \emph{second type}. It is easily seen that $\dot c$ is an isomorphism of degree $1$, so classes of the second type on the second page of the spectral sequence are indeed equal to classes of $H\left(V_{v-1}^a[-1]\right)$. Classes are sketched in Figure \ref{fig:ssV}.

\begin{figure}[h]
\begin{tikzpicture}[scale=.5]
 \node at (0,-1) {$0$};
 \node at (-1,-1) {$1$};
 \node at (-2,-1) {$2$};
 \node at (-3,-1) {$3$};
 \node at (-4,-1) {$4$};
 \node at (-5,-1) {$5$};
 \node at (-6,-1) {$6$};
 \node at (-7,-1) {$7$};
 \node at (-8,-1) {$\dots$};
 \node at (0,0) {$\times$};
 \node at (-1,1) {$\otimes$};
 \node at (-2,2) {$\times$};
 \node at (-3,3) {$\times$};
 \node at (-4,4) {$\times$};
 \node at (-5,5) {$\times$};
 \node at (-6,6) {$\times$};
 \node at (-7,7) {$\times$};
 \node at (-8,8) {$\ddots$};
 \node at (-2,3) {$\odot$};
 \node at (-3,5) {$\odot$};
 \node at (-4,7) {$\odot$};
 \node at (-5,9) {$\odot$};
\end{tikzpicture}
\caption{\label{fig:ssV}Classes on the first page of the spectral sequence of ${\dot W^a_v}$. The numbers at the bottom are the number of hairs $h$, while the degree is $d=e+1=a-h+1$. Classes of the first type are labeled by $\times$ and classes of the second type are labeled by $\odot$. The position where there are both classes is labeled by $\otimes$.}
\end{figure}
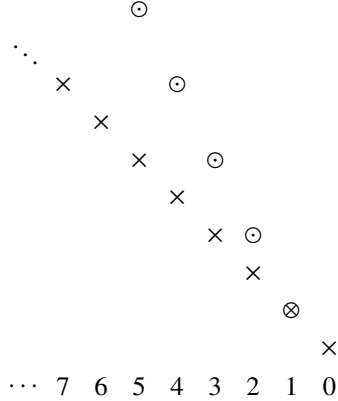

As we have already mentioned, on the second page of the spectral sequence in the line of second-type classes there is a complex isomorphic to $V_{v-1}^a[1]$, the isomorphism being $\dot c$ that transforms a hair to an antenna with the chosen vertex. But there is an ``intruder'' -- for the degree $d=a-2$ ($h=1$) there is a class of the first type $\Gamma\dot\cup\sigma_1$ in the same position with the class of the second type. All other classes of the first type clearly survive on the second page.

We continue the study first for $v=3$. It is not possible that $a<2$. For $a=2$ it is straightforward that $H\left({W^2_3}\right)$ is $1$-dimensional, the class being represented by $\rho_2=-\sigma_1\cup\lambda_1$.

If $a>2$ on the first page of the spectral sequence there is only one class of the first type, namely $\lambda_1\dot\cup\sigma_{a-1}$ of degree $2$. It does not intrude the line with the second type classes and survives until the second page. 
On the second page in the line of second-type classes we have the complex isomorphic to $V_2^a[1]$, what is equal to ${W_2^a}[1]$ in this case. So, there is a class represented by $\dot c(\lambda_a)$ of degree $3$ for odd $a$, and no class for even $a$.

Lemma $\ref{lem:V3a}$ shows that for odd $a$ the two classes cancel on further pages, so $\dot V_3^a$ is acyclic. Proposition \ref{prop:Hinv2} implies that $V_3^a$ is acyclic too.

For even $a$ the class in the degree $2$ survives. The element $\rho_a\in V_3^a$ is also an element of $\dot V_3^a$ where it is equal to the sum over all vertices to be chosen in $\rho_a$. It contains $\lambda_1\dot\cup\sigma_{a-1}$ that forms a class that survives all pages, and since $\Delta(\rho_a)=0$ (Lemma \ref{lem:DeltaRho}), $\rho_a$ represents the class in $H\left(\dot V_3^a\right)$. Proposition \ref{prop:Hinv2} implies that it is also a class in $H\left(V_3^a\right)$. That concludes the third claim of the proposition.

For $v\geq 4$ we continue to prove by induction that $H_{a-h+1}\left({\dot W^a_v}\right)=0$ for $h>0$. That will conclude the proof of the proposition. Let us do the step of the induction, while the base ($v=4$) will be explained later.

Recall that we are not interested in the cohomology $H_{a+1}\left({\dot W^a_v}\right)$. The second type class in that degree on the first page of the spectral sequence can not kill anything after the second page, so we are not interested what remains there after the second page.

Let $\CC$ be the complex isomorphic to the line of the second-type classes on the second page of the spectral sequence and it is depicted in Figure \ref{fig:ssSecond}.
We split the degree-$(a+1)$ term $\mH^1V_{v-1}^{a}[1]$ into the sum $\mathrm{Ker}_\Delta\left(\mH^1V_{v-1}^{a}\right)\oplus\mathrm{Im}_\Delta\left(\mH^1V_{v-1}^{a}\right)$ and form the 2-row spectral sequence with the intruder $\Gamma\dot\cup\sigma_1$, $\Gamma\in \mH^0V_{v-1}^{a-1}$, and $\mathrm{Im}_\Delta\left(\mH^1V_{v-1}^{a}\right)$ in the first row. The intruder $\Gamma\dot\cup\sigma_1$ is sent to the graph obtained from $\Gamma$ by adding an antenna in all possible ways. The isomorphic element in $\mH^1V_{v-1}^{a}$ is $\chi^1(\Gamma)$ where $\chi^1$ adds a hair in all possible ways. That element is further split into $\chi^1_1(\Gamma)\in\mathrm{Ker}_\Delta\left(\mH^1V_{v-1}^{a}\right)$ and $\chi^1_0(\Gamma)=\Delta\left(\chi^1(\Gamma)\right)=\nabla(\Gamma)\in\mathrm{Im}_\Delta\left(\mH^1V_{v-1}^{a}\right)\subset \mH^0V_{v-1}^{a}$.

\begin{figure}[h]
\begin{tikzpicture}[scale=2]
 \node (b1) at (-.4,.7) {$\mathrm{Im}_\Delta(\mH^1V_{v-1}^{a})\subset \mH^0V_{v-1}^{a}$};
 \node (b2) at (-2,.7) {$\mH^0V_{v-1}^{a-1}$};
 \node (a1) at (-.75,0) {$\mathrm{Ker}_\Delta(\mH^1V_{v-1}^{a})$};
 \node (a2) at (-2,0) {$\mH^2V_{v-1}^{a}$};
 \node (a3) at (-3,0) {$\mH^3V_{v-1}^{a}$};
 \node (a4) at (-4,0) {$\mH^4V_{v-1}^{a}$};
 \node (a5) at (-5,0) {$\mH^5V_{v-1}^{a}$};
 \node (a6) at (-6,0) {$\dots$};
 \draw (a6) edge[->] node[below] {$\Delta$} (a5);
 \draw (a5) edge[->] node[below] {$\Delta$} (a4);
 \draw (a4) edge[->] node[below] {$\Delta$} (a3);
 \draw (a3) edge[->] node[below] {$\Delta$} (a2);
 \draw (a2) edge[->] node[below] {$\Delta$} (a1);
 \draw (b2) edge[->] node[above] {$\nabla$} (b1);
 \draw (b2) edge[->] node[above] {$\chi^1_1$} (a1);
 \node at (-.75,-.5) {$a+1$};
 \node at (-2,-.5) {$a$};
 \node at (-3,-.5) {$a-1$};
 \node at (-4,-.5) {$a-2$};
 \node at (-5,-.5) {$a-3$};
 \node at (-6,-.5) {$d=$};
\end{tikzpicture}
\caption{\label{fig:ssSecond}Complex $\CC$ split into 2-row spectral sequence.}
\end{figure}

The second row of this spectral sequence of $\CC$ is the cut version of ${W^a_{v-1}}[1]$ without the degree-$(a+1)$ part, and its cohomology is equal to $H_{a-h+2}\left({W^a_{v-1}}[1]\right)$ for $h>0$. By the induction hypothesis (for $v>4$) it is zero, and for the base ($v=4$) there is for even $a$ a class represented by $\rho_a$ in degree $2$. That class survives the spectral sequence because even if $a=2$ it cannot be killed by the first row, since in that case in the first row there is $\mH^0V_3^{1}=0$. So, in any case $H_2(\CC)=\mathrm{Ker}_\nabla\left(\mH^0V_{v-1}^{a-1}\right)$, and the rest of the cohomology is zero except for $v=4$ there is an extra class in $H_{2}(\CC)$ being represented by $\rho_a$.

On the third page of the spectral sequence we have classes of the first type $\Gamma\dot\cup\sigma_h$ in degree $d\leq a-1$ and $\mathrm{Ker}_\nabla\left(\mH^0V_{v-1}^{a-1}\right)$ in the degree $a$. More precisely, in degree $a$ we have classes represented by $\Gamma\dot\cup\sigma_1$ for $\Gamma\in\mathrm{Ker}_\nabla\left(\mH^0V_{v-1}^{a-1}\right)$. All classes appear in even $s$. If $h>1$ the whole differential $\Delta=\Delta_0+\Delta_1$ sends the representative of the class as pictured in Figure \ref{fig:ssMaps}.
\begin{figure}[h]
\begin{tikzpicture}[baseline=-.65ex]
 \node (a1) at (0,0) {
\begin{tikzpicture}[scale=.5]
 \node (a) at (0,0) {$\Gamma$};
 \node[int] (b) at (1.5,0) {};
 \node (c) at (2.2,0.7) {$\scriptstyle h$};
 \draw (b) edge (c);
\end{tikzpicture} 
 };
 \node (a2) at (4,0) {0};
 \node (b1) at (0,-2) {
\begin{tikzpicture}[scale=.5]
 \node (a) at (0,0) {$\chi^1(\Gamma)$};
 \node[int] (b) at (1.5,0) {};
 \node (c) at (2.2,0.7) {$\scriptstyle h-1$};
 \draw (b) edge (c);
\end{tikzpicture} 
};
 \node (b2) at (4.7,-2) {
\begin{tikzpicture}[scale=.5]
 \node (a) at (0,0) {$\Gamma$};
 \node[int] (b) at (1.5,0) {};
 \node (c) at (2.2,0.7) {$\scriptstyle h-1$};
 \draw (a) edge (b);
 \draw (b) edge (c);
\end{tikzpicture} 
};
 \node (c2) at (5.2,-4) {
\begin{tikzpicture}[scale=.5]
 \node (a) at (0,0) {$\nabla(\Gamma)$};
 \node[int] (b) at (1.5,0) {};
 \node (c) at (2.2,0.7) {$\scriptstyle h-1$};
 \draw (b) edge (c);
\end{tikzpicture} + \; \begin{tikzpicture}[scale=.5]
 \node (a) at (0,0) {$\chi^1(\Gamma)$};
 \node[int] (b) at (2,0) {};
 \node (c) at (2.7,0.7) {$\scriptstyle h-2$};
 \draw (a) edge (b);
 \draw (b) edge (c);
\end{tikzpicture} 
};
 \draw (a1) edge[|->] node[above] {$\Delta_0$} (a2);
 \draw (a1) edge[|->] node[above right] {$\Delta_1$} (b2);
 \draw (b1) edge[|->] node[above] {$\Delta_0$} (b2);
 \draw (b1) edge[|->] node[above right] {$\Delta_1$} (c2);
\end{tikzpicture}.
\caption{\label{fig:ssMaps} Cancelling on the third page of the spectral sequence.}
\end{figure}

The very last term does not represent a class on the first page (or does not exist when $h=1$), so on the third page of the spectral sequence there is the differential $\Gamma\cup \sigma_h\mapsto\nabla(\Gamma)\cup \sigma_{h-1}$. Even if $h=2$ the differential is well defined because it certainly ends in $\mathrm{Ker}_\nabla\left(\mH^0V_{v-1}^{a-1}\right)$. So the complex on the third page looks like in Figure \ref{fig:ssThird}.

\begin{figure}[h]
\begin{tikzpicture}[scale=2]
 \node (a1) at (-.7,0) {$\mathrm{Ker}_\nabla\left(\mH^0V_{v-1}^{a-1}\right)$};
 \node (a2) at (-2,0) {$\mH^0V_{v-1}^{a-2}$};
 \node (a3) at (-3,0) {$\mH^0V_{v-1}^{a-3}$};
 \node (a4) at (-4,0) {$\mH^0V_{v-1}^{a-4}$};
 \node (a5) at (-5,0) {$\mH^0V_{v-1}^{a-5}$};
 \node (a6) at (-6,0) {$\dots$};
 \draw (a6) edge[->] node[above] {$\nabla$} (a5);
 \draw (a5) edge[->] node[above] {$\nabla$} (a4);
 \draw (a4) edge[->] node[above] {$\nabla$} (a3);
 \draw (a3) edge[->] node[above] {$\nabla$} (a2);
 \draw (a2) edge[->] node[above] {$\nabla$} (a1);
\end{tikzpicture}
\caption{\label{fig:ssThird}Third page of the spectral sequence on ${\dot W^a_v}$.}
\end{figure}

This complex is the cut version of $\left(\mV^{v-1}\fGC_0^{\geq 1},\nabla\right)$, and since $v-1\neq 2$ Corollary \ref{cor:nabla} implies that it is acyclic. So, $\dot V_v^a$ is acyclic, and because of Proposition \ref{prop:Hinv2} $V_v^a$ is acyclic too. This finishes the step of the induction.

If $v=4$ and $a>2$ even (for $a=2$ the class is out of interest because it is of degree $a+1$), there is one term left that survives on the third page, and therefore till the end: $\dot c(\rho_a)$. Lemma \ref{lem:DeltaRho} implies that $\Delta\left(\dot c(\rho_a)\right)=\dot c(\Delta(\rho_a))=0$, so $\dot c(\rho_a)$ represents the class. After taking invariants of the action of the symmetric group $S_v$ the class is sent to $c(\rho_a)=0$ (Lemma \ref{lem:cRhoa}), and Proposition \ref{prop:Hinv2} concludes the base of the induction.
\end{proof}

Note that all classes of $H\left(\mA^a\fHGC_{-1,0},\Delta\right)$ are in degrees $0$, $1$ or $a$. To simplify the result let us define another complex
\begin{equation}
\label{def:fHGC0'}
\mH^\flat\fHGC_{-1,0}:=\mH^{\geq 1}\fHGC_{-1,0}\oplus\HL_0
\end{equation}
where
\begin{equation}
\label{def:HL0}
\HL_0:=\Delta\left(\mH^1\fHGC_{-1,0}\right).
\end{equation}

Looking to the subcomplexes $\mA^a\fHGC_{-1,0}$, we have changed the term with the highest degree $\mH^0\mA^a\fHGC_{-1,0}$ with its subspace $\HL_0$, the image of the differential, to ensure the acyclicity at that degree, as stated in the following corollary.

\begin{cor}
\label{cor:Delta}
$$
H\left(\mH^\flat\fHGC_{-1,0},\Delta\right)=\begin{cases}
                        [\sigma_a]\text{ for }a\geq 1, \\
                        [\lambda_a]\text{ for odd }a\geq 3, \\
                        [\rho_a]\text{ for even }a\geq 2.
                        \end{cases}
$$
\end{cor}

\subsection{The differential $\delta+\Delta$}
Since $\delta$ and $\Delta$ anti-commute, $\mH^\flat\fHGC_{-1,0}$ is closed under the operation $\delta$, and the complex $\left(\mH^\flat\fHGC_{-1,0},\delta+\Delta\right)$ is a subcomplex of $\left(\fHGC_{-1,0},\delta+\Delta\right)$. We study its cohomology in the following proposition.

\begin{prop}
\label{prop:fHGC'Even}
All classes of the cohomology $H\left(\mH^\flat\fHGC_{-1,0},\delta+\Delta\right)$ consist of the graphs with at most $1$ edge, i.e.\ $H_{e+1}\left(\mH^\flat\fHGC_{-1,0},\delta+\Delta\right)=0$ for $e\geq 2$.
\end{prop}
\begin{proof}
We set up a spectral sequence on the number of vertices, such that the first differential is $\Delta$. 
By Corollary \ref{cor:Delta} the first page of spectral sequence has only zeros in relevant degrees. After splitting the complex
$$
\left(\mH^\flat\fHGC_{-1,0},\delta+\Delta\right)=\bigoplus_f\left(\mF^f\mH^\flat\fHGC_{-1,0},\delta+\Delta\right)
$$
the standard spectral sequence argument (e.g.\ \cite[Proposition 19]{DGC1}) implies that the spectral sequence converges correctly, hence the result.
\end{proof}

Using Corollary \ref{cor:Delta} and easy constructions, all classes of the cohomology $H\left(\mH^\flat\fHGC_{-1,0},\delta+\Delta\right)$ can be written down. However, classes consist of complicated sums, and since we do not need them later, we will not calculate them. Let us say that some classes are represented by $\alpha$ and $\Sigma_j$ defined in \eqref{def:Sigma}, as the consequence of Lemma \ref{lem:Sigma}.

\subsection{Bounded complex}
Let
\begin{equation}
\label{def:bHGC0'}
\mH^\flat\fHGC_{-1,0}^\ddagger:=\mH^{\geq 1}\fHGC_{-1,0}^\ddagger\oplus\HL_0\subset\mH^\flat\fHGC_{-1,0},
\end{equation}
where $\HL_0$ is as in \eqref{def:HL0} and $\fHGC_{-1,0}^\ddagger$ is defined in Definition \ref{defi:bounded}.

\begin{prop}
\label{prop:bHGC0'}
All classes of the cohomology $H\left(\mH^\flat\fHGC_{-1,0}^\ddagger,\delta+\Delta\right)$ consist of the graphs that have at most 1 edge, i.e.\ $H_{e+1}\left(\mH^\flat\fHGC_{-1,0}^\ddagger,\delta+\Delta\right)=0$ for $e\geq 2$.
\end{prop}
\begin{proof}
On $\mH^\flat\fHGC_{-1,0}$ we set up a spectral sequence of two rows: $\fHGC_{-1,0}^{/\ddagger}$ and $\mH^\flat\fHGC_{-1,0}^\ddagger$. For the degree $d\geq 3$ ($e\geq 2$) the total complex is acyclic by Proposition \ref{prop:fHGC'Even} and in the first row by Corollary \ref{cor:uHGC} there are only two classes in the degrees $1$ and $2$ that go to zero by the whole differential, so they can not cancel anything in the second row. That concludes the proof.
\end{proof}

\subsection{The morphisms $\pi_f$}
\label{ss:pif}
For $j\geq 1$ we define the following.
\begin{equation}
\label{def:Sigma}
\Sigma_j:=\sum_{\substack{k_i\geq 0\\ \sum_iik_i=j}}\prod_i\frac{(-1)^{k_i}}{k_i!((2i+1)!)^{k_i}}\bigcup_i \sigma_{2i+1}^{\cup k_i}\in\fHGC_{-1,0}^\ddagger.
\end{equation}
For example, 
$$
\Sigma_3=
\frac{-1}{6\cdot 6^3}\;
\begin{tikzpicture}[baseline=-.5ex]
\node[int] (a) at (0,0) {};
\draw (a) edge (90:.2);
\draw (a) edge (120+90:.2);
\draw (a) edge (240+90:.2);
\end{tikzpicture}\;
\begin{tikzpicture}[baseline=-.5ex]
\node[int] (a) at (0,0) {};
\draw (a) edge (90:.2);
\draw (a) edge (120+90:.2);
\draw (a) edge (240+90:.2);
\end{tikzpicture}\;
\begin{tikzpicture}[baseline=-.5ex]
\node[int] (a) at (0,0) {};
\draw (a) edge (90:.2);
\draw (a) edge (120+90:.2);
\draw (a) edge (240+90:.2);
\end{tikzpicture}
+\frac{1}{6\cdot 5!}\;
\begin{tikzpicture}[baseline=-.5ex]
\node[int] (a) at (0,0) {};
\draw (a) edge (90:.2);
\draw (a) edge (120+90:.2);
\draw (a) edge (240+90:.2);
\end{tikzpicture}\;
\begin{tikzpicture}[baseline=-.5ex]
\node[int] (a) at (0,0) {};
\draw (a) edge (90:.2);
\draw (a) edge (72+90:.2);
\draw (a) edge (144+90:.2);
\draw (a) edge (216+90:.2);
\draw (a) edge (288+90:.2);
\end{tikzpicture}
+\frac{-1}{7!}\;
\begin{tikzpicture}[baseline=-.5ex]
\node[int] (a) at (0,0) {};
\draw (a) edge (90:.2);
\draw (a) edge (360/7+90:.2);
\draw (a) edge (2*360/7+90:.2);
\draw (a) edge (3*360/7+90:.2);
\draw (a) edge (4*360/7+90:.2);
\draw (a) edge (5*360/7+90:.2);
\draw (a) edge (6*360/7+90:.2);
\end{tikzpicture}.
$$
Coefficients and signs are set up such that 
\begin{equation}
(\delta+\Delta)(\Sigma_j)=0
\end{equation}
as shown in Lemma \ref{lem:Sigma}. The coefficients are indeed not surprising, they divide a graph with its order of symmetry (exchanging same stars and hairs in a star) such that coefficients from the operations disappear.

Let $\chi^d:\mH^h\fHGC_{-1,0}\rightarrow\mH^{h+d}\fHGC_{-1,0}$, $\chi^d=\left(\chi^1\right)^d$, for $d\geq 2$. It adds $d$ hairs in all possible ways, but there is a multinomial coefficient $\binom{d}{k_1,k_2,\dots}$ before, where $k_i$ is the number of hairs added to the vertex $i$. We also set $\chi^d=0$ for $d<0$.

Recall that prefix $\mB^{<f,par}$ means all graphs with $e-v<f$ of the same parity as $f$. Let
\begin{equation}
\label{def:bHGCd}
\fHGCd_{-1,0}\subset \fHGC_{-1,0}
\end{equation}
be the subcomplex spanned by disconnected graphs.

\begin{defi}
For every $f\in\Z$ we define degree-0 map $\pi_f:\mB^{<f,par}\fGC_{0}^{\geq 1}[-1]\rightarrow\mF^f\mH^{\geq 1}\fHGCd_{-1,0}^\ddagger$. Let $\Gamma\in\mB^b\fGC_{0}^{\geq 1}$ for $b<f$ of the same parity as $f$. Then
\begin{equation}
\pi_f(\Gamma):=\sum_{i=0}^{\frac{m-b}{2}-1}\frac{1}{(2i)!} \chi^{2i}(\Gamma)\cup\Sigma_{\frac{f-b}{2}-i}-
\sum_{i=1}^{\frac{f-b}{2}-1}\frac{1}{(2i-1)!} \chi^{2i-1} D(\Gamma)\cup\Sigma_{\frac{f-b}{2}-i}.
\end{equation}
\end{defi}

Lemma \ref{lem:pi} implies that $\pi_f:\left(\mB^{<f,par}\fGC_{0}^{\geq 1}[-1],\tilde\delta\right)\rightarrow\left(\mF^f\mH^{\geq 1}\fHGCd_{-1,0}^\ddagger,\delta+\Delta\right)$ is a morphism of complexes for every $f\in\Z$. Differential $\tilde\delta$ is defined in \eqref{TildeDeltac} and is equal to $\delta+ D\nabla$ \eqref{TildeDeltac2}.

\subsection{Proof of the main result}

\begin{prop}
\label{prop:bHGCc}
\begin{itemize}\item[]
\item The cohomology $H\left(\mH^{\geq 1}\fHGCc_{-1,0}^\ddagger,\delta+\Delta\right)$ is one-dimensional, the class being represented by the star $\sigma_3$,
\item $\pi_f:\left(\mB^{<f,par}\fGC_{0}^{\geq 1}[-1],\tilde\delta\right)\rightarrow\left(\mF^f\mH^{\geq 1}\fHGCd_{-1,0}^\ddagger,\delta+\Delta\right)$ is a quasi-isomorphism in every degree $d\geq 2$ ($e\geq 1$) for every $f\in\Z$.
\end{itemize}
\end{prop}
\begin{proof}
We prove both statements simultaneously by the induction on the degree $d=e+1$.

By Lemma \ref{lem:bHGCc0} the first claim holds for degrees $d=1$ and $2$.

Let $d\geq 2$ and suppose that the first claim holds for all degrees $\leq d$. We want to prove the second claim in degree $d$.

In general, a complex $(\CC,d)$ is acyclic in the degree $i$ if and only if the \emph{cut complex} $F^i(\CC,d)$ defined as
\begin{equation}
\left(F^i(\CC,d)\right)_j:=
\left\{
\begin{array}{ll}
\CC_i&j=i,\\
d(\CC_i)&j=i+1,\\
\CC_{i-1}/\mathrm{Ker}(d)\quad&j=i-1,\\
0&\text{otherwise,}
\end{array}
\right.
\end{equation}
is acyclic. Similarly, a map of complexes $\phi:(\CC,d)\rightarrow(\DD,d)$ is a quasi-isomorphism in degree $i$ if the induced map $\phi:F^i(\CC,d)\rightarrow F^i(\DD,d)$ is a quasi-isomorphism. So, we need to prove that $\pi_f:F^e\left(\mB^{<f,par}\fGC_{0}^{\geq 1}[-1],\tilde\delta\right)\rightarrow F^e\left(\mF^f\mH^{\geq 1}\fHGCd_{-1,0}^\ddagger,\delta+\Delta\right)$ is a quasi-isomorphism. We do that by proving that the mapping cone is acyclic.

Recall that $\mP^{<x}\CC$ is the subcomplex of the graph complex $\CC$ with the number of hairy connected components $c$ smaller than $x$.
On $\mB^{<f,par}\fGC_{0}^{\geq 1}[-1]$ and $\mF^f\mH^{\geq 1}\fHGCd_{-1,0}^\ddagger$ we consider the filtrations on $x$: $\mB^{\geq 2x-f}\mB^{<f,par}\fGC_{0}^{\geq 1}[-1]$, respectively $\mP^{\leq x}\mF^f\mH^{\geq 1}\fHGCd_{-1,0}^\ddagger$.

\begin{lemma}
For every $f\in\Z$ the map $\pi_f$ respects the filtrations $\mB^{\geq 2x-f}\mB^{<f,par}\fGC_{0}^{\geq 1}[-1]$ and $\mP^{\leq x}\mF^f\mH^{\geq 1}\fHGCd_{-1,0}^\ddagger$, i.e.\ for every $x\in\Z$ $\pi_f\left(\mB^{\geq 2x-f}\mB^{<f,par}\fGC_{0}^{\geq 1}[-1]\right)\subset\mP^{\leq x}\mF^f\mH^{\geq 1}\fHGCd_{-1,0}^\ddagger$.
\end{lemma}
\begin{proof}
Let $\Gamma\in\mB^{\geq 2x-f}\mB^{<f,par}\fGC_{0}^{\geq 1}[-1]$ be a graph with $v$ vertices and $e$ edges. All those constraints mean that $b:=e-v$ is of the same parity as $f$ and $2x-f\leq b<f$.

$\Sigma_j$ has at most $j$ hairy connected components. Therefore, the first sum in $\pi_f(\Gamma)$ has the most number of hairy connected components for $n=0$ and it is $c+\frac{f-e+v}{2}\leq x$. The second sum has even less hairy connected components, hence the result.
\end{proof}

Actually, we set up a spectral sequence on the cut version of complexes. The lemma implies that those filtrations induce a filtration of the mapping cone. One checks that the filtration is bounded above, so it converges correctly. It is now enough to show that the complex with the first differential is acyclic. The complex with the first differential is the mapping cone of $\pi^0_f:F^e\left(\mB^{<f,par}\fGC_{0}^{\geq 1}[-1],\delta\right)\rightarrow F^e\left(\mF^f\mH^{\geq 1}\fHGCd_{-1,0}^\ddagger,\delta+\Delta_P\right)$ where $\Delta_P$ is the part of the differential that does not change the number of hairy connected components and $\pi_f^0:\mB^{<f,par}\fGC_{0}^{\geq 1}[-1]\rightarrow\mF^f\mH^{\geq 1}\fHGC_{-1,0}^\ddagger$ is
\begin{equation}
\pi^0_f(\Gamma)=C\,\Gamma\cup \sigma_3^{\cup \frac{f-b}{2}},
\end{equation}
where $C$ is an irrelevant coefficient. The acyclicity of this mapping cone is equivalent to $\pi^0_f:\left(\mB^{<f,par}\fGC_{0}^{\geq 1}[-1],\delta\right)\rightarrow \left(\mF^f\mH^{\geq 1}\fHGCd_{-1,0}^\ddagger,\delta+\Delta_P\right)$ being a quasi-isomorphism in degree $d$, and that are we going to show. The ``trip'' to the cut version of the complex was needed only to ``save'' the degree through the mapping cones.

Let $\Delta_{CP}$ be the part of $\Delta$ that fixes both number of hairy connected components and number of connected components. The assumption says that $H\left(\mH^{\geq 1}\fHGCc_{-1,0}^\ddagger,\delta+\Delta_{CP}\right)$ is one-dimensional up to degree $d$, the class being represented by the star $\sigma_3$. It implies that $H\left(\fHGCc_{-1,0}^\ddagger,\delta+\Delta_{CP}\right)$ is generated by the star $\sigma_3$ and classes of $H\left(\fGCc_0^{\geq 1},\delta\right)[-1]$, up to the degree $d$.

The complex $\left(\fHGC_{-1,0}^\ddagger,\delta+\Delta_{CP}\right)$ is the symmetric product of $\left(\fHGCc_{-1,0}^\ddagger,\delta+\Delta_{CP}\right)$. Since the cohomology commutes with the symmetric product, its cohomology up to degree $d$ is generated with the classes that are union of hairless classes and any number of $[\sigma_3]$. A class in $H_d\left(\mH^{\geq 1}\fHGC_{-1,0}^\ddagger,\delta+\Delta_{CP}\right)$ has to have at least one star class $[\sigma_3]$, and also a hairless component. Therefore, it also represents a class in the disconnected part $H_d\left(\mH^{\geq 1}\fHGCd_{-1,0}^\ddagger,\delta+\Delta_{CP}\right)$. The union of hairless components is itself a class in $H_d\left(\fGC_0^{\geq 1}[-1],\delta\right)$, and adding stars is exactly the map $\pi^0_f$ for some $f$, so that proves that all classes of $H_d\left(\mH^{\geq 1}\fHGCd_{-1,0}^\ddagger,\delta+\Delta_{CP}\right)$ come bijectively from classes in $H_d\left(\fGC_0^{\geq 1}[-1],\delta\right)$ and some $f\in\Z$ by the map $\pi^0_f$. That proves that $\pi^0_f:\left(\mB^{<f,par}\fGC_{0}^{\geq 1}[-1],\delta\right)\rightarrow \left(\mF^f\mH^{\geq 1}\fHGCd_{-1,0}^\ddagger,\delta+\Delta_{CP}\right)$ is a quasi-isomorphism at the degree $d$.

We now need to change $\Delta_{CP}$ with $\Delta_P$. Simply set up a spectral sequence on $\left(\mF^f\mH^{\geq 1}\fHGCd_{-1,0}^\ddagger,\delta+\Delta_P\right)$ on the number of connected components, such as the first differential is $\Delta_{CP}$. The spectral sequence obviously converges correctly. A class on the first page has a hairless connected components and some stars $\sigma_3$. The differential on further pages comes from the part $\Delta_P-\Delta_{CP}$ that connects a hair of one component to a hairless component, making a hairy component. That map clearly can not produce a star or anything in its class because it has too few edges, so there can not be cancellations in this spectral sequence and $\left(\mF^f\mH^{\geq 1}\fHGCd_{-1,0}^\ddagger,\delta+\Delta_P\right)$ and $\left(\mF^f\mH^{\geq 1}\fHGCd_{-1,0}^\ddagger,\delta+\Delta_{CP}\right)$ are quasi-isomorphic. That was to be demonstrated for the second claim of the proposition.

Now let $d\geq 3$ and suppose that the second claim holds for the degree $d-1$. We want to prove the first claim in degree $d$.

On $\left(\mH^\flat\fHGC_{-1,0}^\ddagger,\delta+\Delta\right)$ we set up a spectral sequence of three rows: $\mH^{\geq 1}\fHGCd_{-1,0}^\ddagger$, $\mH^{\geq 1}\fHGCc_{-1,0}^\ddagger$ and $\HL_0$. On the first page there are $H\left(\mH^{\geq 1}\fHGCd_{-1,0}^\ddagger,\delta+\Delta\right)$, $H\left(\mH^{\geq 1}\fHGCc_{-1,0}^\ddagger,\delta+\Delta\right)$ and $H\left(\HL_0,\delta\right)$. Proposition \ref{prop:bHGC0'} implies that $H_d\left(\mH^\flat\fHGC_{-1,0}^\ddagger,\delta+\Delta\right)=0$, so all classes of the first page in the degree $d$ cancel on further pages.

Let us take a class in the first row $H_{d-1}\left(\mH^{\geq 1}\fHGCd_{-1,0}^\ddagger,\delta+\Delta\right)$. The assumption of the induction implies that it is generated by $\pi_f(\Gamma)$ for some $f\in\Z$ and $\Gamma\in\mB^{<f,par}\fGC_{0}^{\geq 1}[-1]$.
Lemma \ref{lem:pi} implies that $(\delta+\Delta)$ maps that class to $0$ in $\mH^{\geq 1}\fHGC_{-1,0}^\ddagger$, so it may only have non-zero part in the third row $\HL_0$. So, all classes of the first row in degree $d-1$ map directly to the third row, and there are no cancellations between first two rows between degrees $d-1$ and $d$ on the second page of the spectral sequence.

\begin{lemma}
\label{prop:1cancel}
In the spectral sequence of $\left(\mH^\flat\fHGC_{-1,0}^\ddagger,\delta+\Delta\right)$ containing rows $\mH^{\geq 1}\fHGCd_{-1,0}^\ddagger$, $\mH^{\geq 1}\fHGCc_{-1,0}^\ddagger$ and $\HL_0$ classes of $H\left(\mH^{\geq 1}\fHGCc_{-1,0}^\ddagger,\delta+\Delta\right)$ and $H\left(\HL_0,\delta\right)$ from the first page do not cancel on the second page.
\end{lemma}
\begin{proof}
Suppose the opposite, i.e.\ there is $\Gamma\in\mH^{\geq 1}\fHGCc_{-1,0}^\ddagger$ and $\gamma\in\HL_0$ that represent classes in $H\left(\mH^{\geq 1}\fHGCc_{-1,0}^\ddagger,\delta+\Delta\right)$, respectively $H\left(\HL_0,\delta\right)$, and that cancel each other. Then, $(\delta+\Delta)\Gamma$ is in the class of $\gamma$. We may choose $\gamma$ such that $(\delta+\Delta)\Gamma=\gamma$.

Let $\Gamma=\sum_{i\geq 1}\mH^i\Gamma$ where $\mH^i\Gamma$ is the part with $i$ hairs. Propositions \ref{prop:D1}, \ref{prop:D2} and \ref{prop:D3} imply
$$
\delta D^{(1)}\left(\mH^1\Gamma\right)-D^{(1)}\delta\left(\mH^1\Gamma\right)=\Delta\left(\mH^1\Gamma\right),
$$
$$
\delta D^{(2)}\left(\mH^2\Gamma\right)-D^{(2)}\delta\left(\mH^2\Gamma\right)=-D^{(1)}\Delta\left(\mH^2\Gamma\right),
$$
$$
D^{(2)}\Delta\left(\mH^3\Gamma\right)=0
$$
and summing all those equalities implies
$$
\delta\left(D^{(1)}\left(\mH^1\Gamma\right)+D^{(2)}\left(\mH^2\Gamma\right)\right)=
D^{(1)}\left(\delta\left(\mH^1\Gamma\right)-\Delta\left(\mH^2\Gamma\right)\right)
+D^{(2)}\left(\delta\left(\mH^2\Gamma\right)-\Delta\left(\mH^3\Gamma\right)\right)
+\Delta\left(\mH^1\Gamma\right)=\gamma.
$$
Lemma \ref{lem:D1'} and Proposition \ref{prop:D2} imply that $D^{(1)}\left(\mH^1\Gamma\right)+D^{(2)}\left(\mH^2\Gamma\right)\in\HL_0$, so $\gamma$ is exact in $\left(\HL_0,\delta\right)$, contradicting the assumption.
\end{proof}

Therefore, a class in the middle row at degree $d$, i.e.\ in $H_d\left(\mH^{\geq 1}\fHGCc_{-1,0}^\ddagger,\delta+\Delta\right)$, can not cancel with anything. Since everything cancels, there can not be a class in $H_d\left(\mH^{\geq 1}\fHGCc_{-1,0}^\ddagger,\delta+\Delta\right)$. That was to be demonstrated.
\end{proof}

Corollary \ref{cor:HGC0} says that the inclusion $\left(\HGC_{-1,0},\delta+\Delta\right)\hookrightarrow\left(\mH^{\geq 1}\fHGCc_{-1,0}^\ddagger,\delta+\Delta\right)$ is quasi-isomorphism. Therefore Proposition \ref{prop:bHGCc} concludes the proof of Theorem \ref{thm:main}.

\appendix

\section{Group action}
\label{app:groupinv}

In this section of appendix we clarify one way of calculating cohomology of a graph complex, by doing so first with distinguishing some or all vertices.

In defining graph complexes we start with a graph with distinguishable elements (vertices, edges, hairs). The graph complex is the space of coinvariants of a finite group that permutes elements acting on the starting space (recall subsections \ref{ss:GC} and \ref{ss:HGC}). Since the groups are finite, the space of coinvariants is isomorphic to the space of invariants. For the simplicity we work with the latter space.

In this paper we are interested only in the action of the symmetric group $S_v$ that permutes vertices, taking as a starting space the space that already is the space of invariants of the action of the other groups. The same can similarly be done with the other elements of the graph.

The space of invariants of the action $\rho$ of finite group $G$ on the space $V$ is
\begin{equation}
V^G=\left\{\gamma\in V|\rho_g(\gamma)=\gamma\text{ for all }g\in G\right\}.
\end{equation}
Let $(V,d)$ be a complex and $G$ a finite group acting on $\CC$ by the action of degree $0$ $\rho_g:\CC\rightarrow \CC$ for $g\in G$. Let the action and the differential commute, i.e.\
\begin{equation}
\rho_g d(\gamma)=d\rho_g(\gamma)
\end{equation}
for every $g\in G$ and $\gamma\in\CC$.
The action of the group can be extended to cohomology of $\CC$ as $\rho_g([\gamma])=[\rho_g(\gamma)]$ for $[\gamma]\in H(\CC)$.

\begin{prop}
\label{prop:Hinv}
\begin{equation*}
H\left(\CC^G,d\right)=H(\CC,d)^G.
\end{equation*}
\end{prop}
\begin{proof}
Straightforward verification.
\end{proof}

In particular, we are interested in the graph complex $(\CC,d)$ where $d$ does not change the number of vertices, i.e.\
\begin{equation}
(\CC,d)=\prod_v\left(\mV^v\CC,d\right).
\end{equation}

In those cases differential $d$ can already be defined on the space with $v$ distinguishable vertices $\bar\mV^v\CC$, where $\mV^v\CC=(\bar\mV^v\CC)^{S_v}$ and the differential on $\mV^v\CC$ is induced from the one on $\bar\mV^v\CC$. The proposition \ref{prop:Hinv} gives us an easy tool to calculate the cohomology of the graph complex:
\begin{equation}
H(\CC,d)=\prod_v H\left(\bar\mV^v\CC,d\right)^{S_v}.
\end{equation}
Cohomology is often easier to calculate on the space $\bar\mV\CC$ with distinguishable vertices.

The more interesting use in this paper will be an intermediate step: distinguishing one vertex and indistinguishing other vertices. On $\bar\mV^v\CC$ there is the action of $S_{v-1}$ that permutes first $v-1$ vertices, while always leaving the last vertex fixed. It is the sub-action of the action of the whole $S_v$. We define
\begin{equation}
\dot\mV^v\CC:=\left(\bar\mV^v\CC\right)^{S_{v-1}}.
\end{equation}
The inclusion $i:\dot\mV^v\CC\hookrightarrow\bar\mV^v\CC$ induces the map $i:H\left(\dot\mV^v\CC\right)\rightarrow H\left(\bar\mV^v\CC\right)$. The following proposition states that it is enough to consider classes of $H\left(\dot\mV^v\CC\right)$ in finding $H\left(\mV^v\CC\right)$.

\begin{prop}
\label{prop:Hinv2}
$$
H\left(\mV^v\CC\right)=\left\{c\in i\left(H\left(\dot\mV^v\CC\right)\right)| \rho_g(c)=c\text{ for all }g\in G\right\}.
$$
\end{prop}
\begin{proof}
Straightforward verification.
\end{proof}

\section{Standard techniques}
\label{app:standard}

In this section of appendix we get some new results needed in this paper, but using techniques that have already been used in the previous papers about graph complexes. Therefore, for the reader familiar to those techniques, proofs here may be straightforward.

The following corollary extends the result of Corollary \ref{cor:H>fHGCc3-H>fHGCc} to the extra differential. The proof uses standard spectral sequence argument that sees a known differential at the first page.
\begin{cor}
\label{cor:HGC0}
The inclusions $\left(\HGC_{-1,n},\delta+\Delta\right)
\hookrightarrow\left(\mH^{\geq 1}\fHGCc^{\geq 2}_{-1,n},\delta+\Delta\right)
\hookrightarrow\left(\mH^{\geq 1}\fHGCc_{-1,n}^\ddagger,\delta+\Delta\right)$ are quasi-isomorphisms.
\end{cor}
\begin{proof}
First of all we split all complexes as in \eqref{eq:split5}:
$$
\left(\CC,\delta+\Delta\right)=\prod_{f\in\Z}\left(\mF^f\CC,\delta+\Delta\right)
$$
where $f=e+h-v$. On the mapping cone of any inclusion we set up a spectral sequence on the number $e-v$. One checks that in each degree $d=1+vn+(1-n)e-nh$ the spectral sequence is bounded, so it converges correctly. On the first page of the spectral sequence there is a mapping cone of the inclusion of complexes only with the standard differential $\delta$, so it is acyclic by Corollary \ref{cor:H>fHGCc3-H>fHGCc}. Therefore the whole  mapping cone is acyclic, and the inclusion is quasi-isomorphism.
\end{proof}

The following is the corollary of \cite[Corollary 3]{DGC1}. 
\begin{cor}
\label{cor:nabla}
$H\left(\fGC_0^{\geq 1},\nabla\right)$ is $1$-dimensional, the class being represented by $\lambda_1$.
\end{cor}
\begin{proof}
The \cite[Corollary 3]{DGC1} clearly implies that $H\left(\fGCc_0^{\geq 1},\nabla\right)$ is one-dimensional, the class being represented by $\lambda_1$.
On $\fGC_0^{\geq 1}$ we set up a spectral sequence on the number of connected components. The first differential does not change that number, so the cohomology is the symmetric product of $\lambda_1$. But there can not be more than one $\lambda_1$ because of the symmetry reasons, so the only class remaining is connected graph $\lambda_1$, concluding the proof. 
\end{proof}

\begin{lemma}
\label{lem:antennas}
The complex $\fHGC_{-1,n}^{\dagger\ddagger}$ with the differential $\Gamma\mapsto\Delta(\Gamma) + \sum_{x\in V(\Gamma)}\frac{1}{2} s_x$ is acyclic.
\end{lemma}
\begin{proof}
Let an \emph{antenna} be a maximal connected subgraph consisting of at least one 1-valent vertex and 2-valent vertices. Note that there are two kinds of linear graphs (those that do not have 3- or more-valent vertex) that are entirely considered an antenna: hairless one (that is longer than two vertices) and the one with a hair on one end. They can also be a connected component in a graph, and we call them \emph{linear antennas}.

The length of an antenna is the number of vertices in it. Let $l$ be the total length of all antennas in the graph. We set up a spectral sequence on the number $l-e$. The differential can not increase that number. To ensure the correct converging of the spectral sequence, we split the complex into the product of complexes for fixed $c=e+h-v$. For fixed degree $e$, $h-v$ is fixed too. Since from the definition of $\fHGC_{-1,n}^{\dagger\ddagger}$ there is no $\sigma_1$ in the graph as a connected component, with the fixed number of edges, the number of vertices $v$ is bounded, and so is the number of hairs $h$. Therefore, for the fixed degree $e$ and the row in the spectral sequence $l-e$ the space is finite-dimensional, and the standard spectral sequence argument (e.g.\ \cite[Proposition 19]{DGC1}) implies that the spectral sequence converges correctly.

It is easy to check that the first part of the differential is only extending an antenna. When summed all together, even-length antennas (including ``$0$-length antennas'', i.e.\ vertices that are not in an antenna) are extended by one vertex, and odd-length antennas are sent to $0$. For linear antennas it is the opposite. There is a homotopy that contracts an odd-length antenna and even-length linear antenna different than $\lambda$ (because the result would be $\sigma$ what is not allowed). There is always another antenna (at least $0$-length) in the graph, so the homotopy implies the acyclicity.
\end{proof}

The following proposition is the consequence of Proposition \ref{prop:HGCdis-1}. It transforms the result from the whole complex to the connected part.
\begin{prop}
\label{prop:fHGCc1}
The complex $\left(\mH^{\geq 1}\fHGCc_{-1,1}^*,\delta+\Delta\right)$ is acyclic.
\end{prop}
\begin{proof}
The complex splits as:
\begin{equation*}
\left(\mH^{\geq 1}\fHGCc_{-1,1}^*,\delta+\Delta\right)=\prod_{f\in\Z}\left(\mF^f\mH^{\geq 1}\fHGCc_{-1,1}^*,\delta+\Delta\right).
\end{equation*}
In each of the subcomplexes the degree $d=v+1-h=e+1-f$ is determined by the number of edges $e$. We will prove the proposition simultaneously for all subcomplexes by the induction on the number of edges $e$, i.e.\ that $H_{e+1-f}\left(\mF^f\mH^{\geq 1}\fHGCc_{-1,1}^*\delta+\Delta\right)=0$. For $e=0$ it is clear. Let us suppose that the claim holds for every number of edges less than $e$. We prove the claim for $e$ edges.

Suppose the opposite, that there is a class represented by $\Gamma$ with $e$ edges. By Proposition \ref{prop:HGCdis-1} there is $\gamma\in\mH^{\geq 1}\fHGC_{-1,1}^*$ such that $\Gamma=(\delta+\Delta)(\gamma)$. It clearly has $e-1$ edges.

Since in $\mH^{\geq 1}\fHGC_{-1,1}^*$ every connected component has at least one edge, the number of connected components is bounded.
We write $\gamma=\sum_{i=1}^k\mC^i\gamma$ where $\mC^i\gamma$ is the part with $i$ connected components. Choose $\gamma$ such that $k$ is minimal possible. If $k=1$ we are done, so suppose that $k>1$.

Let $\Delta_0$ be the part of $\Delta$ that does not connect two connected components. We now switch to the complexes that include hairless graphs and write:
$$
\left(\mC^k\fHGC^*_{-1,1},\delta+\Delta_0\right)=
\left(\left(\fHGCc^*_{-1,1},\delta+\Delta\right)^{\otimes k}\right)^{S_k}[k-1].
$$
Taking cohomology commutes with the symmetric product, so
$$
H\left(\mC^k\fHGC^*_{-1,1},\delta+\Delta_0\right)=
\left(H\left(\fHGCc^*_{-1,1},\delta+\Delta\right)^{\otimes k}\right)^{S_k}[k-1].
$$
In particular, because of the assumption of the induction, cohomology lives in the hairless part for less than $e$ edges.

It holds that $\mC^k\gamma\in\mC^k\fHGC^*_{-1,1}$ and $(\delta+\Delta_0)\mC^k\gamma=0$. Since $\mC^k\gamma$ has $e-1$ edges, there is $\gamma'\in\mC^k\fHGC^*_{-1,1}$ such that $(\delta+\Delta_0)\gamma'-\mC^k\gamma$ is hairless.

Now $\gamma-(\delta+\Delta)\gamma'$ is also mapped to $\Gamma$ by $(\delta+\Delta)$. $\Delta$ does not act on hairless part, so the part with $k$ connected components $\left(\delta+\Delta_0\right)\gamma'-\mC^k\gamma$ can be removed and the resulting element is still mapped to $\Gamma$ by $(\delta+\Delta)$. It has less than $k$ connecting components, contradicting the minimality of $k$.
\end{proof}

\section{Technical results}
\label{app:technical}

\subsection{Unbounded remainder}
The following lemma calculates the cohomology of $\left(\HC_{-1,n},\delta+\Delta\right)$ defined in Definition \ref{defi:unbounded}.
\begin{lemma}
\label{lem:HC}
$H\left(\HC_{-1,n},\delta+\Delta\right)$ is $1$-dimensional, the class being represented by $\alpha=\sum_{n\geq 1}\frac{1}{n!} \sigma_1^{\cup n}$.
\end{lemma}
\begin{proof}
On $\HC_{-1,n}$ we set up a spectral sequence on the number of vertices. The spectral sequence clearly converges correctly. The first differential is $\Delta$. In the first row there is just $\sigma_1$, and it survives on the first page. All other rows have two terms, one with and one without $\lambda_2$, that cancel on the first page. Therefore, the cohomology must be $1$-dimensional. One easily checks that $(\delta+\Delta)\alpha=0$ and since $\alpha$ has $\sigma_1$ in the first row, it represents the class.
\end{proof}

\subsection{Results for Proposition \ref{prop:fHGCDeltaEven}}
The following three lemmas are necessarily calculations used in the proof of Proposition \ref{prop:fHGCDeltaEven}.
\begin{lemma}
\label{lem:V3a}
$\left(\mA^a\dot\mV^3\fHGC_{-1,0},\Delta\right)$ is acyclic for odd $a\geq 3$.
\end{lemma}
\begin{proof}
Like in the proof of Proposition \ref{prop:fHGCDeltaEven} we set up a spectral sequence on the total valence of nonchosen vertices, including hairs, and get two classes that survive second page: $\lambda_1\dot\cup\sigma_{a-1}$ and $\dot c(\lambda_a)$. Let
\begin{equation*}
\xi:=\sum_{i=1}^{a-1}\sum_{j=0}^{a-i-1}(-1)^i\binom{a-1}{i,j,a-i-j-1}
\begin{tikzpicture}[scale=.7,baseline=1ex]
 \node[int] (a) at (0,0) {};
 \node[int] (b) at (1,0) {};
 \draw (a) edge (b);
 \node (as) at (0,-.6) {$\scriptstyle a-i-j-1$};
 \node (bs) at (1,-.6) {$\scriptstyle j$};
 \draw (a) edge (as);
 \draw (b) edge (bs);
 \node[int] (c) at (.5,.75) {};
 \node (cs) at (.5,1.35) {$\scriptstyle i$};
 \draw (c) edge (cs);
 \node at (1.2,0) {};
\end{tikzpicture}
\in\dot V_3^a
\end{equation*}
where  the upper vertex is chosen. We will work also with
\begin{equation*}
\bar\xi:=\sum_{i=0}^{a-1}\sum_{j=0}^{a-i-1}(-1)^i\binom{a-1}{i,j,a-i-j-1}
\begin{tikzpicture}[scale=.7,baseline=1ex]
 \node[int] (a) at (0,0) {};
 \node[int] (b) at (1,0) {};
 \draw (a) edge (b);
 \node (as) at (0,-.6) {$\scriptstyle a-i-j-1$};
 \node (bs) at (1,-.6) {$\scriptstyle j$};
 \draw (a) edge (as);
 \draw (b) edge (bs);
 \node[int] (c) at (.5,.75) {};
 \node (cs) at (.5,1.35) {$\scriptstyle i$};
 \draw (c) edge (cs);
 \node at (1.2,0) {};
\end{tikzpicture}
=\xi+\sum_{j=0}^{a-1}\binom{a-1}{j}
\begin{tikzpicture}[scale=.7,baseline=1ex]
 \node[int] (a) at (0,0) {};
 \node[int] (b) at (1,0) {};
 \draw (a) edge (b);
 \node (as) at (0,-.6) {$\scriptstyle a-j-1$};
 \node (bs) at (1,-.6) {$\scriptstyle j$};
 \draw (a) edge (as);
 \draw (b) edge (bs);
 \node[int] (c) at (.5,.75) {};
 \node at (1.2,0) {};
\end{tikzpicture}.
\end{equation*}
It is not an element of $\mA^a\dot\mV^3\fHGC_{-1,0}$ because it has terms with an isolated vertex. It holds that
\begin{equation*}
\begin{split}
\Delta(\bar\xi) & = \sum_{j=0}^{a-1}\sum_{i=0}^{a-j-1}(-1)^i\binom{a-1}{i,j,a-i-j-1}\Delta\left(
\begin{tikzpicture}[scale=.7,baseline=1ex]
 \node[int] (a) at (0,0) {};
 \node[int] (b) at (1,0) {};
 \draw (a) edge (b);
 \node (as) at (0,-.6) {$\scriptstyle a-i-j-1$};
 \node (bs) at (1,-.6) {$\scriptstyle j$};
 \draw (a) edge (as);
 \draw (b) edge (bs);
 \node[int] (c) at (.5,.75) {};
 \node (cs) at (.5,1.35) {$\scriptstyle i$};
 \draw (c) edge (cs);
 \node at (1.3,0) {};
\end{tikzpicture}
\right) \\
& = 2\sum_{j=0}^{a-1}\sum_{i=1}^{a-j-1}(-1)^i\frac{(a-1)!}{(i-1)!j!(a-i-j-1)!}
\begin{tikzpicture}[scale=.7,baseline=1ex]
 \node[int] (a) at (0,0) {};
 \node[int] (b) at (1,0) {};
 \draw (a) edge (b);
 \node (as) at (0,-.6) {$\scriptstyle a-i-j-1$};
 \node (bs) at (1,-.6) {$\scriptstyle j$};
 \draw (a) edge (as);
 \draw (b) edge (bs);
 \node[int] (c) at (.5,.75) {};
 \draw (a) edge (c);
 \node (cs) at (.5,1.35) {$\scriptstyle i-1$};
 \draw (c) edge (cs);
 \node at (1.3,0) {};
\end{tikzpicture} \\
& \quad + 2\sum_{j=0}^{a-1}\sum_{i=0}^{a-j-2}(-1)^i\frac{(a-1)!}{i!j!(a-i-j-2)!}
\begin{tikzpicture}[scale=.7,baseline=1ex]
 \node[int] (a) at (0,0) {};
 \node[int] (b) at (1,0) {};
 \draw (a) edge (b);
 \node (as) at (0,-.6) {$\scriptstyle a-i-j-2$};
 \node (bs) at (1,-.6) {$\scriptstyle j$};
 \draw (a) edge (as);
 \draw (b) edge (bs);
 \node[int] (c) at (.5,.75) {};
 \draw (a) edge (c);
 \node (cs) at (.5,1.35) {$\scriptstyle i$};
 \draw (c) edge (cs);
 \node at (1.3,0) {};
\end{tikzpicture} = 0,
\end{split}
\end{equation*}
\begin{equation*}
\Delta(\xi)=\Delta(\bar\xi)-\Delta\left(
\sum_{j=0}^{a-1}\binom{a-1}{j}
\begin{tikzpicture}[scale=.7,baseline=1ex]
 \node[int] (a) at (0,0) {};
 \node[int] (b) at (1,0) {};
 \draw (a) edge (b);
 \node (as) at (0,-.6) {$\scriptstyle a-j-1$};
 \node (bs) at (1,-.6) {$\scriptstyle j$};
 \draw (a) edge (as);
 \draw (b) edge (bs);
 \node[int] (c) at (.5,.75) {};
 \node at (1.3,0) {};
\end{tikzpicture}
\right)=-\dot c\left(\sum_{j=0}^{a-1}\binom{a-1}{j}
\begin{tikzpicture}[scale=.7,baseline=-1ex]
 \node[int] (a) at (0,0) {};
 \node[int] (b) at (1,0) {};
 \draw (a) edge (b);
 \node (as) at (0,-.6) {$\scriptstyle a-j-1$};
 \node (bs) at (1,-.6) {$\scriptstyle j$};
 \draw (a) edge (as);
 \draw (b) edge (bs);
 \node at (1.2,0) {};
\end{tikzpicture}
\right).
\end{equation*}
Let
\begin{equation*}
\nu:=\sum_{j=1}^{a-1}\binom{a}{j}
\begin{tikzpicture}[scale=.7,baseline=-1ex]
 \node[int] (a) at (0,0) {};
 \node[int] (b) at (1,0) {};
 \node (as) at (0,-.6) {$\scriptstyle a-j$};
 \node (bs) at (1,-.6) {$\scriptstyle j$};
 \draw (a) edge (as);
 \draw (b) edge (bs);
 \node at (1.2,0) {};
\end{tikzpicture},
\end{equation*}
\begin{equation*}
\bar\nu:=\sum_{j=0}^{a}\binom{a}{j}
\begin{tikzpicture}[scale=.7,baseline=-1ex]
 \node[int] (a) at (0,0) {};
 \node[int] (b) at (1,0) {};
 \node (as) at (0,-.6) {$\scriptstyle a-j$};
 \node (bs) at (1,-.6) {$\scriptstyle j$};
 \draw (a) edge (as);
 \draw (b) edge (bs);
 \node at (1.2,0) {};
\end{tikzpicture}=\nu+2
\begin{tikzpicture}[scale=.7,baseline=-1ex]
 \node[int] (a) at (0,0) {};
 \node[int] (b) at (1,0) {};
 \node (as) at (0,-.6) {$\scriptstyle a$};
 \draw (a) edge (as);
 \node at (-.2,0) {};
 \node at (1.2,0) {};
\end{tikzpicture}.
\end{equation*}
It holds that
\begin{equation*}
\Delta(\bar\nu)=\sum_{j=0}^{a}\frac{a!}{j!(a-j)!}\left((a-j)
\begin{tikzpicture}[scale=.7,baseline=-1ex]
 \node[int] (a) at (0,0) {};
 \node[int] (b) at (1,0) {};
 \draw (a) edge (b);
 \node (as) at (0,-.6) {$\scriptstyle a-j-1$};
 \node (bs) at (1,-.6) {$\scriptstyle j$};
 \draw (a) edge (as);
 \draw (b) edge (bs);
 \node at (1.2,0) {};
\end{tikzpicture}+j
\begin{tikzpicture}[scale=.7,baseline=-1ex]
 \node[int] (a) at (0,0) {};
 \node[int] (b) at (1,0) {};
 \draw (a) edge (b);
 \node (as) at (0,-.6) {$\scriptstyle a-j$};
 \node (bs) at (1,-.6) {$\scriptstyle j-1$};
 \draw (a) edge (as);
 \draw (b) edge (bs);
 \node at (1.2,0) {};
\end{tikzpicture}
\right)=
\sum_{j=0}^{a-1}2a\binom{a-1}{j}
\begin{tikzpicture}[scale=.7,baseline=-1ex]
 \node[int] (a) at (0,0) {};
 \node[int] (b) at (1,0) {};
 \draw (a) edge (b);
 \node (as) at (0,-.6) {$\scriptstyle a-j-1$};
 \node (bs) at (1,-.6) {$\scriptstyle j$};
 \draw (a) edge (as);
 \draw (b) edge (bs);
 \node at (1.2,0) {};
\end{tikzpicture},
\end{equation*}
\begin{equation*}
\Delta(\nu)=\Delta(\bar\nu)-2a\lambda_a,
\end{equation*}
\begin{equation*}
\Delta(2a\xi+\dot c(\nu))=2a\Delta(\xi)+\dot c(\Delta(\nu))=-2a\dot c(\lambda_a).
\end{equation*}
Element $2a\xi+\dot c(\nu)$ contains $\lambda_1\dot\cup\sigma_{a-1}$ that survives the second page of the spectral sequence, and the target is exactly multiple of $\lambda_a$ that also survives, so the two cancel each other. That was to be demonstrated.
\end{proof}

\begin{lemma}
\label{lem:DeltaRho}
For even $a\geq 2$ it holds that
$$
\Delta(\rho_a)=0.
$$
\end{lemma}
\begin{proof}
\begin{equation*}
\begin{split}
\Delta(\rho_a) & = \sum_{i=1}^{a-1}\frac{(-1)^i}{i!(a-i-1)!}\Delta\left(
\begin{tikzpicture}[scale=.7,baseline=1ex]
 \node[int] (a) at (0,0) {};
 \node[int] (b) at (1,0) {};
 \draw (a) edge (b);
 \node (as) at (0,-.6) {$\scriptstyle a-i-1$};
 \draw (a) edge (as);
 \node[int] (c) at (.5,.75) {};
 \node (cs) at (.5,1.35) {$\scriptstyle i$};
 \draw (c) edge (cs);
 \node at (1.3,0) {};
\end{tikzpicture}
\right) \\
& = \sum_{i=1}^{a-1}\frac{(-1)^i}{(i-1)!(a-i-1)!}
\begin{tikzpicture}[scale=.7,baseline=1ex]
 \node[int] (a) at (0,0) {};
 \node[int] (b) at (1,0) {};
 \draw (a) edge (b);
 \node (as) at (0,-.6) {$\scriptstyle a-i-1$};
 \draw (a) edge (as);
 \node[int] (c) at (.5,.75) {};
 \draw (a) edge (c);
 \node (cs) at (.5,1.35) {$\scriptstyle i-1$};
 \draw (c) edge (cs);
 \node at (1.3,0) {};
\end{tikzpicture} \\
& \quad + \sum_{i=1}^{a-1}\frac{(-1)^i}{(i-1)!(a-i-1)!}
\begin{tikzpicture}[scale=.7,baseline=1ex]
 \node[int] (a) at (0,0) {};
 \node[int] (b) at (1,0) {};
 \draw (a) edge (b);
 \node (as) at (0,-.6) {$\scriptstyle a-i-1$};
 \draw (a) edge (as);
 \node[int] (c) at (.5,.75) {};
 \draw (b) edge (c);
 \node (cs) at (.5,1.35) {$\scriptstyle i-1$};
 \draw (c) edge (cs);
 \node at (1.3,0) {};
\end{tikzpicture} \\
& \quad + \sum_{i=1}^{a-2}\frac{(-1)^i}{i!(a-i-2)!}
\begin{tikzpicture}[scale=.7,baseline=1ex]
 \node[int] (a) at (0,0) {};
 \node[int] (b) at (1,0) {};
 \draw (a) edge (b);
 \node (as) at (0,-.6) {$\scriptstyle a-i-2$};
 \draw (a) edge (as);
 \node[int] (c) at (.5,.75) {};
 \draw (a) edge (c);
 \node (cs) at (.5,1.35) {$\scriptstyle i$};
 \draw (c) edge (cs);
 \node at (1.3,0) {};
\end{tikzpicture}
\end{split}
\end{equation*}
The first and the third term cancel each other, while the second term is easily seen to be $0$ for even $a$, concluding the proof.
\end{proof}

\begin{lemma}
\label{lem:cRhoa}
For even $a\geq 2$ it holds that
$$
c(\rho_a)=0.
$$
\end{lemma}
\begin{proof}
Because of the symmetry $c(\lambda_a)=0$ for every $a\geq 1$. Therefore
$$
c(\rho_a)=
\sum_{i=1}^{a-1}\frac{(-1)^i}{i!(a-1-i)!}c(\sigma_i)\cup\lambda_{a-i}=
\sum_{i=1}^{a-1}\frac{(-1)^i}{(i-1)!(a-1-i)!}\lambda_i\cup\lambda_{a-i}=0
$$
\end{proof}

\subsection{The morphisms $\pi_f$}
The following three lemmas are needed to define the morphisms $\pi_f$ in Subsection \ref{ss:pif}. Recall \eqref{def:Sigma}:
\begin{equation*}
\Sigma_m=\sum_{\substack{k_i\geq 0\\ \sum_iik_i=m}}\prod_i\frac{(-1)^{k_i}}{k_i!((2i+1)!)^{k_i}}\bigcup_i \sigma_{2i+1}^{\cup k_i}.
\end{equation*}

\begin{lemma}
\label{lem:Sigma}
For every $m\geq 1$ it holds that
$$
(\delta+\Delta)(\Sigma_m)=0.
$$
\end{lemma}
\begin{proof}
In all sums $i$, $i'$ and $j$ are $\geq 1$.
\begin{equation*}
\begin{split}
\delta(\Sigma_m)
&=\sum_{\substack{k_i\geq 0\\ \sum_iik_i=m}}\prod_i\frac{(-1)^{k_i}}{k_i!((2i+1)!)^{k_i}}\left(\sum_ik_i\delta\left(\sigma_{2i+1}\right)\cup\sigma_{2i+1}^{\cup k_i-1}\cup\bigcup_{j\neq i} \sigma_{2j+1}^{\cup k_j}\right)=\\
&=\sum_{\substack{k_i\geq 0\\ \sum_iik_i=m}}\sum_{\substack{i\\ k_i>0}}\frac{(-1)^{k_i}}{(k_i-1)!((2i+1)!)^{k_i}}\prod_{j\neq i}\frac{(-1)^{k_j}}{k_j!((2j+1)!)^{k_j}}
\sum_{h=1}^{i-1}\binom{2i+1}{2h}
\begin{tikzpicture}[scale=.7,baseline=-1ex]
 \node[int] (a) at (0,0) {};
 \node[int] (b) at (1,0) {};
 \draw (a) edge (b);
 \node (as) at (0,-.6) {$\scriptstyle 2h$};
 \node (bs) at (1,-.6) {$\scriptstyle 2i-2h+1$};
 \draw (a) edge (as);
 \draw (b) edge (bs);
 \node at (1.2,0) {};
\end{tikzpicture}
\cup\sigma_{2i+1}^{\cup k_i-1}\cup\bigcup_{j\neq i} \sigma_{2j+1}^{\cup k_j}=\\
&=-\sum_{\substack{k_i\geq 0\\ \sum_iik_i=m}}\sum_{\substack{i\\ k_i>0}}\sum_{h=1}^{i-1}\left(\frac{1}{(2h)!(2i-2h+1)!}
\begin{tikzpicture}[scale=.7,baseline=-1ex]
 \node[int] (a) at (0,0) {};
 \node[int] (b) at (1,0) {};
 \draw (a) edge (b);
 \node (as) at (0,-.6) {$\scriptstyle 2h$};
 \node (bs) at (1,-.6) {$\scriptstyle 2i-2h+1$};
 \draw (a) edge (as);
 \draw (b) edge (bs);
 \node at (1.2,0) {};
\end{tikzpicture}\right)\\
&\qquad\qquad\qquad\qquad
\cup\left(\frac{1}{(k_i-1)!}\left(\frac{-1}{(2i+1)!}\sigma_{2i+1}\right)^{\cup k_i-1}\right)\cup\bigcup_{j\neq i} \left(\frac{1}{k_j!}\left(\frac{-1}{(2j+1)!}\sigma_{2j+1}\right)^{\cup k_j}\right)=\\
&=-\sum_{i,i'}\sum_{\substack{k_j\geq 0\\ i+i'+\sum_jjk_j=m}}\left(\frac{1}{(2i)!(2i'+1)!}
\begin{tikzpicture}[scale=.6,baseline=-1ex]
 \node[int] (a) at (0,0) {};
 \node[int] (b) at (1,0) {};
 \draw (a) edge (b);
 \node (as) at (0,-.6) {$\scriptstyle 2i$};
 \node (bs) at (1,-.6) {$\scriptstyle 2i'+1$};
 \draw (a) edge (as);
 \draw (b) edge (bs);
 \node at (1.2,0) {};
\end{tikzpicture}\right)
\cup\bigcup_{j} \left(\frac{1}{k_j!}\left(\frac{-1}{(2j+1)!}\sigma_{2j+1}\right)^{\cup k_j}\right),
\end{split}
\end{equation*}
\begin{equation*}
\begin{split}
\Delta(\Sigma_m)
&=\sum_{\substack{k_i\geq 0\\ \sum_iik_i=m}}\prod_i\frac{(-1)^{k_i}}{k_i!((2i+1)!)^{k_i}}
\left(\sum_{i}k_i(k_i-1)(2i+1)
\begin{tikzpicture}[scale=.6,baseline=-1ex]
 \node[int] (a) at (0,0) {};
 \node[int] (b) at (1,0) {};
 \draw (a) edge (b);
 \node (as) at (0,-.6) {$\scriptstyle 2i$};
 \node (bs) at (1,-.6) {$\scriptstyle 2i+1$};
 \draw (a) edge (as);
 \draw (b) edge (bs);
 \node at (1.2,0) {};
\end{tikzpicture}
\cup \sigma_{2i+1}^{\cup k_i-2} \cup \bigcup_{j\neq i} \sigma_{2j+1}^{\cup k_j}+\right.\\
&\qquad\qquad\qquad\qquad\qquad\qquad+\left.\sum_{\substack{i,i'\\ i\neq i'}}k_ik_{i'}(2i+1)
\begin{tikzpicture}[scale=.6,baseline=-1ex]
 \node[int] (a) at (0,0) {};
 \node[int] (b) at (1,0) {};
 \draw (a) edge (b);
 \node (as) at (0,-.6) {$\scriptstyle 2i$};
 \node (bs) at (1,-.6) {$\scriptstyle 2i'+1$};
 \draw (a) edge (as);
 \draw (b) edge (bs);
 \node at (1.2,0) {};
\end{tikzpicture}
\cup \sigma_{2i+1}^{\cup k_i-1} \cup \sigma_{2i'+1}^{\cup k_{i'}-1} \cup \bigcup_{j\neq i,i'} \sigma_{2j+1}^{\cup k_j}\right)=\\
&=\sum_{\substack{k_i\geq 0\\ \sum_iik_i=m}}\sum_{\substack{i\\ k_i\geq 2}}
\left(\frac{1}{(2i)!(2i+1)!}
\begin{tikzpicture}[scale=.6,baseline=-1ex]
 \node[int] (a) at (0,0) {};
 \node[int] (b) at (1,0) {};
 \draw (a) edge (b);
 \node (as) at (0,-.6) {$\scriptstyle 2i$};
 \node (bs) at (1,-.6) {$\scriptstyle 2i+1$};
 \draw (a) edge (as);
 \draw (b) edge (bs);
 \node at (1.2,0) {};
\end{tikzpicture}\right)
\cup\left(\frac{1}{(k_i-2)!}\left(\frac{-1}{(2i+1)!}\sigma_{2i+1}\right)^{\cup k_i-2}\right)\cup\bigcup_{j\neq i} \left(\frac{1}{k_j!}\left(\frac{-1}{(2j+1)!}\sigma_{2j+1}\right)^{\cup k_j}\right)+\\
&\qquad+\sum_{\substack{k_i\geq 0\\ \sum_iik_i=m}}\sum_{\substack{i,i'\\ i\neq i';k_i,k_{i'}\geq 1}}
\left(\frac{1}{(2i)!(2i'+1)!}
\begin{tikzpicture}[scale=.6,baseline=-1ex]
 \node[int] (a) at (0,0) {};
 \node[int] (b) at (1,0) {};
 \draw (a) edge (b);
 \node (as) at (0,-.6) {$\scriptstyle 2i$};
 \node (bs) at (1,-.6) {$\scriptstyle 2i'+1$};
 \draw (a) edge (as);
 \draw (b) edge (bs);
 \node at (1.2,0) {};
\end{tikzpicture}\right)
\cup\left(\frac{1}{(k_i-1)!}\left(\frac{-1}{(2i+1)!}\sigma_{2i+1}\right)^{\cup k_i-1}\right)\\
&\qquad\qquad\qquad\qquad\cup\left(\frac{1}{(k_{i'}-1)!}\left(\frac{-1}{(2i'+1)!}\sigma_{2i'+1}\right)^{\cup k_{i'}-1}\right)\cup\bigcup_{j\neq i} \left(\frac{1}{k_j!}\left(\frac{-1}{(2j+1)!}\sigma_{2j+1}\right)^{\cup k_j}\right)=\\
&=\sum_i\sum_{\substack{k_j\geq 0\\ 2i+\sum_jjk_j=m}}
\left(\frac{1}{(2i)!(2i'+1)!}
\begin{tikzpicture}[scale=.6,baseline=-1ex]
 \node[int] (a) at (0,0) {};
 \node[int] (b) at (1,0) {};
 \draw (a) edge (b);
 \node (as) at (0,-.6) {$\scriptstyle 2i$};
 \node (bs) at (1,-.6) {$\scriptstyle 2i+1$};
 \draw (a) edge (as);
 \draw (b) edge (bs);
 \node at (1.2,0) {};
\end{tikzpicture}\right)
\cup\bigcup_{j} \left(\frac{1}{k_j!}\left(\frac{-1}{(2j+1)!}\sigma_{2j+1}\right)^{\cup k_j}\right)\\
&\qquad+\sum_{\substack{i,i'\\ i\neq i'}}\sum_{\substack{k_j\geq 0\\ i+i'+\sum_jjk_j=m}}
\left(\frac{1}{(2i)!(2i'+1)!}
\begin{tikzpicture}[scale=.6,baseline=-1ex]
 \node[int] (a) at (0,0) {};
 \node[int] (b) at (1,0) {};
 \draw (a) edge (b);
 \node (as) at (0,-.6) {$\scriptstyle 2i$};
 \node (bs) at (1,-.6) {$\scriptstyle 2i'+1$};
 \draw (a) edge (as);
 \draw (b) edge (bs);
 \node at (1.2,0) {};
\end{tikzpicture}\right)\cup\bigcup_{j} \left(\frac{1}{k_j!}\left(\frac{-1}{(2j+1)!}\sigma_{2j+1}\right)^{\cup k_j}\right)=-\delta(\Sigma_m).
\end{split}
\end{equation*}
\end{proof}

\begin{lemma}
\label{lem:GammaSigma}
For every $\Gamma\in\fGC_0^{\geq 1}$ and $m\geq 0$ it holds that
$$
\sum_{n=0}^{m-1}\frac{1}{(2n)!}\left(\Delta \left(\chi^{2n}(\Gamma)\cup\Sigma_{m-n}\right)-\Delta \chi^{2n}(\Gamma)\cup\Sigma_{m-n}-\chi^{2n}(\Gamma)\cup\Delta(\Sigma_{m-n})\right)=-
\sum_{n=1}^m\frac{1}{(2n)!}\left(\delta \chi^{2n}(\Gamma)-\chi^{2n}\delta(\Gamma)\right)\cup\Sigma_{m-n}.
$$
\begin{multline*}
\sum_{n=1}^{m-1}\frac{1}{(2n-1)!}\left(\Delta \left(\chi^{2n-1}(\Gamma)\cup\Sigma_{m-n}\right)-\Delta \chi^{2n-1}(\Gamma)\cup\Sigma_{m-n}-\chi^{2n-1}(\Gamma)\cup\Delta(\Sigma_{m-n})\right)=\\
=-
\sum_{n=1}^m\frac{1}{(2n-1)!}\left(\delta \chi^{2n-1}(\Gamma)-\chi^{2n-1}\delta(\Gamma)\right)\cup\Sigma_{m-n}.
\end{multline*}
\end{lemma}
\begin{proof}
Let us prove only the first equation, while the second equation is similar. We make notations such that the claim is $\sum L_n=-\sum R_n$.
$L_n$ is exactly the part of $\Delta \left(\chi^{2n}(\Gamma)\cup\Sigma_{m-n}\right)$ that connects $\chi^{2n}(\Gamma)$ with $\Sigma_{m-n}$. It can be done in two ways, a hair from $\chi^{2n}(\Gamma)$ connects to a star in $\Sigma_{m-n}$, or a star from $\Sigma_{m-n}$ connects to $\chi^{2n}(\Gamma)$. In both cases ``a flower'' with 2 or more hairs is added to $\chi^{2n}(\Gamma)$.
Let $f_k$ be a map that adds a flower with $k$ hairs to a vertex in all possible ways.
In all sums $i$ and $j$ are $\geq 1$. It holds that:
\begin{equation*}
\begin{split}
L_n&=\frac{1}{(2n)!}\sum_{\substack{k_j\geq 0\\ \sum_jjk_j=m-n}}\prod_j\frac{(-1)^{k_j}}{k_j!((2j+1)!)^{k_j}}\sum_ik_i
\left(2nf_{2i+1}\chi^{2n-1}(\Gamma)+(2i+1)f_{2i}\chi^{2n}(\Gamma)\right)\cup\sigma_{2i+1}^{\cup k_i-1}\cup\bigcup_{j\neq i} \sigma_{2j+1}^{\cup k_j}=\\
&=-\sum_i\sum_{\substack{k_j\geq 0\\ i+\sum_jjk_j=m-n}}\left(\frac{1}{(2i+1)!(2n-1)!}f_{2i+1}\chi^{2n-1}(\Gamma)+\frac{1}{(2i)!(2n)!}f_{2i}\chi^{2n}(\Gamma)\right)
\cup\bigcup_{j} \left(\frac{1}{k_j!}\left(\frac{-1}{(2j+1)!}\sigma_{2j+1}\right)^{\cup k_j}\right)=\\
&=-\sum_i\left(\frac{1}{(2i+1)!(2n-1)!}f_{2i+1}\chi^{2n-1}(\Gamma)+\frac{1}{(2i)!(2n)!}f_{2i}\chi^{2n}(\Gamma)\right)
\cup\Sigma_{m-n-i}.
\end{split}
\end{equation*}

On the other side there is
\begin{equation*}
R_n=\frac{\Sigma_{m-n}}{(2n)!}\cup\sum_x\left(\frac{1}{2}s_x(\chi^{2n}(\Gamma))-a_x(\chi^{2n}(\Gamma))-h(x)e_x(\chi^{2n}(\Gamma))-\frac{1}{2}\chi^{2n}(s_x(\Gamma))+\chi^{2n}(a_x(\Gamma))\right),
\end{equation*}
where $x$ runs through vertices of $\Gamma$. It is clear that $s_x(\chi^{2n}(\Gamma))=\chi^{2n}(s_x(\Gamma))$. In $\chi^{2n}(a_x(\Gamma))$ hairs are added to vertices of $\Gamma$ or to the new vertex of the antenna. Terms $-a_x(\chi^{2n}(\Gamma))$ and $-h(x)e_x(\chi^{2n}(\Gamma))$ cancel exactly the terms of $\chi^{2n}(a_x(\Gamma))$ where no or one hair is added to the new vertex. So
$$
R_n=\frac{\Sigma_{m-n}}{(2n)!}\cup\sum_{i=2}^{2n}\binom{2n}{i} f_i(\chi^{2n-i}(\Gamma))=
\sum_{i=2}^{2n}\frac{1}{i!(2n-i)!}f_i(\chi^{2n-i}(\Gamma))\cup\Sigma_{m-n}.
$$

Now a simple play with the sums leads to the result.
\end{proof}

\begin{lemma}
\label{lem:pi}
For every $\Gamma\in\mB^b\fGC_{0}^{\geq 1}$ and $m\in\Z$ of the same parity as $b$ it holds that
$$
(\delta+\Delta)\pi_m(\Gamma)=\pi_m\tilde\delta(\Gamma)
$$
in the complex $\mH^{\geq 1}\fHGC_{-1,0}^\ddagger$.
\end{lemma}
\begin{proof}
The left-hand side is
\begin{equation*}
\begin{split}
(\delta+\Delta)\pi_m(\Gamma)&=(\delta+\Delta)\left(\sum_{n=0}^{\frac{m-b}{2}-1}\frac{1}{(2n)!} \chi^{2n}(\Gamma)\cup\Sigma_{\frac{m-b}{2}-n}-
\sum_{n=1}^{\frac{m-b}{2}-1}\frac{1}{(2n-1)!} \chi^{2n-1} D(\Gamma)\cup\Sigma_{\frac{m-b}{2}-n}\right)=\\
&=\sum_{n=0}^{\frac{m-b}{2}-1}\left(\frac{\delta\left(\chi^{2n}(\Gamma)\right)}{(2n)!}\cup\Sigma_{\frac{m-b}{2}-n}+
\frac{\chi^{2n}(\Gamma)}{(2n)!}\cup\delta\Sigma_{\frac{m-b}{2}-n}\right)\\
&\qquad-
\sum_{n=1}^{\frac{m-b}{2}-1}\left(\frac{\delta\left(\chi^{2n-1} D(\Gamma)\right)}{(2n-1)!} \cup\Sigma_{\frac{m-b}{2}-n}+
\frac{\chi^{2n-1} D(\Gamma)}{(2n-1)!} \cup\delta\Sigma_{\frac{m-b}{2}-n}\right)\\
&\qquad+
\sum_{n=0}^{\frac{m-b}{2}-1}\frac{\Delta\left(\chi^{2n}(\Gamma)\cup\Sigma_{\frac{m-b}{2}-n}\right)}{(2n)!} -
\sum_{n=1}^{\frac{m-b}{2}-1}\frac{\Delta\left(\chi^{2n-1} D(\Gamma)\cup\Sigma_{\frac{m-b}{2}-n}\right)}{(2n-1)!}.
\end{split}
\end{equation*}
Using Lemmas \ref{lem:Sigma} and \ref{lem:GammaSigma}, Propositions \ref{DelEvenD4} to \ref{DelEvenD2nab} and a clear fact that $\Delta \chi^n(\Gamma)=n\chi^{n-1}\nabla(\Gamma)$ it follows that:
\begin{equation*}
\begin{split}
(\delta+\Delta)\pi_m(\Gamma)&=
\sum_{n=0}^{\frac{m-b}{2}-1}\left(\frac{\chi^{2n}\delta(\Gamma)}{(2n)!}\cup\Sigma_{\frac{m-b}{2}-n}+
\frac{\chi^{2n}(\Gamma)}{(2n)!}\cup\delta\Sigma_{\frac{m-b}{2}-n}\right)\\
&\qquad-
\sum_{n=1}^{\frac{m-b}{2}-1}\left(\frac{\chi^{2n-1}\delta D(\Gamma)}{(2n-1)!} \cup\Sigma_{\frac{m-b}{2}-n}+
\frac{\chi^{2n-1} D(\Gamma)}{(2n-1)!} \cup\delta\Sigma_{\frac{m-b}{2}-n}\right)\\
&\qquad+
\sum_{n=0}^{\frac{m-b}{2}-1}\frac{\Delta \chi^{2n}(\Gamma)\cup\Sigma_{\frac{m-b}{2}-n}+
\chi^{2n}(\Gamma)\cup\Delta\Sigma_{\frac{m-b}{2}-n}}{(2n)!}\\
&\qquad-
\sum_{n=1}^{\frac{m-b}{2}-1}\frac{\Delta \chi^{2n-1} D(\Gamma)\cup\Sigma_{\frac{m-b}{2}-n}+\chi^{2n-1} D(\Gamma)\cup\Delta\Sigma_{\frac{m-b}{2}-n}}{(2n-1)!}=\\
&=
\sum_{n=0}^{\frac{m-b}{2}-1}\left(\frac{\chi^{2n}\delta(\Gamma)}{(2n)!}\cup\Sigma_{\frac{m-b}{2}-n}\right)-
\sum_{n=1}^{\frac{m-b}{2}-1}\left(\frac{\chi^{2n-1} D\delta(\Gamma)}{(2n-1)!} \cup\Sigma_{\frac{m-b}{2}-n}\right)+
\sum_{n=1}^{\frac{m-b}{2}-1}\frac{\chi^{2n-2} D\nabla(\Gamma)}{(2n-2)!}\cup\Sigma_{\frac{m-b}{2}-n}.
\end{split}
\end{equation*}

The right-hand side is
$$
\pi_m\tilde\delta(\Gamma)=\pi_m(\delta(\Gamma)+ D\nabla(\Gamma)).
$$
The first term under $\pi_m$ is in $\mB^b\fGC_{0}^{\geq 1}$, and the second one is in $\mB^{b+2}\fGC_{0}^{\geq 1}$, so
\begin{multline*}
\pi_m\tilde\delta(\Gamma)=
\sum_{n=0}^{\frac{m-b}{2}-1}\frac{1}{(2n)!} \chi^{2n}\delta(\Gamma)\cup\Sigma_{\frac{m-b}{2}-n}-
\sum_{n=1}^{\frac{m-b}{2}-1}\frac{1}{(2n-1)!} \chi^{2n-1} D\delta(\Gamma)\cup\Sigma_{\frac{m-b}{2}-n}\\
+\sum_{n=0}^{\frac{m-b}{2}-2}\frac{1}{(2n)!} \chi^{2n} D\nabla(\Gamma)\cup\Sigma_{\frac{m-b}{2}-n-1}-
\sum_{n=1}^{\frac{m-b}{2}-2}\frac{1}{(2n-1)!} \chi^{2n-1} D D\nabla(\Gamma)\cup\Sigma_{\frac{m-b}{2}-n-1}=(\delta+\Delta)\pi_m(\Gamma).
\end{multline*}
\end{proof}

\subsection{Base of the induction in Proposition \ref{prop:bHGCc}}
The following lemma is the base of the induction in the proof of Proposition \ref{prop:bHGCc}.
\begin{lemma}
\label{lem:bHGCc0}
\begin{itemize}\item[]
\item $H_1\left(\mH^{\geq 1}\fHGCc_{-1,0}^\ddagger,\delta+\Delta\right)$ is one dimensional, the class being generated by $\sigma_3=
\begin{tikzpicture}[baseline=-.5ex]
\node[int] (a) at (0,0) {};
\draw (a) edge (90:.2);
\draw (a) edge (120+90:.2);
\draw (a) edge (240+90:.2);
\end{tikzpicture}$.
\item $H_2\left(\mH^{\geq 1}\fHGCc_{-1,0}^\ddagger,\delta+\Delta\right)=0$.
\end{itemize}
\end{lemma}
\begin{proof}
Relevant graphs have at most 1 edge. Connected graphs with at most $1$ edge are either a star, or a graph with $2$ vertices and an edge between them. In both cases, second differential $\Delta$ does not do anything. So we need to show the lemma for the complex $\left(\mH^{\geq 1}\fHGCc_{-1,0}^\ddagger,\delta\right)=\prod_{h\geq 1}\left(\mH^h\fHGCc_{-1,0}^\ddagger,\delta\right)$. For $h\in\{1,2,3\}$ the claim can be easily checked by hand. For $h>3$ the action of the differential $\delta$ on the degrees $1$ and $2$ in the complex with $h$ hairs is depicted in Figure \ref{fig:2}.

\begin{figure}[h]
\begin{tikzpicture}
\node (a1) at (1,0) {
\begin{tikzpicture}[scale=.7,baseline=4ex]
 \node[int] (a) at (0,0) {};
 \node[int] (b) at (1,0) {};
 \draw (a) edge (b);
 \node (as) at (0,-.6) {$\scriptstyle h$};
 \draw (a) edge (as);
 \node at (1.2,0) {};
\end{tikzpicture}
};
\node (a2) at (3.5,0) {
\begin{tikzpicture}[scale=.7,baseline=4ex]
 \node[int] (a) at (0,0) {};
 \node[int] (b) at (1,0) {};
 \node[int] (c) at (2,0) {};
 \draw (a) edge (b);
 \draw (c) edge (b);
 \node (as) at (0,-.6) {$\scriptstyle h$};
 \draw (a) edge (as);
 \node at (1.2,0) {};
\end{tikzpicture}
};
\draw (a1) edge[->] (a2);
\node (b1) at (1,-1) {
\begin{tikzpicture}[scale=.7,baseline=4ex]
 \node[int] (a) at (0,0) {};
 \node[int] (b) at (1,0) {};
 \draw (a) edge (b);
 \node (as) at (0,-.6) {$\scriptstyle h-1$};
 \draw (a) edge (as);
 \draw (b) edge (1,-.3);
 \node at (1.2,0) {};
\end{tikzpicture}
};
\node (b2) at (3.5,-1) {
\begin{tikzpicture}[scale=.7,baseline=4ex]
 \node[int] (a) at (0,0) {};
 \node[int] (b) at (1,0) {};
 \node[int] (c) at (2,0) {};
 \draw (a) edge (b);
 \draw (c) edge (b);
 \node (as) at (0,-.6) {$\scriptstyle h-1$};
 \draw (a) edge (as);
 \draw (c) edge (2,-.3);
 \node at (1.2,0) {};
\end{tikzpicture}
};
\draw (b1) edge[->] (b2);
\node (c0) at (-.8,-2) {
\begin{tikzpicture}[scale=.7,baseline=4ex]
 \node[int] (a) at (0,0) {};
 \node (as) at (0,-.6) {$\scriptstyle h$};
 \draw (a) edge (as);
 \node at (.2,0) {};
\end{tikzpicture}
};
\node (c1) at (1,-2) {
\begin{tikzpicture}[scale=.7,baseline=4ex]
 \node[int] (a) at (0,0) {};
 \node[int] (b) at (1,0) {};
 \draw (a) edge (b);
 \node (as) at (0,-.6) {$\scriptstyle h-2$};
 \node (bs) at (1,-.6) {$\scriptstyle 2$};
 \draw (a) edge (as);
 \draw (b) edge (bs);
 \node at (1.2,0) {};
\end{tikzpicture}
};
\draw (c0) edge[->] (c1);
\node (d1) at (1,-3) {
\begin{tikzpicture}[scale=.7,baseline=4ex]
 \node[int] (a) at (0,0) {};
 \node[int] (b) at (1,0) {};
 \draw (a) edge (b);
 \node (as) at (0,-.6) {$\scriptstyle h-3$};
 \node (bs) at (1,-.6) {$\scriptstyle 3$};
 \draw (a) edge (as);
 \draw (b) edge (bs);
 \node at (1.2,0) {};
\end{tikzpicture}
};
\node (d2) at (3.5,-3) {
\begin{tikzpicture}[scale=.7,baseline=4ex]
 \node[int] (a) at (0,0) {};
 \node[int] (b) at (1,0) {};
 \node[int] (c) at (2,0) {};
 \draw (a) edge (b);
 \draw (c) edge (b);
 \node (as) at (0,-.6) {$\scriptstyle h-3$};
 \node (cs) at (2,-.6) {$\scriptstyle 2$};
 \draw (a) edge (as);
 \draw (b) edge (1,-.3);
 \draw (c) edge (cs);
 \node at (1.2,0) {};
\end{tikzpicture}
};
\draw (d1) edge[->] (d2);
\node (e1) at (1,-4) {
\begin{tikzpicture}[scale=.7,baseline=4ex]
 \node[int] (a) at (0,0) {};
 \node[int] (b) at (1,0) {};
 \draw (a) edge (b);
 \node (as) at (0,-.6) {$\scriptstyle h-4$};
 \node (bs) at (1,-.6) {$\scriptstyle 4$};
 \draw (a) edge (as);
 \draw (b) edge (bs);
 \node at (1.2,0) {};
\end{tikzpicture}
};
\node (e2) at (3.5,-4) {
\begin{tikzpicture}[scale=.7,baseline=4ex]
 \node[int] (a) at (0,0) {};
 \node[int] (b) at (1,0) {};
 \node[int] (c) at (2,0) {};
 \draw (a) edge (b);
 \draw (c) edge (b);
 \node (as) at (0,-.6) {$\scriptstyle h-4$};
 \node (cs) at (2,-.6) {$\scriptstyle 3$};
 \draw (a) edge (as);
 \draw (b) edge (1,-.3);
 \draw (c) edge (cs);
 \node at (1.2,0) {};
\end{tikzpicture}
};
\draw (e1) edge[->] (e2);
\node at (1,-5) {$\vdots$};
\node at (3.5,-5) {$\vdots$};
\end{tikzpicture}
\caption{\label{fig:2}
The action of the differential $\delta$ on the degrees $1$ and $2$ in the complex $\left(\mH^h\fHGCc_{-1,0}^\ddagger,\delta\right)$. The differential does not map from the lower row to the higher one. The list goes until the left vertex in degree $d=2$ ($e=1$) has more or equal hairs as the right vertex.}
\end{figure}
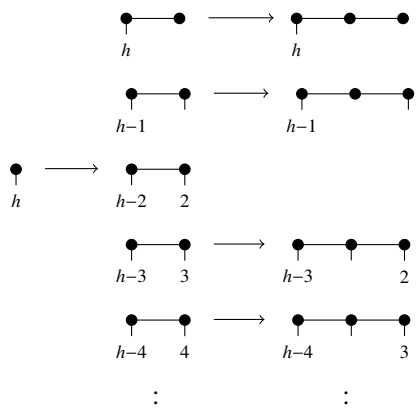

One can make a spectral sequence with the rows in the figure, such that the first differential is the one depicted by arrows. This concludes the proof.
\end{proof}

\end{document}